\documentclass[a4paper]{article}
\usepackage[margin=1.3in]{geometry}

\usepackage{amsmath}
\usepackage{amsthm}

\usepackage{times}
\usepackage{bm}
\usepackage{natbib}
\usepackage{amsfonts}
\usepackage{tikz}
\usetikzlibrary{arrows.meta,arrows}
\usetikzlibrary{shapes,decorations,arrows,calc,arrows.meta,fit,positioning}
\usetikzlibrary{arrows.meta, automata, positioning, quotes}
\usepackage[caption=false]{subfig}
\usepackage{authblk}
\usepackage[plain,noend]{algorithm2e}

\makeatletter
\newcommand{\leqnos}{\tagsleft@true\let\veqno\@@leqno}
\newcommand{\reqnos}{\tagsleft@false\let\veqno\@@eqno}
\reqnos
\makeatother

\theoremstyle{plain}
\newtheorem{theorem}{Theorem}[section]
\newtheorem{proposition}[theorem]{Proposition}
\newtheorem{lemma}[theorem]{Lemma}
\newtheorem{example}[theorem]{Example}
\newtheorem{corollary}[theorem]{Corollary}
\theoremstyle{definition}
\newtheorem{definition}[theorem]{Definition}

\theoremstyle{remark}
\newtheorem{remark}[theorem]{Remark}

\DeclareMathOperator{\Forbb}{forb}
\DeclareMathOperator{\pa}{pa}
\DeclareMathOperator{\de}{de}
\DeclareMathOperator{\an}{an}
\DeclareMathOperator{\CN}{cn}

\newcommand{\g}[1][G]{\mathcal{#1}}
\newcommand{\f}[2][X,Y]{\Forbb(#1,#2)}
\newcommand{\cn}[2][X,Y]{\CN(#1,#2)}
\newcommand{\vsp}{\vspace{.2cm}}
\newcommand\ci{\perp\!\!\!\perp}
\newcommand\notci{\not\!\perp\!\!\!\perp}

\usepackage{etoolbox}
\AtBeginEnvironment{algorithm}{\let\textbf\relax}

\usepackage{xr-hyper}
\usepackage{hyperref}

\begin{document}

	\title{Graphical tools for selecting conditional instrumental sets}

\author[1]{Leonard Henckel}
\author[2]{Martin Buttenschön}
\author[3]{Marloes H. Maathuis}
\affil[1]{School of Mathematics and Statistics, University College Dublin}
\affil[2]{Department of Statistics, University of Oxford}
\affil[3]{Seminar for Statistics, ETH Zurich}
	\date{} 
	\maketitle

	\begin{abstract}
	We consider the efficient estimation of total causal effects in the presence of unmeasured confounding using conditional instrumental sets. Specifically, we consider the two-stage least squares estimator in the setting of a linear structural equation model with correlated errors that is compatible with a known acyclic directed mixed graph. 
	To set the stage for our results, we characterize the class of linearly valid conditional instrumental sets that yield consistent two-stage least squares estimators for the target total effect and derive a new asymptotic variance formula for these estimators. Equipped with these results, we provide three graphical tools for selecting more efficient linearly valid conditional instrumental sets. First, a graphical criterion that for certain pairs of linearly valid conditional instrumental sets identifies which of the two corresponding estimators has the smaller asymptotic variance. Second, an algorithm that greedily adds covariates that reduce the asymptotic variance to a given linearly valid conditional instrumental set. Third, a linearly valid conditional instrumental set for which the corresponding estimator has the smallest asymptotic variance that can be ensured with a graphical criterion.
	\end{abstract}

	
\section{Introduction}
		\label{introIV}
  
Suppose we want to estimate the total causal effect of an exposure $X$ on an outcome $Y$ in the presence of latent confounding. This is often done with instrumental variable estimators. A popular instrumental variable estimator is the two-stage least squares estimator which has attractive statistical properties for linear models \citep{vansteelandt2018improving}, such as being asymptotically efficient among a broad class of estimators for Gaussian linear structural equation models \citep[Chapter 5,][]{wooldridge2010econometric}.
The two-stage least squares estimator is computed based on a tuple $(Z,W)$ of covariate sets, where we call $Z$ the instrumental set, $W$ the conditioning set and $(Z,W)$ the conditional instrumental set. It is well-known that the two-stage least squares estimator is only consistent if we use certain tuples $(Z,W)$ and we call such tuples \emph{valid} conditional instrumental sets relative to $(X,Y)$. Determining which conditional instrumental sets are valid has received considerable research attention  \citep[e.g.][]{bowden1990instrumental,angrist1996identification,hernan2006instruments,wooldridge2010econometric}. In the graphical framework and for linear structural equation models, the setting we adopt in this paper, the most relevant results are by \citet{brito2002new} and \citet{pearl2009causality} who have proposed a necessary graphical criterion for a conditional instrumental set to be linearly valid, i.e., valid for linear structural equation models.

Besides validity, there is another aspect to this problem that has received less attention: the choice of the conditional instrumental set also affects the estimator's statistical efficiency. While in some cases no linearly valid conditional instrumental set may exist, in other cases multiple linearly valid tuples may be available.
In such cases, each linearly valid tuple leads to a consistent estimator of the total effect but the respective estimators may differ in terms of their asymptotic variance. This raises the following question: can we identify tuples that lead to estimators that are more efficient than others? This is an important problem, as it is well-known that the two-stage least squares estimator can suffer from low efficiency \citep{kinal1980existence}. 
		
It is known that increasing the number of instruments, i.e., enlarging $Z$, can only improve the asymptotic variance \citep[e.g.][Chapter 5]{wooldridge2010econometric}. However, this may also increase the finite sample bias \citep{bekker1994alternative}. A related result that does not suffer from this drawback is given by \citet{kuroki2004selection}. They provide a graphical criterion that can identify in some cases which of two tuples $(Z_1,W)$ and $(Z_2,W)$ provides a more efficient estimator. However, this result is limited to singletons $Z_1$ and $Z_2$ with a fixed conditioning set $W$. The selection of the conditioning set $W$ has received even less attention. The only result we are aware of is by \citet{vansteelandt2018improving}, who establish that if a set $W$ is independent of the instrumental set $Z$ then using $(Z,W)$ rather than $(Z,\emptyset)$ cannot harm the efficiency and may improve it. There are no results for pairs where both the instrumental set and the conditioning set differ.
		
There exists a similar problem in the setting without latent confounding that has received more research attention. In this setting and under a linearity assumption, it is possible to estimate the target total effect by using the ordinary least squares estimator adjusted for a valid adjustment set. There have been a series of more and more general graphical criteria to compare pairs of adjustment sets in terms of the asymptotic variance of the corresponding estimators \citep{kuroki2003covariate,kuroki2004selection,henckel2019graphical}. These advances have allowed \citet{henckel2019graphical} to graphically characterize an adjustment set guaranteed to attain the optimal asymptotic variance among all valid adjustment sets in causal linear models. \citet{rotnitzky2019efficient} show that this result also holds for a large class of non-parametric estimators \citep[see also][]{rotnitzky2022minimum,guo2022variable}. \citet{runge2021necessary} and \citet{smucler2020efficient} provide conditions under which this result generalizes to settings with latent variables, although they do assume that at least one valid adjustment set is observed.

In this paper we aim to fill the gap between the growing literature on selecting adjustment sets and the more limited literature on conditional instrumental sets. We consider the setting of a linear structural equation model with correlated errors that is compatible with a known acyclic directed mixed graph. To set the stage for our results, we first derive a necessary and sufficient graphical condition for a conditional instrumental set to be linearly valid. Moreover, we derive a new formula for the asymptotic variance of the two-stage least squares estimator. This formula only holds for linearly valid conditional instrumental sets but has a simpler dependence on the conditional instrumental set and the underlying graph compared to the traditional formula.

Equipped with these two results, we provide three graphical tools to select more efficient conditional instrumental sets. The first result is a graphical criterion that for certain pairs of linearly valid conditional instrumental sets identifies which of the two corresponding estimators has the smaller asymptotic variance. This criterion includes the results by \citet{kuroki2004selection} and \citet{vansteelandt2018improving} as special cases. It is also the first criterion that can compare pairs of linearly valid tuples where both the instrumental set and the conditioning set differ. This tool also forms the theoretical foundation for our other two tools. The second result is an algorithm that takes any linearly valid conditional instrumental set and greedily adds covariates that decrease the asymptotic variance, while preserving validity. We also use the basic principle underlying our algorithm to propose a simple, intuitive guideline for practitioners on how to select a more efficient tuple when the underlying causal structure is not fully known.
The third result is the graphical characterization of a conditional instrumental set that, given mild constraints, is linearly valid and for which the corresponding estimator has the smallest asymptotic variance we can ensure with a graphical criterion; a property we formally define and call \emph{graphical optimality}. 

Lastly, we provide a simulation study to quantify the gains that can be expected from applying our results in practice. We also illustrate our results on the settler-mortality data set by \citet{acemoglu2001colonial} in order to more efficiently estimate the effect of a country's institutions on its economic wealth. All proofs are provided in the Supplementary Material.

\section{Preliminaries}
\label{prelimIV}

We consider acyclic directed mixed graphs where nodes represent random variables, directed edges $(\rightarrow)$ represent direct effects and bi-directed edges $(\leftrightarrow)$ represent error correlations induced by latent variables. We now give the most important definitions. The remaining ones together with an illustrating example are given in Section \ref{sec:prelim:graph} of the Supplementary Material.

\vsp\noindent \emph{Notation:} We generally use $\{\cdot\}$ to denote sets but we drop the brackets around singletons to ease the notation. For the same reason we use lowercase letters in sub- and superscripts.
 
\vsp\noindent \emph{Linear structural equation model:} 
Consider an acyclic directed mixed graph $\g=(V,E)$, with nodes $V=(V_1,\dots,V_p)$ and edges $E$, where the nodes represent random variables. A random vector $V$ is said to be generated from a linear structural equation model compatible with $\g$ if 
\begin{equation}
V \leftarrow \mathcal{A} V + \epsilon, \label{sem}
\end{equation}
such that the following three properties hold: First, $\mathcal{A}=(\alpha_{ij})$ is a matrix with $\alpha_{ij}=0$ for all $i,j $ where $V_j \rightarrow V_i \notin E$. Second, $\epsilon=(\epsilon_{v_1},\dots,\epsilon_{v_p})$ is a random vector of errors such that $E(\epsilon)=0$ and $\mathrm{cov}(\epsilon)=\Omega=(\omega_{ij})$ is a matrix with $\omega_{ij}=\omega_{ji}=0$ for all $i,j$ where $V_i \leftrightarrow V_j \notin E$. Third, for any two disjoint sets $V', V'' \subseteq V$ such that for all $V_i\in V'$ and all $V_j \in V''$, $V_i \leftrightarrow V_j \notin E$, the random vector $(\epsilon_{v_i})_{v_i \in v'}$ is independent of $ (\epsilon_{v_i})_{v_i \in v''}$. 

We use the symbol $\leftarrow$ in equation \eqref{sem} to emphasize that it is interpreted as a generating mechanism rather than just an equality. As a result we can use it to identify the effect of an outside intervention that sets a treatment $V_i$ to a value $v_i$ uniformly for the entire population. Such interventions are typically called do-interventions and denoted $do(V_i=v_i)$ \citep{pearl1995causal}. The edge coefficient $\alpha_{ij}$ is also called the \emph{direct effect} of $V_j$ on $V_i$ with respect to $V$. The non-zero error covariances $\omega_{ij}$ can be interpreted as the effect of latent variables. 

 \vsp\noindent \emph{Causal paths and forbidden nodes:}  
A path from $X$ to $Y$ in $\g$ is called a causal path from $X$ to $Y$ if all edges on $p$ are directed towards $Y$. The descendants of $X$ in $\g=(V,E)$ are all nodes $D \in V$ such that there exists a causal path from $X$ to $D$ in $\g$ and the set of descendants is denoted by $\de(X,\g)$. We use the convention that $X \in \de(X,\g)$. Moreover, for a set $W=\{W_1,\dots,W_k\}$ we let $\de(W,\g)=\bigcup_{i=1}^k \de(W_i,\g)$.
The causal nodes with respect to $(X,Y)$ in $\g$ are all nodes on causal paths from $X$ to $Y$ excluding $X$ and they are denoted by $\cn{\g}$. 
We define the forbidden nodes relative to $(X,Y)$ in $\g$ as $\f{\g}=\de(\cn{\g}, \g) \cup X$.

\vsp\noindent \emph{Total effects:}
We define the total effect of $X$ on $Y$ at $X=x$ as 
	\[
	\tau_{yx}(x) = \frac{\partial}{\partial x} E\{Y\mid do(X=x)\}.
	\]
In a linear structural equation model the function $\tau_{yx}(x)$ is constant, which is why we simply write $\tau_{yx}$.
The path tracing rules by \citet{wright1934method} allow for the following alternative definition. Consider two nodes $X$ and $Y$ in an acyclic directed mixed graph $\g=(V,E)$ and suppose that $V$ is generated from a linear structural equation model compatible with $\g$. Then $\tau_{yx}$ is the sum over all causal paths from $X$ to $Y$ of the product of the edge coefficients along each such path.

\vsp\noindent \emph{Two-stage least squares estimator:} \citep{basmann1957generalized}
Consider two random variables $X$ and $Y$, and two random vectors $Z$ and $W$. Let $S_{n},T_n$ and $Y_n$ be the random matrices corresponding to $n$ i.i.d. observations from the random vectors $S=(X,W)$, $T=(Z,W)$ and $Y$, respectively. Then the two-stage least squares estimator $\hat{\tau}_{yx}^{z.w}$ is defined as the first entry of
\begin{equation}
    \hat{\gamma}_{ys.t}=Y_n T_n^\top (T_n T_n^\top)^{-1} T_n S_n^\top \{S_n T_n^\top (T_n T_n^\top)^{-1} T_n S_n^\top\}^{-1},
    \label{eq:2sls}
\end{equation}
where we omit the dependence on the sample size $n$ for simplicity and let $\gamma_{ys.t}$ denote the population level version of $\hat{\gamma}_{ys.t}$. The estimator is also commonly written as 
\begin{equation*}
    \hat{\gamma}_{ys.t}=Y_n P_{T_n} S_n^\top \{S_n P_{T_n} S_n^\top\}^{-1},
\end{equation*}
where $P_{T_n}=T_n^\top (T_n T_n^\top)^{-1} T_n$ is the symmetric and idempotent hat matrix from an ordinary least squares regression on $T_n$. This notation emphasizes that we can also obtain $\hat{\gamma}_{ys.t}$ by regressing $S_n$ on $T_n$ obtaining the fitted values $\hat{S}_n=S_n P_{T_n}$ and then regressing $Y_n$ on $\hat{S}_n$; hence the name two-stage least squares estimator.

\vsp\noindent \emph{Latent projection and m-separation:} \citep{koster1999validity,richardson2003markov} Consider an acyclic directed mixed graph $\g=(V,E)$, with $V$ generated from a linear structural equation model compatible with $\g$. We can read off conditional independence relationships between the variables in $V$ directly from $\g$ with a graphical criterion known as m-separation. A formal definition is given in the Supplementary Material. We use the notation $S \perp_{\g} T\mid W$ to denote that $S$ is m-separated from $T$ given $W$ in $\g$. We use the convention that $\emptyset \perp_{\g} T \mid W$ holds for any $T$ and $W$.

Consider a subset of nodes $L \subseteq V$. We can use a tool called the \emph{latent projection} \citep{richardson2003markov} to remove the nodes in $L$ from $\g$ while preserving all m-separation statements between subsets of $V \setminus L$. We use the notation $\g^{L}$ to denote the acyclic directed mixed graph with node set $V\setminus L$ that is the latent projection of $\g$ over $L$. A formal definition of the latent projection is given in the Supplementary Material.

\vsp\noindent \emph{Covariance matrices and regression coefficients:}
Consider random vectors $S=(S_1,\dots,S_{k_s}),T=(T_1,\dots,T_{k_t})$ and $W$.
We denote the population level covariance matrix of $S$ by $\Sigma_{ss} \in\mathbb{R}^{k_s \times k_s}$ and the covariance matrix between $S$ and $T$ by $\Sigma_{st}  \in \mathbb{R}^{k_s \times k_t}$, where its $(i,j)$-th element equals $\mathrm{cov}(S_i,T_j)$. We further define 
$\Sigma_{st.w} = \Sigma_{st} - \Sigma_{sw} \Sigma^{-1}_{ww} \Sigma_{wt}$. If $k_s=k_t=1$, we write $\sigma_{st.w}$ instead of $\Sigma_{st.w}$. The value $\sigma_{ss.w}$ can be interpreted as the residual variance of the ordinary least squares regression of $S$ on $W$. We also refer to $\sigma_{ss.w}$ as the conditional variance of $S$ given $W$.  
Let $\beta_{st.w} \in \mathbb{R}^{k_s \times k_t}$ represent the population level least squares coefficient matrix whose $(i,j)$-th element is the regression coefficient of $T_j$ in the regression of $S_i$ on $T$ and $W$. We denote the corresponding estimator by $\hat{\beta}_{st.w}$. Finally, for random vectors $W_1, \dots, W_m$ with $W=(W_1, \dots, W_m)$ we use the notation that $\beta_{st.{w_1\cdots w_m}} = \beta_{st.w}$ and $\Sigma_{st.w_1 \cdots w_m} = \Sigma_{st.w}$. 

\vsp\noindent \emph{Adjustment sets:} \citep{shpitser2010validity,perkovic16}
A node set $W$ is a valid adjustment set relative to $(X,Y)$ in $\g$, if $\beta_{yx.w}=\tau_{yx}$ for all linear structural equation models compatible with $\g$. There exists a necessary and sufficient graphical criterion for a set $W$ to be a valid adjustment set, which can be found in the Supplementary Material.

\section{Linearly valid conditional instrumental sets}
\label{sec:prep results}

In this section, we graphically characterize the class of conditional instrumental sets $(Z,W)$ such that the two-stage least squares estimator $\hat{\tau}_{yx}^{z.w}$ is consistent for the total effect $\tau_{yx}$ in linear structural equation models.

\begin{definition}
	Consider disjoint nodes $X$ and $Y$, and node sets $Z$ and $W$ in an acyclic directed mixed graph $\g$. 
	We refer to $(Z,W)$ as a linearly valid conditional instrumental set relative to $(X,Y)$ in $\g$ if the following hold: (i) 
    there exists a linear structural equation model compatible with $\g$ such that $\Sigma_{xz.w}\neq 0$ 
	and (ii) for all linear structural equation models compatible with $\g$ such that $\Sigma_{xz.w}\neq 0$, the two-stage least squares estimator $\hat{\tau}_{yx}^{z.w}$ converges in probability to $\tau_{yx}$.
	\label{definition:validCIS}
\end{definition}

Condition (i) of Definition \ref{definition:validCIS} ensures that Condition (ii) is not void. In graphical terms, it corresponds to requiring that $Z \not\perp_{\g} X \mid W$. Condition (ii) then ensures that if $\Sigma_{xz.w} \neq 0$, which can be checked with observational data, the two-stage least squares estimator $\hat{\tau}_{yx}^{z.w}$ is  consistent for the total effect $\tau_{yx}$. If $\Sigma_{xz.w}=0$, the estimator $\hat{\tau}_{yx}^{z.w}$ is generally inconsistent and has non-standard asymptotic theory \citep{staiger1997instrumental}, which is why we do not consider this case. In finite samples, a non-zero but small $\Sigma_{xz.w}$ is also problematic due to weak instrument bias \citep{bound1995problems,stock2002survey} but since we only consider the asymptotic regime this is not relevant for our results.
We now provide a graphical characterization for the class of linearly valid conditional instrumental sets.

\begin{theorem}
	Consider disjoint nodes $X$ and $Y$, and node sets $Z$ and $W$ in an acyclic directed mixed graph $\g$. 
	Then $(Z,W)$ is a linearly valid conditional instrumental set relative to $(X,Y)$ in $\g$ if and only if 
	(i) $(Z \cup W) \cap \f{\g}=\emptyset$, (ii) $Z \not\perp_{\g} X \mid W$ and (iii) $Z \perp_{\tilde{\g}} Y \mid W$,
	where the graph $\tilde{\g}$ is $\g$ with all edges out of $X$ on causal paths from $X$ to $Y$ removed.
	\label{theorem:graphical valid CIS}
\end{theorem}

The graphical criterion in Theorem \ref{theorem:graphical valid CIS} is similar to the well-known graphical criterion from \citet{pearl2009causality} \citep[see also][]{brito2002generalized,brito2002new}. In fact the two are equivalent for any triple $(X,Y,\g)$ with no causal path from $X$ to $Y$ in $\g$ except for the edge $X \rightarrow Y$. The main contribution of Theorem \ref{theorem:graphical valid CIS} is that it gives a necessary and sufficient criterion for general graphs.

By Condition (i) of Theorem \ref{theorem:graphical valid CIS}, a linearly valid conditional instrumental set $(Z,W)$ may not contain nodes in $\f{\g}$. We can therefore use the latent projection and remove the nodes in $F=\f{\g}\setminus\{X,Y\}$ from $\g$ to obtain the smaller graph $\g^{F}$, without loosing any relevant information, i.e., we treat them as unobserved even if we did in fact observe them. This is an example of the forbidden projection originally proposed by \citet{witte2020efficient} in the context of adjustment sets. We formalize this result in the following proposition.

\begin{proposition}
    Consider nodes $X$ and $Y$ in an acyclic directed mixed graph $\g$ and let $F=\f{\g}\setminus\{X,Y\}$. Then $(Z,W)$ is a linearly valid conditional instrumental set relative to $(X,Y)$ in $\g$ if and only if it is a linearly valid conditional instrumental set relative to $(X,Y)$ in $\g^{F}$.
	\label{prop:non forb}
\end{proposition}

 We now provide an additional proposition regarding the remaining descendants of $X$.

\begin{figure}[t]
	\centering
	\subfloat[\label{subfig:weird forb}]{
		\begin{tikzpicture}[>=stealth',shorten >=1pt,auto,node distance=0.7cm,scale=.9, transform shape,align=center,minimum size=3em]
		\node[state] (x) at (0,0) {$X$};
		\node[state] (e) [right =of x] {$M$};
		\node[state] (y) [right =of e] {$Y$};
		\node[state] (d) at ($(x)+(-1.5,-.75)$) {$D$};  
		\node[state] (a) [left =of d] {$A$};  
		\node[state] (b) at ($(x)+(-1.5,+.75)$)   {$B$};  
		\node[state] (c) at ($(y)+(0,1.5)$)   {$C$};  
		
		\path[->]   (x) edge    (e);
		\path[->]   (e) edge    (y);
		\path[<->]   (x) edge  [bend left]  (y);
		\path[->]   (x) edge  [bend right]  (y);
		\path[<->]   (d) edge     (x);
		\path[->]   (a) edge     (d);
		\path[->]   (b) edge     (x);
		\path[->]   (b) edge     (d);
		\path[->]   (c) edge     (b);
		\path[->]   (c) edge     (y);
		\end{tikzpicture}
	} 
	\hspace{1cm}
	\subfloat[\label{subfig:weird}]{
		\begin{tikzpicture}[>=stealth',shorten >=1pt,auto,node distance=0.7cm,scale=.9, transform shape,align=center,minimum size=3em]
		\node[state] (x) at (0,0) {$X$};
		\node[state] (y) [right =of x] {$Y$};
		\node[state] (d) at ($(x)+(-1.5,-.75)$) {$D$};  
		\node[state] (a) [left =of d] {$A$};  
		\node[state] (b) at ($(x)+(-1.5,+.75)$)   {$B$};  
		\node[state] (c) at ($(y)+(0,1.5)$)   {$C$};  
		\path[->]   (x) edge    (y);
		\path[<->]   (x) edge  [bend left]  (y);
		\path[<->]   (d) edge     (x);
		\path[->]   (a) edge     (d);
		\path[->]   (b) edge     (x);
		\path[->]   (b) edge     (d);
		\path[->]   (c) edge     (b);
		\path[->]   (c) edge     (y);
		\end{tikzpicture}
	} 
	\caption{\protect\subref{subfig:weird forb} Acyclic directed mixed graph for Example \ref{ex:valid CIS} and \protect\subref{subfig:weird} Acyclic directed mixed graph for Examples \ref{ex:valid CIS}, \ref{ex:ambigous}, \ref{ex:weird}, \ref{ex:algorithm} and \ref{ex:optimal cis}.}
	\label{fig:weird}
\end{figure}
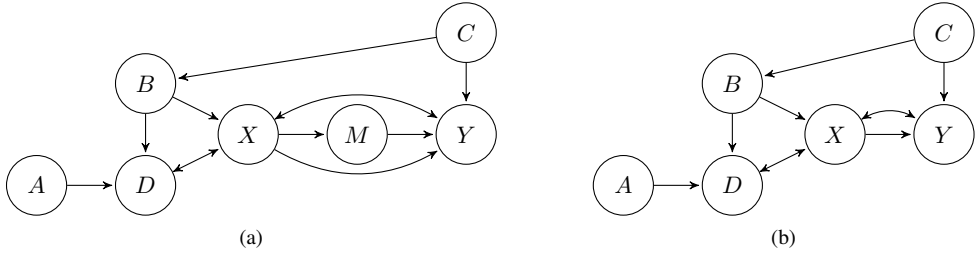

\begin{proposition}
	Consider nodes $X$ and $Y$ in an acyclic directed mixed graph $\g$ and let $(Z,W)$ be a linearly valid conditional instrumental set relative to $(X,Y)$ in $\g$. If $(Z \cup W) \cap \de(X,\g) \neq \emptyset$, then $W$ is a valid adjustment set relative to $(X,Y)$ in $\g$. 
	\label{prop:descX}
\end{proposition}

Proposition \ref{prop:descX} implies that whenever we may use descendants of $X$ in the conditional instrumental set, we can also use covariate adjustment to estimate $\tau_{yx}$, that is, use the ordinary least squares estimator. As the latter is known to be a more efficient estimator than the two-stage least squares estimator \cite[e.g. Chapter 5.2.3 of][ Section \ref{sec:asy formula app} of the Supplementary Material]{wooldridge2010econometric}, we disregard such cases and apply the latent projection to also marginalize out any remaining variables in $\de(X,\g)\setminus\{X,Y\}$. We also assume that $Y \in \de(X,\g)$ as otherwise $\tau_{yx}=0$ by default. We summarize our assumptions in a remark at the end of this section.

\begin{example}[Linearly valid conditional instrumental sets]\label{ex:valid CIS}
We now characterize all linearly valid conditional instrumental sets $(Z,W)$ relative to $(X,Y)$ in the graph $\g$ from Figure \ref{subfig:weird forb} using Theorem \ref{theorem:graphical valid CIS}. First, $\cn{\g} = \{M,Y\}$ and $\f{\g} = \{X,M,Y\}$. By Condition (i), we therefore only need to consider sets $Z$ and $W$ that are subsets of $\{A,B,C,D\}$. 

We first consider potential instrumental sets $Z$. Since $C \not\perp_{\tilde{\g}} Y \mid  W$ for all $W \subseteq \{A,B,D\}$, no $Z$ may contain $C$ or we have a violation of Condition (iii). This means we have seven candidates for $Z:  A,B,D,\{A,B\},\{A,D\},\{B,D\}$ and $\{A,B,D\}$. We now consider the potential conditioning sets corresponding to each of these seven candidates.

Consider first the four candidate instrumental sets $B,\{A,B\},\{B,D\}$ and $\{A,B,D\}$. Letting $Z$ be any of these sets, it holds that $Z \not\perp_{\tilde{\g}} X \mid  W$ for all $W \subseteq \{A,B,C,D\} \setminus Z$. Therefore Condition (ii) holds irrespective of $W$. Further, $Z \perp_{\tilde{\g}} Y \mid  W$ with $W \subseteq \{A,B,C,D\}\setminus Z$ if and only if $C \in W$. Therefore, Condition (iii) holds and $(Z,W)$ is a linearly valid conditional instrumental set if an only if $C \in W$. The conditioning sets for $B$ are therefore $C,\{A,C\},\{C,D\}$ and $\{A,C,D\}$, for $\{A,B\}$ they are $C$ and $\{C,D\}$, for $\{B,D\}$ they are $C$ and $\{A,C\}$, and for $\{A,B,D\}$ there is only $C$.

Consider now the two candidate instrumental sets $D$ and $\{A,D\}$. Letting $Z$ be either of these two sets, it holds that $Z \not\perp_{\tilde{\g}} X \mid  W$ for all $W \subseteq \{A,B,C,D\} \setminus Z$ and therefore Condition (ii) holds irrespective of $W$. Further, $Z \perp_{\tilde{\g}} Y \mid  W$ with $W \subseteq \{A,B,C,D\}\setminus Z$ if and only if $B \in W$ or $C \in W$. Therefore, Condition (iii) holds and $(Z,W)$ is a linearly valid conditional instrumental set if and only if $B \in W$ or $C \in W$. The conditioning sets for $D$ are therefore $B,C,\{A,B\},\{A,C\},\{B,C\}$ and $\{A,B,C\}$, while for $\{A,D\}$ they are $B,C$ and $\{B,C\}$.

Finally consider the case $Z=A$. It holds that $A \not\perp_{\g} X \mid  W$ for $W \subseteq \{A,B,C,D\}\setminus Z$ if and only if $D \in W$ and therefore Condition (ii) holds if and only if $D \in W$. Further, given any $W \subseteq \{A,B,C,D\}$ with $D \in W$, $A \perp_{\tilde{\g}} Y \mid  W$ if and only if $B \in W$ or $C \in W$. The conditioning sets for $Z=A$ are therefore $\{B,D\},\{C,D\}$ and $\{B,C,D\}$. In total there are therefore $21$ linearly valid conditional instrumental sets (see Table \ref{table:valid cis} in the Supplementary Material).

Consider now the graph $\g'$ from Figure \ref{subfig:weird}. It is the forbidden projection graph $\g^{F}$ with $F=M$. It is easy to verify that the arguments we gave for $\g$ also apply to $\g'$ and therefore every linearly valid conditional instrumental set relative to $(X,Y)$ in $\g$ is also a linearly valid conditional instrumental set relative to $(X,Y)$ in $\g'$ and vice versa.

\end{example}

\begin{remark}
  For the remainder of this paper we consider graphs $\g$ such that $ \de(X,\g) =\{X,Y\}$. The graph $\tilde{\g}$ then simply equals the graph $\g$ with the edge $X \rightarrow Y$ removed.
\end{remark}

\section{Efficient and linearly valid conditional instrumental sets}
\label{mainIV}

\subsection{Asymptotic variance formula for linearly valid conditional instrumental sets}
\label{subsec:asymptotic variance}

We present a new formula for the asymptotic variance of the two-stage least squares estimator, in terms of three conditional variances, that is, terms of the form $\sigma_{xx.z}$ as defined in Section \ref{prelimIV}. The new formula has a simpler dependence on the tuple $(Z,W)$ than the traditional formula (see Equation \eqref{2SLSavar_old} below), but it only holds if the tuple $(Z,W)$ is a linearly valid conditional instrumental set. 

\begin{theorem}
    \label{theorem:varequ}
	Consider nodes $X$ and $Y$ in an acyclic directed mixed graph $\g$ such that $\de(X,\g)=\{X,Y\}$. 
	Let $(Z,W)$ be a linearly valid conditional instrumental set relative to $(X,Y)$ in $\g$ and $\tilde{Y}=Y-\tau_{yx}X$. Then for all linear structural equation models compatible with $\g$ such that $\Sigma_{xz.w}\neq 0$, $\hat{\tau}_{yx}^{z.w}$ is an asymptotically normal estimator of $\tau_{yx}$ with asymptotic variance
	\begin{equation}\label{2SLSavar}
	a.var(\hat{\tau}_{yx}^{z.w}) = \frac{\sigma_{\tilde{y}\tilde{y}.w}}{\sigma_{xx.w}-\sigma_{xx.zw}}. 
	\end{equation}
\end{theorem}

The new formula in Equation \eqref{2SLSavar} differs in two ways from the traditional asymptotic variance formula which is
\begin{equation}\label{2SLSavar_old}
a.var(\hat{\tau}_{yx}^{z.w}) = \left(\eta_{ys.t} (\Sigma_{st}\Sigma^{-1}_{tt} \Sigma_{st}^\top)^{-1}\right)_{1,1},
\end{equation}
where $S=(X,W), T=(Z,W)$ and $\eta_{ys.t}=\mathrm{var}(Y-\gamma_{ys.t} S)$ is the population level residual variance of the estimator $\hat{\gamma}_{ys.t}$ defined in Section \ref{prelimIV}. The first difference is that in Equation \eqref{2SLSavar} the residual variance $\eta_{ys.t}$ from Equation \eqref{2SLSavar_old} is replaced by the conditional variance $\sigma_{\tilde{y}\tilde{y}.w}$ of the oracle random variable $\tilde{Y}=Y-\tau_{yx}X$ on $W$. This is possible because under the assumption that $(Z,W)$ is a linearly valid conditional instrumental set, we have $\eta_{ys.t}=\sigma_{\tilde{y}\tilde{y}.w}$. The advantage of this change is that it is easier to describe how the term $\sigma_{\tilde{y}\tilde{y}.w}$ behaves as a function of the tuple $(Z,W)$ than $\eta_{ys.t}$. For example, it is immediately clear that $\sigma_{\tilde{y}\tilde{y}.w}$ depends on $(Z,W)$ only via $W$. We refer to the numerator of Equation $\eqref{2SLSavar}$, $\sigma_{\tilde{y}\tilde{y}.w}$, as the residual variance of the estimator $\hat{\tau}_{yx}^{z.w}$ or, when $(X,Y)$ is clear, of the tuple $(Z,W)$. 

The second change is that in Equation \eqref{2SLSavar} the denominator from Equation \eqref{2SLSavar_old} is replaced with the difference $\sigma_{xx.w}-\sigma_{xx.zw}$. The difference $\sigma_{xx.w}-\sigma_{xx.zw}$ measures how much the residual variance of $X$ on $W$ decreases when we add $Z$ to the conditioning set. Intuitively, this is the information on $X$ that $Z$ contains and which was not already contained in $W$. Based on this intuition, we refer to the denominator of Equation \eqref{2SLSavar}, $\sigma_{xx.w}-\sigma_{xx.zw}$, as the conditional instrumental strength of the estimator $\hat{\tau}_{yx}^{z.w}$ or, when $(X,Y)$ is clear, of $(Z,W)$. 

Another important contribution of Theorem \ref{theorem:varequ} is that it also holds for linear structural equation models with non-Gaussian errors. This is a non-trivial result, because Equation \eqref{2SLSavar_old} is usually derived under a homoscedasticity assumption on the residuals $Y-\gamma_{ys.t} S$, which generally does not hold if the errors in the underlying linear structural equation model are non-Gaussian. However, in our proof we show that if $(Z,W)$ is a \emph{linearly valid} conditional instrumental set, Equations \eqref{2SLSavar} and \eqref{2SLSavar_old} hold, even if the residuals are heteroskedastic. 

Finally, the new asymptotic variance formula elegantly mirrors the ordinary least squares asymptotic variance formula $a.var(\hat{\beta}_{yx.w})=\sigma_{yy.xw}/\sigma_{xx.w}$. We can use this to illustrate how the two-stage least squares estimator is less efficient than the ordinary least squares estimator. We discuss this further in Section \ref{sec:asy formula app} of the Supplementary Material.

\subsection{Comparing pairs of linearly valid conditional instrumental sets}
\label{subsec:efficiency}

We derive a graphical criterion that, for certain pairs of linearly valid conditional instrumental sets, identifies which of the two corresponding estimators has the smaller asymptotic variance.

\begin{theorem}
	Consider nodes $X$ and $Y$ in an acyclic directed mixed graph $\g$ such that $\de(X,\g)=\{X,Y\}$. Let $(Z_1,W_1)$ and $(Z_2,W_2)$ be linearly valid conditional instrumental sets relative to $(X,Y)$ in $\g$. Let $W_{1\setminus 2}=W_1 \setminus W_2$ and $W_{2\setminus 1}=W_2 \setminus W_1$. If 
	\begingroup\leqnos
	\begin{align}
	    & W_{1\setminus 2} \perp_{\tilde{\g}} Y \mid W_2, \tag{a} \label{condition1}\\
	    & \text{(i) } W_{1\setminus 2} \perp_{\g} Z_2 \mid W_2 \text{ or (ii) }W_{1\setminus 2} \setminus Z_2 \perp_{\g} X \mid Z_2 \cup W_2, \tag{b} \label{condition2}\\
	    & \text{(i) } W_{2\setminus 1} \setminus Z_1 \perp_{\g} Z_1\mid W_1 \text{ and } W_{2\setminus 1} \cap Z_1 \perp_{\g} X \mid W_1 \cup (W_{2\setminus 1}\setminus Z_1) \text{ or} \tag{c} \label{condition3}\\ 
	    & \quad \text{(ii) } W_{2\setminus 1} \perp_{\g} X \mid  W_1, \nonumber\\
	    & Z_1 \setminus (Z_2 \cup W_{2\setminus 1}) \perp_{\g} X \mid Z_2 \cup W_1 \cup W_{2\setminus 1}, \tag{d} \label{condition4}
	\end{align}\endgroup
	then for all linear structural equation models compatible with $\g$ such that $\Sigma_{xz_1.w_1}\neq 0$, 
	$
	a.var(\hat{\tau}_{yx}^{z_2.w_2}) \leq a.var(\hat{\tau}_{yx}^{z_1.w_1}).
    $
	\label{theorem:CIVcomparison}
\end{theorem}

Theorem \ref{theorem:CIVcomparison} is more general in terms of the linearly valid conditional instrumental sets it can compare, than the results by \cite{kuroki2004selection} and \citet{vansteelandt2018improving}. In particular, it can compare tuples $(Z_1,W_1)$ and $(Z_2,W_2)$ with $Z_1 \neq Z_2$ and $W_1 \neq W_2$. 

As a result, the graphical condition of Theorem \ref{theorem:CIVcomparison} is, however, rather complex. For easier intuition, we can think of it as consisting of two separate parts. Condition \eqref{condition1} verifies that $(Z_1,W_1)$ may not provide a smaller residual variance than $(Z_2,W_2)$, that is, $\sigma_{\tilde{y}\tilde{y}.w_1} \geq \sigma_{\tilde{y}\tilde{y}.w_2}$. It does so by checking that the covariates in $W_{1\setminus 2}$ do not reduce the residual variance of the oracle random variable $\tilde{Y}=Y-\tau_{yx}X$. Because $\tilde{Y}$ is not a node in $\g$ but can be thought of as the node $Y$ in the graph $\tilde{\g}$, this requires checking an m-separation statement in $\tilde{\g}$.
Conditions (\ref{condition2}--\ref{condition4}) jointly verify that $(Z_1,W_1)$ may not provide a larger conditional instrumental strength than $(Z_2,W_2)$, that is, $\sigma_{xx.w_1} -\sigma_{xx.z_1w_1} \leq \sigma_{xx.w_2} -\sigma_{xx.z_2w_2}$. This requires three conditions because we need to verify in turn that (i) the covariates in $W_{1\setminus 2}$ do not increase, (ii) the nodes in $W_{2\setminus 1}$ do not decrease and (iii) the nodes in $Z_1 \setminus (Z_2 \cup W_{2\setminus 1})$ do not increase the conditional instrumental strength.

Theorem \ref{theorem:CIVcomparison} also simplifies whenever $W_1 \subseteq W_2$, $W_2 \subseteq W_1$ or $Z_1 \subseteq Z_2$. In the first case, $W_{1\setminus 2} = \emptyset$ and therefore Conditions \eqref{condition1} and \eqref{condition2} are void. In the second case, $W_{2\setminus 1} = \emptyset$ and therefore Condition \eqref{condition3} is void. In the third case, $Z_1\setminus Z_2 =\emptyset$ and therefore Condition \eqref{condition4} is void. If $W_1=W_2$ and $Z_1 \subseteq Z_2$ all four conditions are void. Therefore, the well-known result that adding covariates to the instrumental set is beneficial for the asymptotic variance \citep[e.g. Chapter 12.17 of][]{hansen2019} is a corollary of Theorem \ref{theorem:CIVcomparison}. We formally state it for completeness.

\begin{corollary}
	\label{corollary:more instruments}
	Consider nodes $X$ and $Y$ in an acyclic directed mixed graph $\g$ such that $\de(X,\g)=\{X,Y\}$. Let $(Z_1,W_1)$ and $(Z_2,W_2)$ be linearly valid conditional instrumental sets relative to $(X,Y)$ in $\g$. If $Z_1 \subseteq Z_2$ and $W_1=W_2=W$, then for all linear structural equation models compatible with $\g$ such that $\Sigma_{xz_1.w}\neq 0$, $a.var(\hat{\tau}_{yx}^{z_2.w}) \leq 	a.var(\hat{\tau}_{yx}^{z_1.w}).$
\end{corollary}

Theorem \ref{theorem:CIVcomparison} also gives interesting new insights. We first discuss the special case with $Z$ fixed and only $W$ varying. Here, the most important insight is that there exists a class of covariates that practitioners should avoid adding to $W$ because they increase the asymptotic variance. 
We now illustrate this class along with other consequences in a series of examples.

\begin{example}[Harmful conditioning]\label{ex:harmful}
    Consider the graph $\g$ from Fig. \ref{subfig:simple}. We are interested in estimating $\tau_{yx}$ with conditional instrumental sets of the form $(D,W)$, where $W \subseteq \{A,B,C\}$. Since $\f{\g}=\{X,Y\}$, $D \not\perp_{\g} X \mid W$ and $D \perp_{\tilde{\g}} Y \mid W$ for all $W \subseteq \{A,B,C\}$, all tuples of this form are linearly valid conditional instrumental sets relative to $(X,Y)$ in $\g$. 
    
    Let $W \subseteq \{B,C\}$ and $W' = W \cup A$. As $A \perp_{\tilde{\g}} Y \mid W$ and $A \perp_{\g} X \mid W \cup D$,
    we can apply Theorem \ref{2SLSavar} with $W_1=W',W_2=W$ and $Z_1=Z_2=D$ (Conditions \eqref{condition1} and \eqref{condition2} (ii) hold, Conditions \eqref{condition3} and \eqref{condition4} are void). We can therefore conclude that adding $A$ to any $W \subseteq \{B,C\}$ can only increase the asymptotic variance. 
    
    Conditioning on $A$ is harmful, because $A \ci \tilde{Y}\mid W$ and $A \ci X\mid D \cup  W$ but $A \notci X \mid W$. Therefore, $\sigma_{\tilde{y}\tilde{y}.wa}=\sigma_{\tilde{y}\tilde{y}.w}$ and $\sigma_{xx.wad}=\sigma_{xx.wd}$ but $\sigma_{xx.wa} \leq \sigma_{xx.w}$. 
    The node $A$ is representative of a larger class of covariates that we should avoid conditioning on because they do not affect the residual variance and may reduce the conditional instrumental strength.
\end{example}

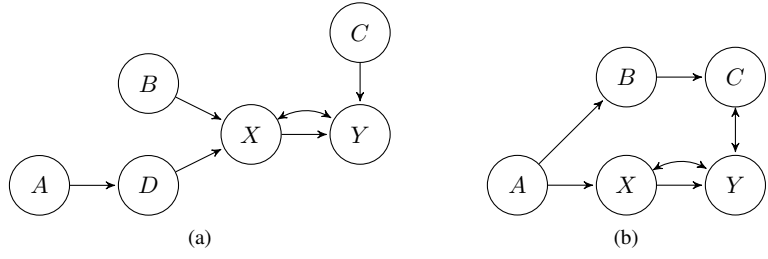
\begin{figure}[t]
	\centering
	\subfloat[\label{subfig:simple}]{
		\begin{tikzpicture}[>=stealth',shorten >=1pt,auto,node distance=0.7cm,scale=.9, transform shape,align=center,minimum size=3em]
		\node[state] (x) at (0,0) {$X$};
		\node[state] (y) [right =of x] {$Y$};
		\node[state] (w1) at ($(x)+(-1.5,-.75)$) {$D$};  
		\node[state] (w2) at ($(y)+(0,+1.5)$)   {$C$};  
		\node[state] (w3) at ($(x)+(-1.5,+.75)$)   {$B$};  
		\node[state] (w4) [left =of w1]   {$A$};  
		\path[->]   (x) edge    (y);
		\path[<->]   (x) edge  [bend left]  (y);
		\path[->]   (w1) edge     (x);
		\path[->]   (w2) edge     (y);
		\path[->]   (w3) edge     (x);
		\path[->]   (w4) edge     (w1);
		\end{tikzpicture}
	} 
	\hspace{1cm}
	\subfloat[\label{subfig:no optimal}]{
		\begin{tikzpicture}[>=stealth',shorten >=1pt,auto,node distance=0.7cm,scale=.9, transform shape,align=center,minimum size=3em]
	\node[state] (x) at (0,0) {$X$};
	\node[state] (y) [right =of x] {$Y$};
	\node[state] (w1) [above =of x] {$B$}; 
	\node[state] (z) [left = of x] {$A$};   
	\node[state] (w2)  [above =of y]   {$C$};  
	
	\path[->]   (x) edge    (y);
	\path[<->]   (x) edge  [bend left]  (y);
	\path[->]   (z) edge     (x);
	\path[->]   (z) edge     (w1);
	\path[->]   (w1) edge     (w2);
	\path[<->]   (w2) edge     (y);
	
	\end{tikzpicture}
	} 
	\caption{\protect\subref{subfig:simple} Acyclic directed mixed graph for Examples  \ref{ex:harmful} and \ref{ex:beneficial}. \protect\subref{subfig:no optimal} Acyclic directed mixed graph for Examples \ref{ex:no optimal} and \ref{ex:optimal cis}.}
	\label{fig:bad neutral good}
\end{figure}

\begin{example}[Beneficial conditioning]\label{ex:beneficial}
    Consider the graph $\g$ from Fig. \ref{subfig:simple} and any two tuples of the form 
    $(D,W)$ and $(D,W')$, where $W \subseteq \{A,B\}$ and $W'=W\cup C$. By Example \ref{ex:harmful}, any such tuple is a linearly valid conditional instrumental set relative to to $(X,Y)$ in $\g$. Further, as $C \perp_{\g} D \mid W$ we can apply Theorem \ref{2SLSavar} with $W_1=W,W_2=W'$ and $Z_1=Z_2=D$  (Condition \eqref{condition3} (i) holds, Conditions \eqref{condition1}, \eqref{condition2} and \eqref{condition4} are void). We can therefore conclude that adding $C$ to any $W \subseteq \{A,B\}$ can only decrease the asymptotic variance.
    
    Conditioning on $C$ is beneficial, because $C \notci \tilde{Y} \mid W$ and $C \ci D \mid W$. Therefore, $\sigma_{\tilde{y}\tilde{y}.wc} \leq \sigma_{\tilde{y}\tilde{y}.w}$ but $\sigma_{xx.w}-\sigma_{xx.wd}=\sigma_{xx.wc}-\sigma_{xx.wcd}$ (see Lemma \ref{lemma:equalinstruments} in the Supplementary Material). Interestingly, this remains true if we add the edge $C \rightarrow X$ to $\g$, making $C$ a confounder. The covariate $C$ is representative of a larger class of covariates that we should aim to condition on because they may reduce the residual variance and do not affect the conditional instrumental strength. 
    This class, in particular includes confounders between $X$ and $Y$ that are independent of $Z$.
\end{example}

\begin{example}[Ambiguous conditioning]\label{ex:ambigous}
    Consider the graph $\g$ from Fig. \ref{subfig:weird}. Consider the two tuples $(D,C)$ and $(D,\{B,C\})$.  By Example \ref{ex:valid CIS}, both tuples are linearly valid conditional instrumental sets relative to $(X,Y)$ in $\g$. Further, as $B \not\perp_{\g} D \mid C$ and $B \not\perp_{\g} X \mid \{C,D\}$ we cannot apply Theorem \ref{2SLSavar} with $W_1=C,W_2=\{B,C\}$ and $Z_1=Z_2=D$ (Condition \eqref{condition3} does not hold). For the same reason, we cannot apply Theorem \ref{2SLSavar} with $W_1=\{B,C\},W_2=C$ and $Z_1=Z_2=D$ (Condition \eqref{condition2} does not hold). We can therefore not use Theorem \ref{theorem:CIVcomparison} to determine the effect of adding $B$ to the conditioning set on the asymptotic variance. 
    
\end{example}

There are two additional interesting classes of covariates, one neutral and one beneficial, that we discuss in Section \ref{sec: ex+} of the Supplementary Material.

Since Theorem \ref{theorem:CIVcomparison} can compare tuples $(Z_1,W_1)$ and $(Z_2,W_2)$ with $Z_1 \neq Z_2$ and $W_1 \neq W_2$, it also gives interesting new insights regarding covariates that we may add both to $Z$ and $W$. The following corollary is an example of such a result. 

\begin{corollary}
	\label{corollary:always instrument}
	Consider nodes $X$ and $Y$ in an acyclic directed mixed graph $\g$ such that $\de(X,\g) = \{X,Y\}$. Let $(Z,W\cup S)$ be a linearly valid conditional instrumental set relative to $(X,Y)$ in $\g$. If $(Z \cup S,W)$ is a linearly valid conditional instrumental set relative to $(X,Y)$ in $\g$, then for all linear structural equation models compatible with $\g$ such that $\Sigma_{xz.ws}\neq 0$, it holds that
	$a.var(\hat{\tau}_{yx}^{zs.w}) \leq a.var(\hat{\tau}_{yx}^{z.ws}).$
\end{corollary}

Intuitively, Corollary \ref{corollary:always instrument} states that covariates that may be added to both $Z$ and $W$ should be added to $Z$. 
If we restrict ourselves to the special case that $S$ is a singleton $N$, we can also show the following complementary result.

\begin{proposition}\label{prop:moreW}
	Consider nodes $X,Y$ and $N$ in an acyclic directed mixed graph $\g$ such that $\de(X,\g) = \{X,Y\}$. Let $(Z,W)$ be a linearly valid conditional instrumental sets relative to $(X,Y)$ in $\g$. 
	If $(Z,W\cup N)$ is a linearly valid conditional instrumental set relative to $(X,Y)$ in $\g$ but $(Z\cup N,W)$ is not, then for all linear structural equation models compatible with $\g$ such that $\Sigma_{xz.w}\neq 0$,
	$a.var(\hat{\tau}_{yx}^{z.wn}) \leq a.var(\hat{\tau}_{yx}^{z.w}).$
\end{proposition}   

Corollaries \ref{corollary:more instruments} and \ref{corollary:always instrument} along with Proposition \ref{prop:moreW} are very useful tools for identifying efficient tuples. We illustrate this by revisiting the graph from Fig. $\ref{subfig:weird}$ and using the three results to identify a tuple guaranteed to provide the smallest asymptotic variance among all linearly valid tuples.

\begin{example}[Asymptotically optimal tuple] \label{ex:weird}
	Consider the graph $\g$ in Fig. \ref{subfig:weird}. By Example \ref{ex:valid CIS} there are $21$ linearly valid conditional instrumental sets relative to $(X,Y)$ in $\g$, which we list in Table \ref{table:valid cis} of the Supplementary Material. 
Let $(Z,W)$ be a linearly valid tuple such that $Z\cup W = \{A,B,C,D\}$. Then, $C \in W$. We can therefore apply Corollary \ref{corollary:always instrument} with $S = \{A,B,D\} \setminus Z$ and conclude that the linearly valid tuple $(\{A,B,D\},C)$ is more efficient than $(Z,W)$. Consider now a linearly valid tuple $(Z,W)$ such that $Z\cup W \neq \{A,B,C,D\}$. We now construct a more efficient linearly valid tuple $(Z',W')$ such that $Z'\cup W' = \{A,B,C,D\}$. If $C \notin W$, then $(Z,W')$ with $W'=W\cup C$ is a linearly valid tuple. Further, $C \notin Z$ for all linearly valid tuples in $\g$. Therefore, we can invoke Proposition \ref{prop:moreW} with $N=C$ and conclude that $(Z,W')$ is more efficient than $(Z,W)$. Let $S=\{A,B,C,D\} \setminus (Z\cup W')$ and $Z'=Z\cup S$. Since $C \in W'$, $(Z',W')$ is a linearly valid tuple and we can therefore apply Corollary \ref{corollary:more instruments} to conclude that $(Z',W')$ is more efficient than $(Z,W')$. Combining the two results, we can conclude that $(\{A,B,D\},C)$ provides the smallest attainable asymptotic variance among all linearly valid tuples. 
\end{example}

\subsection{Greedy forward procedure}
\label{subsec:algorithm}

As shown in the previous section, the asymptotic variance can behave in complex ways. 
However, in the simple case where we are given a linearly valid tuple $(Z,W)$ and a covariate $N$, we can use the results of Section \ref{subsec:efficiency} to derive the following simple rules. If we can add $N$ to $Z$, then we should do so by Corollaries
\ref{corollary:more instruments} and \ref{corollary:always instrument}. If we cannot add $N$ to $Z$, but may add it to $W$, then we should add $N$ to $W$ by Proposition \ref{prop:moreW}.

\begin{algorithm}[t]
\caption{ Greedy forward procedure \label{algorithm}}
 	\LinesNumbered 
 	\SetKwData{Set}{Set}
 	\SetKwFunction{Foreach}{Foreach}
 	\SetKwInOut{Input}{input}
 	\SetKwInOut{Output}{output}
 	\Input{Acyclic directed mixed graph $\g=(V,E)$ and $X,Y,Z,W$ such that $(Z,W)$ is a linearly valid conditional instrumental set relative to $(X,Y)$ in $\g$}
 	\Output{Linearly valid conditional instrumental set $(Z',W')$ relative to $(X,Y)$ in $\g$ such that $a.var(\hat{\tau}^{z'.w'}_{yx}) \leq a.var(\hat{\tau}^{z.w}_{yx})$} 
	\BlankLine
  	\Begin{
  		$Z' = Z, W'=W, V'=V \setminus (\{X,Y\} \cup Z \cup W$); \\
        \ForEach{$N \in W'$}{
        \If{ $ (Z'\cup N,W'\setminus \{N\})$ linearly valid}
  			{ $Z' = Z' \cup N$ \text{ and} $W' = W' \setminus N$;}}
  		\ForEach{ $N \in V'$}{
  			\If{ $ (Z'\cup N,W')$ linearly valid \rm{and} $N \not\perp_{\g}X\mid W'\cup Z'$}
  			{ $Z' = Z' \cup N$;}
 			 \ElseIf{ $ (Z',W' \cup N)$ linearly valid \rm{and} $N \not\perp_{\tilde{\g}} Y \mid W'$}
 			 {$W'=W' \cup N$}}}
 		\Return{$(Z',W')$;}
\end{algorithm}

Based on these two rules, we propose the greedy two phase Algorithm \ref{algorithm} which given a linearly valid conditional instrumental set $(Z,W)$ returns a more efficient tuple. In the first phase, Algorithm \ref{algorithm} moves all covariates in $W$ where this is possible to $Z$. In the second phase, it greedily adds any additional covariates to either $Z$ or $W$ while minimizing the asymptotic variance.
The m-separation checks in the second phase of Algorithm \ref{algorithm} are not necessary in the sense that we could drop them and the algorithm would still greedily minimize the asymptotic variance. We include them, nonetheless, as otherwise the algorithm would add covariates that do not affect the asymptotic variance to the tuple. This would needlessly grow the output tuple, likely negatively affecting the finite sample performance. Adding the two m-separation checks ensures that in each step the algorithm only adds nodes to the tuple that may decrease the asymptotic variance. We prove the soundness of Algorithm \ref{algorithm} in Section \ref{sec:algo prelim} of the Supplementary Material. We emphasize that due to the second phase, the output tuple is order dependent and has no efficiency guarantee when compared to all linearly valid tuples in $\g$.

\begin{example}[Illustrating Algorithm \ref{algorithm}]
	\label{ex:algorithm}
	Consider the graph $\g$ from Fig. \ref{subfig:weird} (see Table \ref{table:valid cis} in the Supplementary Material for a list of the linearly valid tuples in $\g$). Suppose our starting linearly valid conditional instrumental set is $(B,C)$. Since $(\{B,C\},\emptyset)$ is not a linearly valid tuple, the first phase returns $(B,C)$. Regarding the second phase, consider first the case that Algorithm \ref{algorithm} checks $A$ before $D$. Because $A\perp_{\g} X \mid \{B,C\}$ and $A \perp_{\tilde{\g}} Y \mid C$, Algorithm \ref{algorithm} discards $A$ even though $(\{A,B\},C)$ and $(B,\{A,C\})$ are linearly valid tuples. Because $(\{B,D\},C)$ is a linearly valid tuple and $D\not\perp X \mid \{B,C\}$, it then adds $D$ to $Z'$ and outputs the tuple $(\{B,D\},C)$. Consider now the case that Algorithm \ref{algorithm} checks $D$ before $A$. Again, $D$ is added to $Z'$. However, as $(\{A,B,D\},C)$ is a linearly valid tuple and $A \not\perp_{\g} X \mid \{B,C,D\}$, Algorithm \ref{algorithm} also adds $A$ to $Z'$. The output tuple is therefore $(\{A,B,D\},C)$ in this case.
	Both $(\{B,D\},C)$ and $(\{A,B,D\},C)$ are more efficient than $(B,C)$ but only $(\{A,B,D\},C)$ provides the smallest attainable asymptotic variance among all tuples. 
\end{example}

 Based on the two rules underpinning Algorithm \ref{algorithm}, we propose the following guidelines for practitioners. They are particularly useful in cases where we have a basic understanding of the causal structure, but not enough knowledge of the causal graph to apply Theorem \ref{theorem:CIVcomparison} directly.

\begin{remark}[Practical guidelines]\label{remark:guidelines}
Covariates that can be used as instrumental variables should be used as instrumental variables. Covariates that can \emph{only} be used as conditioning variables should be used as conditioning variables. 
\end{remark}

\subsection{Graphically optimal linearly valid conditional instrumental sets}
\label{subsec:CIVselection}

    In Example \ref{ex:weird} we identified a linearly valid conditional instrumental set that provides the smallest attainable asymptotic variance among all linearly valid tuples using only the graph. In the closely related literature on efficient valid adjustment sets \citet{henckel2019graphical} referred to this property as asymptotic optimality. They also provided a graphical characterization of a valid adjustment set that for linear structural equation models with independent errors is asymptotically optimal and showed that in cases with latent variables, i.e., correlated errors, no asymptotically optimal valid adjustment set may exist. The following example shows that, similarly, an asymptotically optimal linearly valid conditional instrumental set may not exist.

	\begin{example}[No asymptotically optimal tuple]
	Consider the graph $\g$ from Fig. \ref{subfig:no optimal}. There are five linearly valid tuples with respect to $(X,Y)$ in $\g$:  $(A,\emptyset),(B,\emptyset),(\{A,B\},\emptyset)$, $(A,B)$ and $(A,\{B,C\})$. We consider two linear structural equation models compatible with $\g$: for model $\mathcal{M}_1$ let all error variances and edge coefficients be $1$, where we model the bi-directed edges with a latent variable, i.e, $X \leftarrow L_1 \rightarrow Y$ and $C \leftarrow L_2 \rightarrow Y$. For model $\mathcal{M}_2$ do the same, except for setting the edge coefficient for the edge $A \rightarrow B$ to $0.1$. For both models we computed the asymptotic variance corresponding to the five linearly valid tuples. The results are given in Table \ref{table:no optimal} and they show that in $\mathcal{M}_1$ the most efficient tuples are $(A,\emptyset)$ and $(\{A,B\},\emptyset)$ while in $\mathcal{M}_2$ the most efficient tuple is $(A,\{B,C\})$. Therefore, there is no asymptotically optimal tuple.
	
	The reason that $(A,\emptyset)$ is more efficient than $(A,\{B,C\})$ in $\mathcal{M}_1$ but less efficient in $\mathcal{M}_2$, is that adjusting for $\{B,C\}$ reduces both the residual variance and the instrumental strength. The overall effect on the asymptotic variance depends on how large these reductions are, which depends on the strength of the edges $A \rightarrow B$ and $C \leftrightarrow Y$, respectively. Because the edge $A \rightarrow B$ has a small edge coefficient in $\mathcal{M}_2$, adjusting for $\{B,C\}$ reduces the asymptotic variance here.
\label{ex:no optimal}
\end{example}

	\begin{table}[t]
 \begin{center}
		\caption{Asymptotic variances for all linearly valid tuples in the graph from Fig. \ref{subfig:no optimal} in the two linear structural equation models described in Example \ref{ex:no optimal}}
		\begin{tabular}{rrrrrr}
		    & $(A,\emptyset)$ & $(A,B)$ & $(A,\{B,C\})$ & 
		        $(B,\emptyset)$ & $(\{A,B\},\emptyset)$ \\
			$\mathcal{M}_1$ & $3$ & $6$ & $5$ & $6$ & $3$ \\ 
			$\mathcal{M}_2$ & $3$ & $3.12$ & $2.6$ & $78$ & $3$
		\end{tabular}
		\label{table:no optimal}
  \end{center}
	\end{table}
	
Example \ref{ex:no optimal} does not contradict our results from Section \ref{subsec:algorithm}, as these cover the special case of adding a single covariate to a linearly valid tuple at a time. 
Given that there may not be an asymptotically optimal conditional instrumental set, we now introduce the concept of a graphically optimal tuple. It is inspired by the definition of an admissible decision rule in statistical decision theory \citep[Chapter 2,][]{robert2007bayesian}.

\begin{definition}
	\label{definition:graphical optimality}
	Let $X$ and $Y$ be nodes in an acyclic directed mixed graph $\g$ and let $(Z,W)$ be a linearly valid conditional instrumental set relative to $(X,Y)$ in $\g$. We call $(Z,W)$ graphically suboptimal relative to $(X,Y)$ in $\g$, if there exists a linearly valid conditional instrumental set $(Z',W')$ such that for all linear structural equation models compatible with $\g$ in which $\Sigma_{xz.w} \neq 0$ holds, $a.var(\hat{\tau}_{yx}^{z'.w'}) \leq a.var(\hat{\tau}_{yx}^{z.w})$ and for at least one model this inequality is strict. We say that $(Z,W)$ is graphically optimal if it is not graphically suboptimal.
\end{definition}

Graphical optimality is a natural generalization of asymptotic optimality. In particular, there always exists a graphically optimal tuple and if an asymptotically optimal one exists, any graphically optimal tuple is asymptotically optimal. We now give a definition which we use to construct a conditional instrumental set that is linearly valid and graphically optimal under mild conditions. 

\begin{definition}
	Let $N$ be a node and $W$ a node set in an acyclic directed mixed graph $\g$ with node set $V$. Let $\mathrm{dis}_{W}(N,\g) = \{A \in V: A \leftrightarrow V_1 \leftrightarrow \dots\leftrightarrow V_k \leftrightarrow N, V_1\notin W,\dots,V_k \notin W \}$.
	Then let $\mathrm{dis}^+_{W}(N,\g)$ denote the set 
	$\{\mathrm{dis}_{W}(N,\g) \cup \pa(\mathrm{dis}_{W}(N,\g),\g)\}\setminus W$.
	\label{definition:district}
\end{definition}

\begin{theorem}
	Consider nodes $X$ and $Y$ in an acyclic directed mixed graph $\g$ such that $\de(X,\g) = \{X,Y\}$. Let $W^o= \mathrm{dis}^+_{X}(Y,\g) \setminus \{X,Y\}$ and $Z^o=\mathrm{dis}^+_Y(X,\g) \setminus (\{X,Y\} \cup W^o)$. Then the following two statements hold: (i) if $Z^o \neq \emptyset$  then $(Z^o, W^o)$ is a linearly valid conditional instrumental set relative to $(X,Y)$ in $\g$; (ii) if $Z^o \cap \{\pa(X,\g) \cup \mathrm{sib}(X,\g)\} \neq \emptyset$ then $(Z^o, W^o)$ is also graphically optimal relative to $(X,Y)$ in $\g$.
	\label{theorem:optimalCIS}
\end{theorem}

The tuple $(Z^o,W^o)$ is constructed as follows. We first choose $W^o$ in order to minimize the residual variance. With $W^o$ fixed we then in turn choose $Z^o$ from the remaining covariates in order to maximize the conditional instrumental strength. The conditions $Z^o \neq \emptyset$  and $Z^o \cap \{\pa(X,\g) \cup \mathrm{sib}(X,\g)\} \neq \emptyset$ ensure that the resulting $Z^o$ is not too small.
We give examples where the two conditions are violated in Section \ref{sec:weird optimal cis} of the Supplementary Material and emphasize that $Z^o=\emptyset$ may occur in cases where other tuples are linearly valid. For interested readers, we also provide an example in Section \ref{sec:weird optimal cis} where all linearly valid tuples are graphically optimal.

\begin{example}[Illustrating Theorem \ref{theorem:optimalCIS}]
	\label{ex:optimal cis}
	We revisit the graphs from Figure $\ref{subfig:weird}$ and $\ref{subfig:no optimal}$, denoting them $\g$ and $\g'$, respectively.
	In $\g$, $Z^o=\{A,B,D\}$ and $W^o=C$. By Examples \ref{ex:valid CIS} and \ref{ex:weird}, $(\{A,B,D\},C)$ is a linearly valid conditional instrumental set relative to $(X,Y)$ in $\g$ and in fact asymptotically optimal. 
	In $\g'$, $Z^o=A$ and $W^o=\{B,C\}$. By Example \ref{ex:no optimal}, $(A,\{B,C\})$ is a linearly valid conditional instrumental set relative to $(X,Y)$ in $\g'$. There is no asymptotically optimal linearly valid conditional instrumental set relative to $(X,Y)$ in $\g'$. However, as $(A,\{B,C\})$ is more efficient than any other tuple in $\mathcal{M}_2$ it follows that it is graphically optimal.

\end{example}

\section{Simulations}
\label{sec:simulations}

 \begin{figure}[t]	
  \captionsetup{width=10cm}
  		\centering
 		\includegraphics[height=10cm, width=12.5cm]{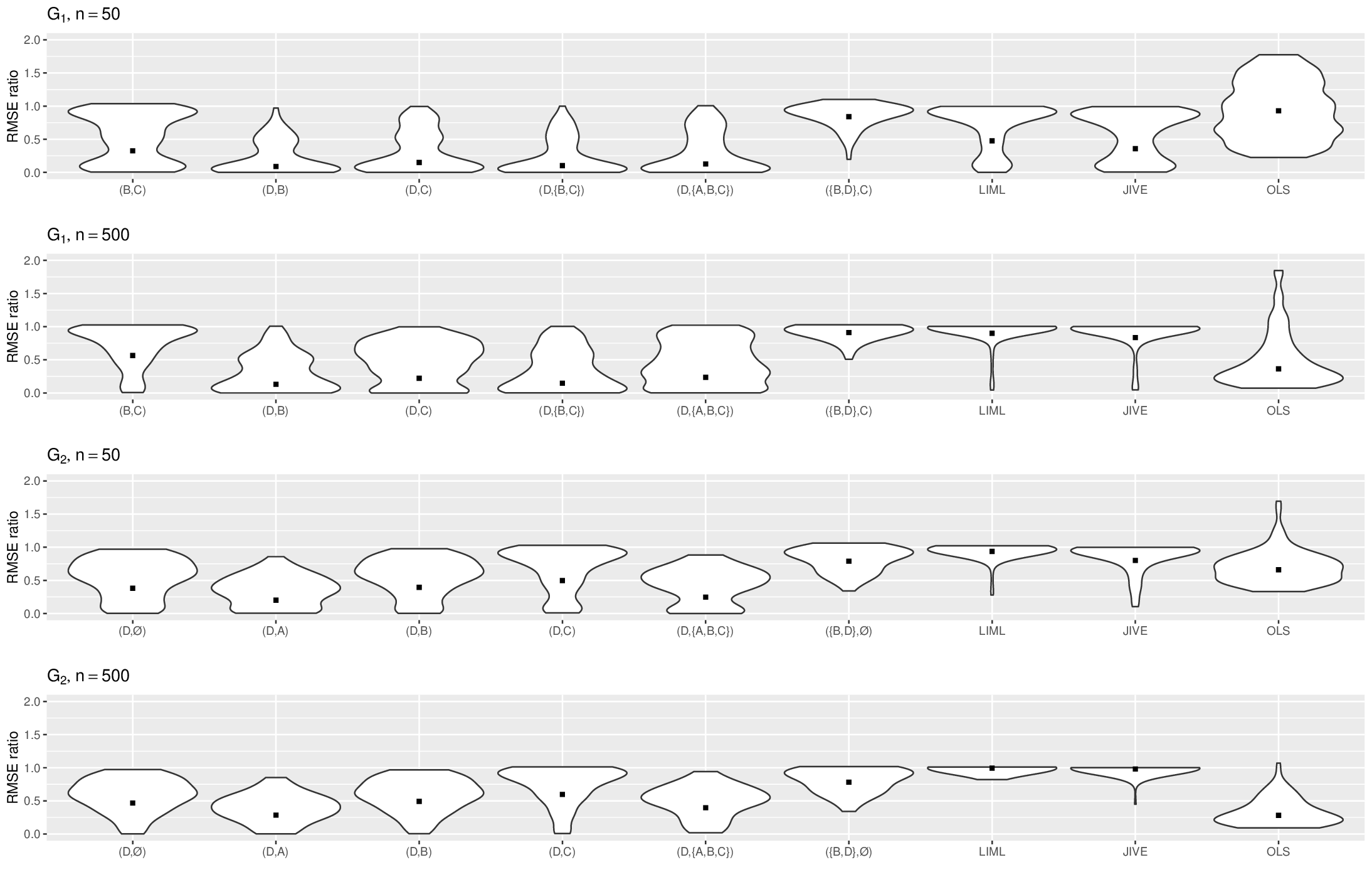}
 	\caption{Violin plot of the root mean squared error ratios for the two-stage least squares estimator with $(Z^o,W^o)$ to the one with the $(Z,W)$ given on the X-axis. We use LIML to denote the limited information maximum likelihood estimator using $(Z^o,W^o)$, JIVE to denote the jackknife instrumental variables estimator using $(Z^o,W^o)$ and OLS to denote the ordinary least squares regression of $Y$ on $X$. 
  The black dots mark the geometric mean of the ratios.}
 	\label{fig:violin}
 \end{figure}

We investigated the finite sample efficiency of $(Z^o,W^o)$ and how it compares to alternative linearly valid conditional instrumental sets. For further comparison, we also considered the following three alternative estimators to the two-stage least squares: (i) the limited information maximum likelihood (LIML) estimator \citep{anderson1949estimation,phillips1977bias}, (ii) the jackknife instrumental variables (JIVE) estimator \citep{angrist1999jackknife} and (iii) the ordinary least squares regression of $Y$ on $X$. The first two are alternative instrumental variable estimators with the same asymptotic distribution as the two-stage least squares estimator but with distinct finite sample characteristics. We include the ordinary least squares regression as a non-causal baseline. We did so in linear structural equation models compatible with the two graphs $\g_1$ and $\g_2$ from Figures \ref{subfig:weird forb} and \ref{subfig:simple}, respectively. Recall that by Example \ref{ex:optimal cis}, the tuple $(Z^o,W^o)$ in $\g_1$ is asymptotically optimal. We also show that the tuple $(Z^o,W^o)$ in $\g_2$ is asymptotically optimal in Example \ref{ex:optimal 2} of the Supplementary Material.

For each graph, we randomly generated $100$ compatible linear structural equation models as follows. We considered models with either all Gaussian or all uniform errors, each with probability $0.5$. For each node we sampled an error variance uniformly on $[0.1,1]$ and for each edge an edge coefficient uniformly on $[-2,-0.1] \cup [0.1,2]$. Any bi-directed edge $V_i \leftrightarrow V_j$ was modelled as a latent variable $V_i \leftarrow L \rightarrow V_j$, with the error variances and edge coefficients generated as for the other nodes and edges.
From each linear structural equation model we generated $100$ data sets with sample size $n$ and used these to compute the two-stage least squares, the limited information maximum likelihood and  the jackknife instrumental variables estimators corresponding to all available linearly valid conditional instrumental sets relative to $(X,Y)$ in the respective graph as well as the ordinary least squares regression of $Y$ on $X$. For each of these estimators, we computed Monte-Carlo root mean squared errors with respect to the known true total effect $\tau_{yx}$ over the $100$ simulated data sets. We did this separately for the sample sizes $n=50$ and $n=500$. 
To compare the performances, we computed the ratio of the root mean squared error for the two-stage least squares with $(Z^o,W^o)$ to the one for each of the alternative linearly valid tuples and estimators. A ratio smaller than $1$ indicates that two-stage least squares with $(Z^o,W^o)$ was more efficient than the other estimator. A ratio larger than $1$ indicates the opposite. Fig. \ref{fig:violin} shows violin plots of these ratios over the $100$ linear structural equation models. As there are $21$ linearly valid conditional instrumental sets relative to $(X,Y)$ in $\g_1$ and $34$ in $\g_2$, we only show the violin plots for a representative subset for the two-stage least squares estimator and only consider the optimal tuple for the alternative instrumental variable estimators. Please see Section \ref{sec:simulations appendix} of the Supplementary Material for further plots.

The violin plots corroborate our theoretical results. The optimal two-stage least squares outperforms all other estimators, with few of the ratios larger than $1$. This is also true for the sample size $n=50$ even though our results only hold asymptotically and the alternative estimators which our results do not cover. Interestingly, the most competitive estimator in $\g_1$ and for $n=50$ is the inconsistent ordinary least squares estimator. This illustrates that using the two-stage least squares estimator rather than the ordinary least squares estimator represents a bias-variance trade-off. Our simulation indicates that it is important to consider statistical efficiency when selecting the conditional instrumental set to ensure that this bias-variance trade-off is favourable in small samples. The fact that the optimal limited information maximum likelihood and the optimal jackknife instrumental variables estimator perform worse than the ordinary least squares for $\g_1$ and $n=50$ indicates that this bias-variance trade-off is not just a characteristic of the two-stage least squares but rather of instrumental variable estimators in general, which further emphasizes the importance of carefully selecting the conditional instrumental set in practice.

\section{Estimating the causal effect of institutions on wealth}

	\begin{table}[t]
 \begin{center}
	\def~{\hphantom{0}}
		\caption{Estimated total effect of institutions on economic wealth and standard errors with varying conditional instrumental sets}
		\begin{tabular}{rrrr}
			$Z$ & $W$ & estimate & standard error\\ 
			$\mathrm{settler \ mort.}$ &  $\emptyset$ & 0.94 & 0.16 \\ 
			$\mathrm{settler \ mort.}$  & $\mathrm{latitude}$ & 1.00 & 0.22 \\ 
			$\mathrm{settler \ mort.}$  & $\mathrm{euro.\ anc.}$ & 0.96 & 0.28 \\ 
			$\mathrm{settler \ mort.}$ & $\mathrm{ethnoling. \ frag.}$ & 0.74 & 0.13 \\ 
			$\{\mathrm{settler \ mort., euro. \ anc.}\}$  & $\emptyset$ & 0.94 & 0.14 \\ 
			$\{\mathrm{settler \ mort., euro. \ anc.}\}$  & $\mathrm{ethnoling. \ frag.}$ & 0.74 & 0.11 \\ 
		\end{tabular}
		\label{table:real data}
 \end{center}
	\end{table}

We revisit the analysis by \citet{acemoglu2001colonial} of the effect of institutions on economic wealth with the mortality of European settlers as an instrument. In their analysis \citet{acemoglu2001colonial} consider a variety of different conditioning sets. They diligently check that all of them yield a coherent estimate of the total effect of institutions on economic wealth. This serves as a robustness check for their analysis. The authors do not, however, consider how the standard errors differ depending on the conditioning set. We investigate this for three of the available covariates: latitude, ethnolinguistic fragmentation, and percentage of population of European ancestry. We summarize our results in Table~\ref{table:real data}.
Even though we do not have a causal graph, we can follow arguments by \citeauthor{acemoglu2001colonial} to approximately check the conditions of Theorem~\ref{theorem:CIVcomparison}.

According to the authors, latitude and percentage of population with European ancestry are correlated with settler mortality, because (i) tropical diseases were a major cause of settler mortality and (ii) low settler mortality led to a larger settler population. Accepting this, we classify both as covariates of the same type as variable $A$ in Figure~\ref{subfig:simple}. By the argument of harmful conditioning from Example~\ref{ex:harmful} we expect that conditioning on them leads to a larger standard error and we indeed observe this in our calculations. 
Adding percentage of population with European ancestry to the instrumental set, on the other hand, 
leads to a smaller standard error, as expected per Corollary~\ref{corollary:always instrument}.
Similarly, the variable ethnolinguistic fragmentation belongs to the same class of covariates as variable $C$ in Figure~\ref{subfig:simple}. This is because cultural barriers are often also market barriers and as a result, ethnolinguistic fragmentation is predictive of economic wealth. This time, by the argument of beneficial conditioning from Example \ref{ex:beneficial}, adjusting for ethnolinguistic fragmentation should lead to a smaller standard error and we indeed observe this in our calculations. 

In our analysis the smallest standard error (0.11) is less than half the size of the largest (0.28), illustrating the potential advantages of applying our results to select the best linearly valid tuple. We emphasize that identifying an efficient tuple by computing the standard errors for many tuples and selecting the one with the smallest standard error leads to invalid standard errors due to post-selection inference \citep{berk2013valid}. Using qualitative causal background knowledge, ideally in form of a graph, to apply our results does not suffer from this drawback.

\section{Discussion}
There are many instrumental variable estimators other than the two-stage least squares estimator, such as the limited information maximum likelihood and the jackknife instrumental variables estimator we considered in our simulation study, which were developed for settings with many weak instruments. It would be interesting to try to generalize our results to a broad class of alternative instrumental variable estimators and beyond the linear setting with a fixed number of instruments.

Finally, we would like to point out four other interesting avenues for future research: First, to graphically characterize when an asymptotically optimal linearly valid conditional instrumental set exists \citep[cf.][]{runge2021necessary}. Second, to graphically characterize a linearly valid conditional instrumental set that is graphically optimal and has maximal conditional instrumental strength. As the finite sample efficiency of our estimators appears to be more vulnerable to a small conditional instrumental strength than a large residual variance (see Section \ref{sec:simulations}), such a tuple might be a good alternative to $(Z^o,W^o)$, especially in small samples. Third, to generalize Theorem \ref{theorem:CIVcomparison} such that it covers all cases where we can use the graph to decide which of two linearly valid conditional instrumental sets provides the smaller asymptotic variance, that is, derive a necessary and sufficient graphical criterion. Fourth, to generalize the paper's results and in particular Theorem \ref{theorem:graphical valid CIS} to the setting where $X$ and $Y$ are sets.

\section*{Acknowledgement}
We thank Nicola Gnecco and Jonas Peters for valuable comments.

\bibliography{paper-ref}%

\begin{thebibliography}{}

\bibitem[Acemoglu et~al., 2001]{acemoglu2001colonial}
Acemoglu, D., Johnson, S., and Robinson, J.~A. (2001).
\newblock The colonial origins of comparative development: An empirical
  investigation.
\newblock {\em American Economic Review}, 91(5):1369--1401.

\bibitem[Anderson et~al., 1949]{anderson1949estimation}
Anderson, T.~W., Rubin, H., et~al. (1949).
\newblock Estimation of the parameters of a single equation in a complete
  system of stochastic equations.
\newblock {\em Ann. Math. Statistics}, 20(1):46--63.

\bibitem[Angrist et~al., 1999]{angrist1999jackknife}
Angrist, J.~D., Imbens, G.~W., and Krueger, A.~B. (1999).
\newblock Jackknife instrumental variables estimation.
\newblock {\em J. Appl. Econometrics}, 14(1):57--67.

\bibitem[Angrist et~al., 1996]{angrist1996identification}
Angrist, J.~D., Imbens, G.~W., and Rubin, D.~B. (1996).
\newblock Identification of causal effects using instrumental variables.
\newblock {\em J. Amer. Statist. Assoc.}, 91(434):444--455.

\bibitem[Basmann, 1957]{basmann1957generalized}
Basmann, R.~L. (1957).
\newblock A generalized classical method of linear estimation of coefficients
  in a structural equation.
\newblock {\em Econometrica}, 25:77--83.

\bibitem[Bekker, 1994]{bekker1994alternative}
Bekker, P.~A. (1994).
\newblock Alternative approximations to the distributions of instrumental
  variable estimators.
\newblock {\em Econometrica}, 62:657--681.

\bibitem[Berk et~al., 2013]{berk2013valid}
Berk, R., Brown, L., Buja, A., Zhang, K., and Zhao, L. (2013).
\newblock Valid post-selection inference.
\newblock {\em Ann. Statist.}, 41:802--837.

\bibitem[Bound et~al., 1995]{bound1995problems}
Bound, J., Jaeger, D.~A., and Baker, R.~M. (1995).
\newblock Problems with instrumental variables estimation when the correlation
  between the instruments and the endogenous explanatory variable is weak.
\newblock {\em J. Amer. Statist. Assoc.}, 90(430):443--450.

\bibitem[Bowden and Turkington, 1990]{bowden1990instrumental}
Bowden, R.~J. and Turkington, D.~A. (1990).
\newblock {\em Instrumental variables}.
\newblock Cambridge University Press, Cambridge.

\bibitem[Brito and Pearl, 2002a]{brito2002generalized}
Brito, C. and Pearl, J. (2002a).
\newblock Generalized instrumental variables.
\newblock In {\em Proceedings of the Eighteenth Annual Conference on
  Uncertainty in Artificial Intelligence (UAI-02)}, pages 85--93, San
  Francisco, CA. Morgan Kaufmann.

\bibitem[Brito and Pearl, 2002b]{brito2002new}
Brito, C. and Pearl, J. (2002b).
\newblock A new identification condition for recursive models with correlated
  errors.
\newblock {\em Struct. Equ. Model.}, 9(4):459--474.

\bibitem[Buja et~al., 2014]{buja2014models}
Buja, A., Berk, R., Brown, L., George, E., Pitkin, E., Traskin, M., Zhan, K.,
  and Zhao, L. (2014).
\newblock Models as approximations, part {I}: A conspiracy of nonlinearity and
  random regressors in linear regression.
\newblock {\em arXiv:1404.1578}.

\bibitem[Guo et~al., 2022]{guo2022variable}
Guo, F.~R., Perkovi{\'c}, E., and Rotnitzky, A. (2022).
\newblock Variable elimination, graph reduction and efficient g-formula.
\newblock {\em arXiv:2202.11994}.

\bibitem[Hansen, 2022]{hansen2019}
Hansen, B.~E. (2022).
\newblock {\em Econometrics}.
\newblock Princeton University Press.

\bibitem[Henckel et~al., 2022]{henckel2019graphical}
Henckel, L., Perković, E., and Maathuis, M.~H. (2022).
\newblock Graphical criteria for efficient total effect estimation via
  adjustment in causal linear models.
\newblock {\em J. R. Statist. Soc. B}, 84(2):579--599.

\bibitem[Hern{\'a}n and Robins, 2006]{hernan2006instruments}
Hern{\'a}n, M.~A. and Robins, J.~M. (2006).
\newblock Instruments for causal inference: an epidemiologist's dream?
\newblock {\em Epidemiology}, 17:360--372.

\bibitem[Kinal, 1980]{kinal1980existence}
Kinal, T.~W. (1980).
\newblock The existence of moments of k-class estimators.
\newblock {\em Econometrica}, 48:241--249.

\bibitem[Koster, 1999]{koster1999validity}
Koster, J.~T. (1999).
\newblock On the validity of the markov interpretation of path diagrams of
  gaussian structural equations systems with correlated errors.
\newblock {\em Scand. J. Statist.}, 26(3):413--431.

\bibitem[Kuroki and Cai, 2004]{kuroki2004selection}
Kuroki, M. and Cai, Z. (2004).
\newblock Selection of identifiability criteria for total effects by using path
  diagrams.
\newblock In {\em Proceedings of the Twentieth Annual Conference on Uncertainty
  in Artificial Intelligence (UAI-04)}, pages 333--340, Arlington, Virginia.
  AUAI Press.

\bibitem[Kuroki and Miyakawa, 2003]{kuroki2003covariate}
Kuroki, M. and Miyakawa, M. (2003).
\newblock Covariate selection for estimating the causal effect of control plans
  by using causal diagrams.
\newblock {\em J. R. Statist. Soc. B}, 65(1):209--222.

\bibitem[Pearl, 1995]{pearl1995causal}
Pearl, J. (1995).
\newblock Causal diagrams for empirical research.
\newblock {\em Biometrika}, 82(4):669--688.

\bibitem[Pearl, 2009]{pearl2009causality}
Pearl, J. (2009).
\newblock {\em Causality}.
\newblock Cambridge University Press, Cambridge, second edition.

\bibitem[Perkovi{\'c} et~al., 2018]{perkovic16}
Perkovi{\'c}, E., Textor, J., Kalisch, M., and Maathuis, M.~H. (2018).
\newblock Complete graphical characterization and construction of adjustment
  sets in {M}arkov equivalence classes of ancestral graphs.
\newblock {\em J. Mach. Learn. Res.}, 18(220):1--62.

\bibitem[Phillips and Hale, 1977]{phillips1977bias}
Phillips, G. D.~A. and Hale, C. (1977).
\newblock The bias of instrumental variable estimators of simultaneous equation
  systems.
\newblock {\em Internat. Econom. Rev.}, 18(1):219--228.

\bibitem[Richardson, 2003]{richardson2003markov}
Richardson, T. (2003).
\newblock Markov properties for acyclic directed mixed graphs.
\newblock {\em Scand. J. Statist.}, 30(1):145--157.

\bibitem[Richardson et~al., 2017]{richardson2017nested}
Richardson, T.~S., Evans, R.~J., Robins, J.~M., and Shpitser, I. (2017).
\newblock Nested markov properties for acyclic directed mixed graphs.
\newblock {\em arXiv:1701.06686}.

\bibitem[Robert et~al., 2007]{robert2007bayesian}
Robert, C.~P. et~al. (2007).
\newblock {\em The Bayesian choice: from decision-theoretic foundations to
  computational implementation}, volume~2.
\newblock Springer.

\bibitem[Rotnitzky and Smucler, 2020]{rotnitzky2019efficient}
Rotnitzky, A. and Smucler, E. (2020).
\newblock Efficient adjustment sets for population average causal treatment
  effect estimation in graphical models.
\newblock {\em J. Mach. Learn. Res.}, 21(188):1--86.

\bibitem[Rotnitzky and Smucler, 2022]{rotnitzky2022minimum}
Rotnitzky, A. and Smucler, E. (2022).
\newblock A note on efficient minimum cost adjustment sets in causal graphical
  models.
\newblock {\em J. Causal Inference}, 10(1):174--189.

\bibitem[Runge, 2021]{runge2021necessary}
Runge, J. (2021).
\newblock Necessary and sufficient graphical conditions for optimal adjustment
  sets in causal graphical models with hidden variables.
\newblock {\em Advances in Neural Information Processing Systems},
  34:15762--15773.

\bibitem[Shpitser et~al., 2010]{shpitser2010validity}
Shpitser, I., VanderWeele, T., and Robins, J. (2010).
\newblock On the validity of covariate adjustment for estimating causal
  effects.
\newblock In {\em Proceedings of the Twenty-Sixth Annual Conference on
  Uncertainty in Artificial Intelligence (UAI-10)}, pages 527--536, Corvallis,
  Oregon. AUAI Press.

\bibitem[Smucler et~al., 2022]{smucler2020efficient}
Smucler, E., Sapienza, F., and Rotnitzky, A. (2022).
\newblock Efficient adjustment sets in causal graphical models with hidden
  variables.
\newblock {\em Biometrika}, 109(1):49--65.

\bibitem[Spirtes et~al., 2000]{spirtes2000causation}
Spirtes, P., Glymour, C., and Scheines, R. (2000).
\newblock {\em Causation, Prediction, and Search}.
\newblock MIT Press, Cambridge, MA, second edition.

\bibitem[Staiger and Stock, 1997]{staiger1997instrumental}
Staiger, D. and Stock, J.~H. (1997).
\newblock Instrumental variables regression with weak instruments.
\newblock {\em Econometrica}, 65:557--586.

\bibitem[Stock et~al., 2002]{stock2002survey}
Stock, J.~H., Wright, J.~H., and Yogo, M. (2002).
\newblock A survey of weak instruments and weak identification in generalized
  method of moments.
\newblock {\em J. Bus. Econom. Statist.}, 20(4):518--529.

\bibitem[Vansteelandt and Didelez, 2018]{vansteelandt2018improving}
Vansteelandt, S. and Didelez, V. (2018).
\newblock Improving the robustness and efficiency of covariate-adjusted linear
  instrumental variable estimators.
\newblock {\em Scand. J. Statist.}, 45(4):941--961.

\bibitem[Wermuth, 1989]{wermuth1989moderating}
Wermuth, N. (1989).
\newblock Moderating effects in multivariate normal distributions.
\newblock {\em Methodika}, 3:74--93.

\bibitem[Witte et~al., 2020]{witte2020efficient}
Witte, J., Henckel, L., Maathuis, M.~H., and Didelez, V. (2020).
\newblock On efficient adjustment in causal graphs.
\newblock {\em J. Mach. Learn. Res.}, 21(246):1--45.

\bibitem[Wooldridge, 2010]{wooldridge2010econometric}
Wooldridge, J.~M. (2010).
\newblock {\em Econometric analysis of cross section and panel data}.
\newblock MIT press.

\bibitem[Wright, 1934]{wright1934method}
Wright, S. (1934).
\newblock The method of path coefficients.
\newblock {\em Ann. Math. Statistics}, 5(3):161--215.

\end{thebibliography}
\bibliographystyle{apalike}%

\newpage

\appendix

\section{Preliminaries and known results}
	
	\subsection{General preliminaries} 
	
		\vsp\noindent \emph{Covariance matrices and regression coefficients:}
Consider random vectors $S=(S_1,\dots,S_{k_s}),T=(T_1,\dots,T_{k_t})$ and $W$.
We denote the covariance matrix of $S$ by $\Sigma_{ss} \in\mathbb{R}^{k_s \times k_s}$ and the covariance matrix between $S$ and $T$ by $\Sigma_{st}  \in \mathbb{R}^{k_s \times k_t}$, where its $(i,j)$-th element equals $\mathrm{cov}(S_i,T_j)$. We further define 
$\Sigma_{st.w} = \Sigma_{st} - \Sigma_{sw} \Sigma^{-1}_{ww} \Sigma_{wt}$. If $k_s=k_t=1$, we write $\sigma_{st.w}$ instead of $\Sigma_{st.w}$. The value $\sigma_{ss.w}$ can be interpreted as the residual variance of the ordinary least squares regression of $S$ on $W$. We also refer to $\sigma_{ss.w}$ as the conditional variance of $S$ given $W$.  
Let $\beta_{st.w} \in \mathbb{R}^{k_s \times k_t}$ represent the population level least squares coefficient matrix whose $(i,j)$-th element is the regression coefficient of $T_j$ in the regression of $S_i$ on $T$ and $W$. We denote the corresponding estimator as $\hat{\beta}_{st.w}$. Finally, for random vectors $W_1, \dots, W_m$ with $W=(W_1, \dots, W_m)$ we use the notation that $\beta_{st.{w_1\cdots w_m}} = \beta_{st.w}$ and $\Sigma_{st.w_1 \cdots w_m} = \Sigma_{st.w}$.

\vsp\noindent \emph{Two stage least squares estimator:} \citep{basmann1957generalized}
Consider two random variables $X$ and $Y$, and two random vectors $Z$ and $W$. Let $S_{n},T_n$ and $Y_n$ be the random matrices corresponding to taking $n$ i.i.d. observations from the random vectors $S=(X,W)$, $T=(Z,W)$ and $Y$, respectively. Then the two stage least squares estimator $\hat{\tau}_{yx}^{z.w}$ is defined as the first entry of the larger estimator
\[ 
\hat{\gamma}_{ys.t}= Y_n T_n^\top (T_n T_n^\top)^{-1} T_n S_n^\top
\{S_n T_n^\top (T_n T_n^\top)^{-1} T_n S_n^\top \}^{-1},
\]
where we repress the dependence on the sample size $n$ for simplicity.
We refer to the tuple $(Z,W)$ as the \emph{conditional instrumental set}, to $Z$ as the instrumental set and to $W$ as the conditioning set. We also let $\gamma_{ys.t}=\Sigma_{yt} \Sigma_{tt}^{-1} \Sigma_{ts} (\Sigma_{st} \Sigma_{tt}^{-1} \Sigma_{st}^\top)^{-1}$ denote the population level two-stage least squares estimator of $Y$ on $S$ with instrumental set $T$, which exists whenver $\Sigma_{st} \Sigma_{tt}^{-1} \Sigma_{st}^\top$ is invertible.
	
	\subsection{Graphical and causal preliminaries}
	\label{sec:prelim:graph}
	
	\vsp\noindent \emph{Graphs:} 
	We consider graphs $\g=(V,E)$ with node set $V$ and edge set $E$, where
	edges can be either directed ($\rightarrow$) or bidirected ($\leftrightarrow$).  
	If all edges in $E$ are directed, then $\g$ is a \emph{directed graph}. If all edges in $E$ are directed or bidirected, then $\g$ is a directed mixed graph. 
	
	\vsp \noindent \emph{Paths:} 
	Two edges are \emph{adjacent} if they have a common node. A walk is a sequence of adjacent edges. A path $p$ is a sequence of adjacent edges without repetition of a node and may consist of just a single node. The first node $X$ and the last node $Y$ on a path $p$ are called \emph{endpoints} of $p$ and we say that $p$ is a path from $X$ to $Y$. Given two nodes $Z$ and $W$ on a path $p$, we use $p(Z, W)$ to denote the subpath of $p$ from $Z$ to $W$. A path from a set of nodes $S$ to a set of nodes $T$ is a path from a node $X \in S$ to some node $Y \in T$. A path from a set $S$ to a set $T$ is proper if only the first node is in $S$ \citep[cf.][]{shpitser2010validity}. 
	A path $p$ is called \emph{directed} or \emph{causal} from $X$ to $Y$ if all edges on $p$ are directed and point towards $Y$. We use $\oplus$ to denote the concatenation of paths. For example, for any path $p$ from $X$ to $Y$ with intermediary node $Z,p=p(X,Z) \oplus p(Z,Y)$.

	\vsp\noindent \emph{Ancestry:} If $X \to Y$, then $X$ is a parent of $Y$ and $Y$ is a child of $X$. 
	If there is a causal path from $X$ to $Y$, then $X$ is an ancestor of $Y$ and $Y$ a descendant of $X$. If $X \leftrightarrow Y$, then $X$ and $Y$ are siblings. 
	We use the convention that every node is an ancestor, descendant and sibling of itself. 
	The sets of parents, ancestors, descendants and siblings of $X$ in $\g$ are denoted by $\pa(X,\g)$, $\an(X,\g)$, $\de(X,\g)$ and $\mathrm{sib}(X,\g)$, respectively. 
	For sets $S$, we let $\pa(S,\g)=\bigcup_{X \in S} \pa(X,\g)$, with analogous definitions for $\an(S,\g)$, $\de(S,\g)$ and $\mathrm{sib}(S,\g)$.

	\vsp\noindent \emph{Colliders:} A node $V_j$ is a collider on a path $V_1 \cdots V_j \cdots V_m$ if $p$ contains a subpath of the form $V_{j-1} \rightarrow V_j \leftarrow V_{j+1}, V_{j-1} \rightarrow V_j \leftrightarrow V_{j+1}, V_{j-1} \leftrightarrow V_j \leftarrow V_{j+1}$ or $V_{j-1} \leftrightarrow V_j \leftrightarrow V_{j+1}$. A node $V$ on a path $p$ is called a \emph{non-collider} on $p$ if it is neither a collider on $p$ nor an endpoint node of $p$.

	\vsp\noindent \emph{Directed cycles, directed acyclic graphs and acyclic directed mixed graphs:} A directed path from a node $X$ to a node $Y$, together with the edge $Y\to X$ forms a directed cycle. A directed graph without directed cycles is called a directed acyclic graph and a directed mixed graph without directed cycles is called a acyclic directed mixed graph.
	
	\vsp\noindent \emph{Blocking, d-separation and m-separation:} (Definition 1.2.3 in \cite{pearl2009causality} and Section 2.1 in \cite{richardson2003markov}). 
	Let $S$ be a set of nodes in an acyclic directed mixed graph $\g$. A path $p$ in $\g$ is blocked by $S$ if (i) $p$ contains a non-collider that is in $S$, or (ii) $p$ contains a collider $C$ such that no descendant of $C$ is in $S$. A path that is not blocked by a set $S$ is open given $S$. If $S,T$ and $W$ are three pairwise disjoint sets of nodes in $\g$, then $W$ m-separates $S$ from $T$ in $\g$ if $W$ blocks every path between $S$ and $T$ in $\g$. We then write $S \perp_{\g} T\mid W$. Otherwise, we write $S \not\perp_{\g} T\mid W$. 
	We use the convention that for any two disjoint node sets $S$ and $T$ it holds that $\emptyset \perp_{\g} S \mid T$.

 Alternatively and equivalently, we can define m-separation using walks instead of paths as follows. A walk $w$ is blocked by $S$ if (i) $w$ contains a non-collider that is in $S$, or (ii) $p$ contains a collider $C$ such that $C$ is not in $S$. A walk that is not blocked by a set $S$ is open given $S$. Further, $W$ m-separates $S$ from $T$ in $\g$ if $W$ blocks every walk between $S$ and $T$ in $\g$.

	\vsp\noindent \emph{Markov property and faithfulness:} \cite[Definition 1.2.2][]{pearl2009causality} Let $S$, $T$ and $W$ be disjoint sets of random variables. We use the notation $S \ci T\mid W$ to denote that $S$ is conditionally independent of $T$ given $W$. A density $f$ is called \emph{Markov} with respect to an acyclic directed mixed graph $\g$ if $S \perp_{\g} T \mid W$ implies $S \ci T \mid W$ in $f$. If this implication holds in the other direction, then $f$ is \emph{faithful} with respect to $\g$.
	
		\vsp\noindent \emph{Causal paths and forbidden nodes:} \citep[][]{perkovic16} 
Let $X$ and $Y$ be nodes in an acyclic directed mixed graph $\g$. 
A path from $X$ to $Y$ in $\g$ is called a causal path from $X$ to $Y$ if all edges on $p$ are directed and point towards $Y$. 
We define the causal nodes with respect to $(X,Y)$ in $\g$ as all nodes on causal paths from $X$ to $Y$ excluding $X$ and denote them $\cn{\g}$. 
We define the forbidden nodes relative to $(X,Y)$ in $\g$ as $\f{\g} = \de(\cn{\g}, \g) \cup X$.

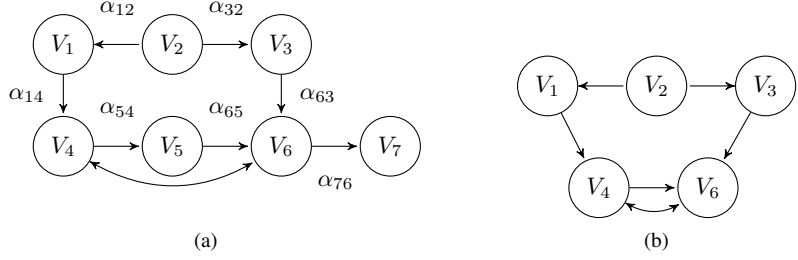
\begin{figure}[t]
	\centering
 	\subfloat[\label{subfig:CIS}]{
		\begin{tikzpicture}[>=stealth',shorten >=1pt,auto,node distance=0.7cm,scale=.9, transform shape,align=center,minimum size=3em]
		\node[state] (x) at (0,0) {$V_4$};
		\node[state] (m) [right =of x]   {$V_5$}; 
		\node[state] (y) [right =of m] {$V_6$};
		\node[state] (w1) at ($(x)+(.0,1.5)$) {$V_1$};  
		\node[state] (w2) at ($(m)+(0,+1.5)$)   {$V_2$};  
		\node[state] (w3) at ($(y)+(0,+1.5)$)   {$V_3$};  
		\node[state] (f2) [right =of y]   {$V_7$};  
		\path[->]   (x) edge node[above] {$\alpha_{54}$} (m);
		\path[->]   (m) edge node[above] {$\alpha_{65}$} (y);
		\path[<->]   (x) edge  [bend right] (y);
		\path[<-]   (w1) edge node[above] {$\alpha_{12}$} (w2);
		\path[->]   (w1) edge node[left] {$\alpha_{14}$} (x);
		\path[->]   (w2) edge node[above] {$\alpha_{32}$}  (w3);
		\path[->]   (w3) edge node[right] {$\alpha_{63}$} (y);
		\path[->]   (y) edge  node[below] {$\alpha_{76}$} (f2);
				\end{tikzpicture}
	} 
	\hspace{1cm}
	\subfloat[\label{subfig:latent}]{
		\begin{tikzpicture}[>=stealth',shorten >=1pt,auto,node distance=0.7cm,scale=.9, transform shape,align=center,minimum size=3em]
		\node[state] (x) at (0,0) {$V_4$};
		\node[state] (y) [right =of x] {$V_6$};
		\node[state] (w1) at ($(x)+(-.75,+1.5)$)   {$V_1$}; 
		\node[state] (w2) [right =of w1]    {$V_2$};  
		\node[state] (w3) [right =of w2]   {$V_3$};  
		\path[->]   (x) edge    (y);
		\path[<->]   (x) edge  [bend right]  (y);
		\path[->]   (w3) edge     (y);
		\path[->]   (w1) edge     (x);
		\path[<-]   (w1) edge     (w2);
		\path[<-]   (w3) edge     (w2);
		\end{tikzpicture}
	} 
	\caption{
	Acyclic directed mixed graphs for Example \ref{ex:prelim}.
}
	\label{fig:CIS example}
\end{figure}

\vsp\noindent \emph{Linear structural equation model:}
Consider an acyclic directed mixed graph $\g=(V,E)$, with nodes $V=(V_1,\dots,V_p)$ and edges $E$, where the nodes represent random variables. The random vector $V$ is generated from a linear structural equation model compatible with $\g$ if 
\begin{equation}
V \leftarrow \mathcal{A} V + \epsilon, \label{app:sem}
\end{equation}
such that the following three properties hold: First, $\mathcal{A}=(\alpha_{ij})$ is a matrix with $\alpha_{ij}=0$ for all $i,j $ where $V_j \rightarrow V_i \notin E$. Second, $\epsilon=(\epsilon_{v_1},\dots,\epsilon_{v_p})$ is a random vector of errors such that $E(\epsilon)=0$ and $\mathrm{cov}(\epsilon)=\Omega=(\omega_{ij})$ is a matrix with $\omega_{ij}=\omega_{ji}=0$ for all $i,j$ where $V_i \leftrightarrow V_j \notin E$. Third, for any two disjoint sets $V', V'' \subseteq V$ such that for all $V_i\in V'$ and all $V_j \in V''$, $V_i \leftrightarrow V_j \notin E$, the random vector $(\epsilon_{v_i})_{v_i \in v'}$ is independent of $ (\epsilon_{v_i})_{v_i \in v''}$. 

Given the matrices $\mathcal{A}$ and $\Omega$ it holds that 
\[
\Sigma_{vv}=(\mathrm{Id}-\mathcal{A})^{-\top} \Omega (\mathrm{Id}-\mathcal{A})^{-1}, 
\]
that is, the covariance matrix of a linear causal model is completely determined by the tuple of matrices $(\mathcal{A},\Omega)$.

\vsp\noindent \emph{Gaussian linear structural equation model:}
If the errors in a linear structural equation model are jointly normal we refer to the linear structural equation model as a Gaussian linear structural equation model. 
Such a model is completely determined by the tuple of matrices $(\mathcal{A},\Omega)$ which is why we will use $(\mathcal{A},\Omega)$ to denote the corresponding Gaussian linear structural equation model.

\vsp\noindent \emph{Latent projections and causal acyclic directed mixed graphs:} \citep[][]{richardson2017nested}
Let $\g$ be an acyclic directed mixed graph with node set $V$ and let $L \subset V$. We can use a tool called the latent projection \citep{richardson2003markov} to remove the nodes in $L$ from $\g$ while preserving all m-separation statements between subsets of $V \setminus L$. We use the notation $\g^{L}$ to denote the acyclic directed mixed graph with node set $V\setminus L$ that is the latent projection of $\g$ over $L$. The latent projection $\g^{L}$ of $\g$ over $L$ is an acyclic directed mixed graph with node set $V\setminus L$ and edges in accordance with the following rules: First, $\g^{L}$ contains a directed edge $W_i \rightarrow W_j$ if and only if there exists a causal path $W_i \rightarrow \dots \rightarrow W_j$ in $\g$ with all non-endpoint nodes in $L$. Second, $\g^{L}$ contains a bi-directed edge $W_i \leftrightarrow W_j$ if and only if there exists a path between $W_i$ and $W_j$, such that all non-endpoint nodes are in $L$ and non-colliders on $p$, and the edges adjacent to $W_i$ and $W_j$ have arrowheads pointing towards $W_i$ and $W_j$, respectively.

The following is an important property of the latent projection for acyclic directed mixed graphs: if $V$ is generated from a linear structural equation model compatible with the acyclic directed mixed graph $\g$ than $V\setminus L$ is generated from a linear structural equation model compatible with $\g^{L}$.

\vsp\noindent \emph{Total effects:}
We define the total effect of $X$ on $Y$ as 
	\[
	\tau_{yx}(x) = \frac{\partial}{\partial x} E\{Y\mid do(X=x)\}.
	\]
In a linear structural equation model the function $\tau_{yx}(x)$ is constant, which is why we simply write $\tau_{yx}$.
The path tracing rules by \citep{wright1934method} allow for the following alternative definition. Consider two nodes $X$ and $Y$ in an acyclic directed mixed graph $\g=(V,E)$ and suppose that $V$ is generated from a linear structural equation model compatible with $\g$. Then $\tau_{yx}$ is the sum over all causal paths from $X$ to $Y$ of the product of the edge coefficients along each such path. 

Given a set of random variables $W=\{W_1,\dots,W_k\}$ we also define $\tau_{wx}$ to be the stacked column vector with $i$th entry $\tau_{w_ix}$. Finally, let $\tau_{w_ix.w_{-i}}$ be the sum over all causal paths from $X$ to $W_i$ that do not contain nodes in $W_{-i}$ of the product of the edge coefficients along each such path.

\vsp\noindent \emph{Valid adjustment sets:} \citep{shpitser2010validity,perkovic16} 
We refer to $W$ as a valid adjustment set relative to $(X,Y)$ in $\g$, if $\beta_{yx.w}=\tau_{yx}$ for all linear structural equation models compatible with $\g$. The set $W$ is a valid adjustment set relative to $(X,Y)$ in $\g$ if and only if (i) $W \cap\f{\g}=\emptyset$ and (ii) $W$ blocks all non-causal paths from $X$ to $Y$.

\begin{example}[Linear structural equation models and total effects]
	\label{ex:prelim}
	Consider the graph $\g$ from Fig. \ref{subfig:CIS}. Then the following generating mechanism is an example of a linear structural equation model compatible with $\g$:
 \begin{align*}
 &V_1 \leftarrow \alpha_{12}V_2 + \epsilon_{v_1},\\
&V_2\leftarrow \epsilon_{v_2},\\
&V_3 \leftarrow \alpha_{32}V_2 + \epsilon_{v_3},\\
&V_4 \leftarrow \alpha_{41}V_1 + \epsilon_{v_4},\\
&V_5 \leftarrow \alpha_{54}V_4 +\epsilon_{v_5},\\
&V_6 \leftarrow \alpha_{63}V_3 + \alpha_{65}V_5 +\epsilon_{v_6},\\
&V_7 \leftarrow \alpha_{76}V_6+\epsilon_{v_7}
 \end{align*}
	with $(\epsilon_{v_1},\dots, \epsilon_{v_7}) \sim \mathcal{N}(0,\Omega)$, where $\Omega=(\omega_{ij})$ and $\omega_{ij}=0$ if $i\neq j$, $i\neq 4$ and $j\neq 6$, or $i\neq j, i\neq 6$ and $j\neq 4$.
	Suppose we are interested in the total effect of $V_4$ on $V_6$. The path $V_4 \rightarrow V_5 \rightarrow V_6$ is the only causal path between the two nodes and therefore $\tau_{v_6v_4}=\alpha_{54}\alpha_{65}$. The causal nodes are $\mathrm{cn}(V_4,V_6,\g)=\{V_5,V_6\}$ and the forbidden nodes are $\mathrm{forb}(V_4,V_6,\g)=\{V_4,V_5,V_6,V_7\}$.

 Since $\mathrm{forb}(V_4,V_6,\g)=\{V_4,V_5,V_6,V_7\}$, any valid adjustment set $W$ relative to $(X,Y)$ in $\g$ has to consist of nodes in $\{V_1,V_2,V_3\}$. The two non-causal paths from $V_4$ to $V_6$ in $\g$ are $V_4 \leftarrow V_1 \leftarrow V_2 \rightarrow V_3 \rightarrow V_6$ and $V_4 \leftrightarrow V_6$. The former path is blocked by $W$ if $W \subseteq \{V_1,V_2,V_3\}$ is non-empty, since all three nodes are non-colliders on the path. The latter path cannot be blocked by $W$ and as a result there is no valid adjustment set relative to $(V_4,V_6)$ in $\g$. If we remove the edge $V_4 \leftrightarrow V_6$ from $\g$, then any $W\subseteq \{V_1,V_2,V_3\}$ would be a valid adjustment set.

 Let $F=\mathrm{forb}(V_4,V_6,\g) \setminus \{V_4,V_6\}=\{V_5,V_7\}$. The acyclic directed mixed graph in Fig. \ref{subfig:latent} is the latent projection graph $\g^{F}$. Replacing the variables in $F$ with their generating equation we obtain the linear structural equation model
 \begin{align*}
 &V_1 \leftarrow \alpha_{12}V_2 + \epsilon_{v_1},\\
&V_2\leftarrow \epsilon_{v_2},\\
&V_3 \leftarrow \alpha_{32}V_2 + \epsilon_{v_3},\\
&V_4 \leftarrow \alpha_{41}V_1 + \epsilon_{v_4},\\
&V_6 \leftarrow \alpha_{64}V_3 + \alpha_{65}(\alpha_{54}V_4 +\epsilon_{v_5}) +\epsilon_{v_6},\\
 \end{align*}
	with $(\epsilon_{v_1},\dots, \epsilon_{v_6}) \sim \mathcal{N}(0,\Omega)$, where $\Omega=(\omega_{ij})$ and $\omega_{ij}=0$ if $i\neq j$, $i\neq 4$ and $j\neq 6$, or $i\neq j, i\neq 6$ and $j\neq 4$. Letting $\tilde{\alpha}_{64}=\alpha_{64} \alpha_{54}$ and $\tilde{\epsilon}_{6}=\alpha_{65}\epsilon_{5} + \epsilon_{6}$ we can rewrite this model as
 \begin{align*}
 &V_1 \leftarrow \alpha_{12}V_2 + \epsilon_{v_1},\\
&V_2\leftarrow \epsilon_{v_2},\\
&V_3 \leftarrow \alpha_{32}V_2 + \epsilon_{v_3},\\
&V_4 \leftarrow \alpha_{41}V_1 + \epsilon_{v_4},\\
&V_6 \leftarrow \alpha_{64}V_3 + \tilde{\alpha}_{54}V_4 +\tilde{\epsilon}_{v_6},\\
 \end{align*}
which is a linear structural equation model compatible with $\g^F$ from which $V\setminus F$ is generated.
	
\end{example}

	\subsection{Existing and preparatory results on covariance matrices}

	\begin{lemma}\citep[e.g.][]{henckel2019graphical}
		\label{lemma:conditional covariance matrix}
		Let $A,B$ and $C$ be mean $0$ random vectors with finite variance, with $B$ possibly of length zero. Then 
		\[
		\Sigma_{aa.bc} = \Sigma_{aa.b} - \beta_{ac.b} \Sigma_{cc.b} \beta^\top_{ac.b}.
		\]
	\end{lemma}
	
\begin{lemma}\citep[e.g.][]{henckel2019graphical}\label{lemma:Cochran}
    Let $A,B,C$ and $D$ be mean $0$ random vectors with finite variance, with $C$ possibly of length zero. 
	Then
	\[
	\beta_{ab.c} = 
	\beta_{ab.cd} 
	+ \beta_{ad.bc} \beta_{db.c}.  
	\]
\end{lemma}

\begin{lemma} \citep{wermuth1989moderating}
Let $A,B,C$ and $D$ be mean $0$ Gaussian vectors, with $C$ possibly of length zero. 
	If $B \ci D \mid C$ or $A \ci D \mid B, C$, then
	$
	\beta_{ab.c} = \beta_{ab.cd}.
	$
	Furthermore, if $A \ci D\mid B , C$, then
	$
	\Sigma_{aa.bcd} = \Sigma_{aa.bc}.$

	\label{lemma:Wermuth}
\end{lemma}

\begin{lemma}
	Consider an acyclic directed mixed graph $\g=(V,E)$ and let $(\mathcal{A},\Omega)$ be a Gaussian linear structural equation model compatible with $\g$ with $\Omega$ a strictly diagonally dominant matrix.
	Let $F\subseteq E$ and let $(\mathcal{A}_{F}(\epsilon),\Omega_{F}(\epsilon))$ denote the Gaussian linear structural equation model $(\mathcal{A},\Omega)$ where the edge coefficients and error covariances corresponding to the edges in $F$ are replaced with some value $\epsilon > 0$. Then
	\[\lim_{\epsilon \rightarrow 0}\Sigma_{ab.c}\{\mathcal{A}_{F}(\epsilon),\Omega_{F}(\epsilon)\} = \Sigma_{ab.c}\{\mathcal{A}_{F}(0),\Omega_{F}(0)\}\]  
	for any $A,B,C \subseteq V$.
	\label{lemma:convergence}
\end{lemma}

\begin{proof}
    The proof is based on the fact that as continuous functions in $\epsilon$, $\lim_{\epsilon \rightarrow 0}\mathcal{A}_F(\epsilon)=\mathcal{A}_F(0)$ and $\lim_{\epsilon \rightarrow 0}\Omega_F(\epsilon)=\Omega_F(0)$, and as we will show, $\Sigma_{ab.c}\{\mathcal{A}_F(\epsilon),\Omega_F(\epsilon)\}$ is a continuous function at $\{\mathcal{A}_F(0),\Omega_F(0)\}$. We show this in two steps by first showing that $\Sigma_{vv}\{\mathcal{A}_F(\epsilon),\Omega_F(\epsilon)\}$ is a continuous function at $\{\mathcal{A}_F(0),\Omega_F(0)\}$ and second showing that $\Sigma_{ab.c}(\Sigma_{vv})$ is a continuous function at $\Sigma_{vv}\{\mathcal{A}_F(0),\Omega_F(0)\}.$

    Regarding the first step,
	\begin{align*}
	\Sigma_{vv}(\mathcal{A},\Omega) = (\mathrm{Id}-\mathcal{A})^{-\top} \Omega (\mathrm{Id}-\mathcal{A})^{-1}
	\end{align*}
	is a continuous function in $\Omega$ and $(\mathrm{Id}-\mathcal{A})^{-1}$. It therefore remains to check that $\{\mathrm{Id}-\mathcal{A}_F(0)\}$ is invertible to show it is continuous at $\{\mathcal{A}_F(0),\Omega_F(0)\}$. To see this we us that $\{\mathcal{A}_F(0),\Omega_F(0)\}$ is compatible with the graph $\g'$, which is the graph $\g$ with the edges in $F$ removed. As $\g$ is acyclic, so is $\g'$ and as a result $\{\mathrm{Id}-\mathcal{A}_F(0)\}$ is invertible (see Neumann series). Therefore, 
	\begin{align*}
	\lim_{\epsilon \rightarrow 0}\Sigma_{vv}\{\mathcal{A}_F(\epsilon),\Omega_F(\epsilon)\} &= \lim_{\epsilon \rightarrow 0} \{\mathrm{Id}-\mathcal{A}_F(\epsilon)\}^{-\top} \Omega_F(\epsilon) \{\mathrm{Id}-\mathcal{A}_F(\epsilon)\}^{-1}\\
	&= \lim_{\epsilon \rightarrow 0} \{\mathrm{Id}-\mathcal{A}_F(\epsilon)\}^{-\top} \lim_{\epsilon \rightarrow 0} \Omega_F(\epsilon) \lim_{\epsilon \rightarrow 0} \{\mathrm{Id}-\mathcal{A}_F(\epsilon)\}^{-1}\\
	&= \{\mathrm{Id}-\mathcal{A}_F(0)\}^{-\top} \Omega_F(0) \{\mathrm{Id}-\mathcal{A}_F(0)\}^{-1}
	\end{align*}
	by continuity. 
 
 Regarding the second step, to show that
	\begin{align*} \Sigma_{ab.c}(\Sigma_{vv})
	&= \Sigma_{ab} - \Sigma_{ac} \Sigma^{-1}_{cc} \Sigma_{cb}
	\end{align*}
	is a continuous function at $\Sigma_{vv}\{\mathcal{A}_F(0),\Omega_F(0)\}$ we need to show that the minor $\Sigma_{cc}(\Sigma_{vv})$ is invertible at $\Sigma_{vv}\{\mathcal{A}_F(0),\Omega_F(0)\}$. We do so by showing that $\Sigma_{vv}\{\mathcal{A}_F(0),\Omega_F(0)\}$ is invertible. By the fact that $\Omega$ is a diagonally dominant matrix, so is $\Omega_F(0)$. Therefore, $\Omega_F(0)$ is invertible by the Levy-Desplanques theorem and as a result, $\Sigma_{vv}\{\mathcal{A}_F(0),\Omega_F(0)\}$ is invertible. 
\end{proof}

\section{Proofs for Section \ref{sec:prep results}}

\begin{theorem}
	Consider mutually disjoint nodes $X$ and $Y$, and node sets $Z$ and $W$ in an acyclic directed mixed graph $\g=(V,E)$. 
	Then $(Z,W)$ is a linearly valid conditional instrumental set relative to $(X,Y)$ in $\g$ if and only if 
	(i) $(Z \cup W) \cap \f{\g}=\emptyset$, (ii) $Z \not\perp_{\g} X \mid W$ and (iii) $Z \perp_{\tilde{\g}} Y \mid W$,
	where the graph $\tilde{\g}$ is $\g$ with all edges out of $X$ on causal paths from $X$ to $Y$ removed.
\end{theorem}

\begin{proof}	

We first show that under the three conditions of Theorem \ref{theorem:graphical valid CIS}, there exists a linear structural equation model compatible with $\g$ such that $\Sigma_{xz.w} \neq 0$ and for any such model $\hat{\tau}_{yx}^{z.w}$ consistently estimates $\tau_{yx}$. By Condition (ii) of Theorem \ref{theorem:graphical valid CIS}, that is, $Z \not\perp_{\g} X \mid W$, it follows that there exist linear structural equation model compatible with $\g$ such that $\Sigma_{xz.w} \neq 0$.

We now show that for any model such that $\Sigma_{xz.w} \neq 0$, Condition (i) and (iii) of Theorem \ref{theorem:graphical valid CIS} ensure that $\hat{\tau}_{yx}^{z.w}$ converges in distribution to $\tau_{yx}$. To do so consider any model compatible with $\g$ such that $\Sigma_{xz.w} \neq 0$ and let $S=(X,W)$ and $T=(Z,W)$. Consider the whole vector valued two-stage least squares estimator $\hat{\gamma}_{ys.t}$, whose first entry defines our estimator $\hat{\tau}_{yx}^{z.w}$. Let $\gamma= (\tau_{yx},\beta_{\tilde{y}w})$, with $\tilde{Y}=Y-\tau_{yx}X$ and let $\epsilon=Y-\gamma S$. Then
	\begin{align*}
	\hat{\gamma}_{ys.t} - \gamma 
	&=   Y_n T_n^\top (T_n T_n^\top)^{-1}
	T_n S_n^\top 
	\{S_n T_n^\top (T_n T_n^\top)^{-1} T_n S_n^\top\}^{-1} - \gamma\\
	&=   \epsilon_n T_n^\top (T_n T_n^\top)^{-1} T_n S_n^\top 
	\{S_n T_n^\top (T_n T_n^\top)^{-1} T_n S_n^\top\}^{-1}\\
	&= \frac{1}{n} \epsilon_n T_n^\top \left(\frac{1}{n}T_n T_n^\top\right)^{-1}
	\frac{1}{n}T_n S_n^\top
    \left\{\frac{1}{n}S_n T_n^\top \left(\frac{1}{n}T_n T_n^\top \right)^{-1} \frac{1}{n} T_n S_n^\top \right\}^{-1}, 
	\end{align*}
	where $S_n,T_n$ and $\epsilon_n$ are the random matrices corresponding to taking $n$ i.i.d. observations from $S,T$ and $\epsilon$, respectively. By assumption $\Sigma_{xz.w} \neq 0$ and we can therefore conclude with Lemma \ref{lemma: cor neq 0} that $\Sigma_{st} \Sigma_{tt}^{-1} \Sigma_{st}^\top$ is invertible. With the continuous mapping theorem it follows that that the limit in probability of $(\hat{\gamma}_{ys.t} - \gamma)$ is $ \Sigma_{\epsilon t} \Sigma_{tt}^{-1} \Sigma_{ts} (\Sigma_{st} \Sigma_{tt}^{-1} \Sigma_{st}^\top)^{-1}$.
    It is thus sufficient to show that $\Sigma_{\epsilon t} = 0$, i.e., $\Sigma_{\epsilon z}=0$ and $\Sigma_{\epsilon w}=0$, to prove that $\hat{\gamma}_{ys.t}$ converges in probability to $\gamma$ which implies that $\hat{\tau}_{yx}^{z.w}$ converges to $\tau_{yx}$. 
	
	Consider first $\Sigma_{\epsilon w}$. Clearly, $\Sigma_{\epsilon w}=\mathrm{cov}(Y-\gamma S,W)=\mathrm{cov}(\tilde{Y} - \beta_{\tilde{y}w} W,W)=0$ by the fact that $\beta_{\tilde{y}w}=\Sigma_{\tilde{y}w}\Sigma^{-1}_{ww}$. Consider now $\Sigma_{\epsilon z}=\mathrm{cov}(Y-\tau_{yx} X - \beta_{\tilde{y}w} W,Z)=\Sigma_{\tilde{y}z}-\beta_{\tilde{y}w} \Sigma_{wz}=\Sigma_{\tilde{y}z.w}$.

 We now show that under Conditions (i) and (iii), $\tilde{Y} \ci Z \mid W$ for all linear structural equation models compatible with $\g$ before arguing why this suffices to conclude that $\Sigma_{\epsilon z}=0$. 
 Fix any linear structural equation model with edge coefficient matrix $\mathcal{A}$ and error covariance matrix $\Omega$ compatible with $\gamma$.
 By condition (i), $(Z \cup W) \cap \f{\g}=\emptyset$ and therefore $Z$ and $W$ are also nodes in the latent projection graph $\g^F$ over $F=\f{\g} \setminus \{X,Y\}$. By the fact that the latent projection preserves m-separation statements $Y \perp_{\tilde{\g}} Z \mid W$ implies $Y \perp_{\tilde{\g}^F} Z \mid W$. By Lemma \ref{lemma:forb + trunc}, $Y \perp_{\tilde{\g}^F} Z \mid W$ implies that $Y \perp_{\check{\g}^F} Z \mid W$, where $\check{\g}^F$ is the graph $\g^F$ with the edge $X \rightarrow Y$ removed. But $Y \perp_{\check{\g}^F} Z \mid W$ implies by Lemma \ref{lemma:Martin} that $\tilde{Y} \ci Z \mid W$. If the model we consider is Gaussian this suffices to conclude that $\Sigma_{\tilde{y}z.w}=0$. If the model is not Gaussian, on the other hand, then by the fact that $\Sigma_{\tilde{y}z.w}$ depends on the linear structural equation model only via the tuple $(\mathcal{A},\Omega)$ it follows that $\Sigma_{\tilde{y}z.w}$ must have the same value as in the Gaussian model, i.e., be $0$.
 Combined this shows that if the three conditions of Theorem \ref{theorem:graphical valid CIS} hold, then $\gamma_{ys.t}=(\tau_{yx}, \beta_{\tilde{y}w})$.

 It remains to show that the two conditions of Definition \ref{definition:validCIS} also imply Conditions (i), (ii) and (iii) of Theorem \ref{theorem:graphical valid CIS}. Condition (i) of Definition \ref{definition:validCIS} implies that $X \not\perp_{\g} Z \mid W$, which is Condition (ii) of Theorem \ref{theorem:graphical valid CIS}, by contraposition. We now show that if $\hat{\tau}_{yx}^{z.w}$ converges to $\tau_{yx}$ for all models compatible with $\g$, such that $\Sigma_{yx.z} \neq 0$ then Condition (i) and (iii) of Theorem \ref{theorem:graphical valid CIS} hold. 

 By the same argument as given in the first half of this proof, $\Sigma_{xz.w}\neq 0$ suffices to conclude that the estimator $\hat{\gamma}_{ys.w}$ converges to some unique limit $\gamma$, that is characterized by the fact that $\Sigma_{\epsilon t}=0$ with $\epsilon=Y - \gamma S$. We now show that if this limit vector $\gamma$ has $\tau_{yx}$ as the first entry for all models compatible with $\g$ such that $\Sigma_{xz.w}\neq 0$, Conditions (i) and (iii) of Theorem \ref{theorem:graphical valid CIS} must hold. We have shown already that if the first entry of $\gamma$ is $\tau_{yx}$ then the remaining entries have to form the vector $\beta_{\tilde{y}w}$ as otherwise $\Sigma_{\epsilon w} \neq 0$. We can therefore assume that $\gamma=(\tau_{yx},\beta_{\tilde{y}w})$. This implies as we have also shown previously that $\Sigma_{\tilde{y}z.w}=0$. For a Gaussian model this is only true if $\tilde{Y} \ci Z \mid W$. It therefore suffices to show that if $\tilde{Y} \ci Z \mid W$ for all linear structural equation models compatible with $\g$ then Conditions (i) and (iii) of Theorem \ref{theorem:graphical valid CIS} hold. We do so by contraposition.
 
 If Condition (i) of Theorem \ref{theorem:CIVcomparison} does not hold, there exists a Gaussian model violating that $\tilde{Y} \ci Z \mid W$ by Lemma \ref{lemma: no forb}. Suppose that Condition (i) holds but Condition (iii) does not. By Lemma \ref{lemma:forb + trunc}, Condition (i) holding and Condition (iii) not holding for $\tilde{\g}$ implies that $Y \not\perp_{\check{\g}^F} Z \mid W$. Therefore, we can invoke Lemma \ref{lemma:Martin} to conclude that there exists a Gaussian linear structural equation model such that $\tilde{Y} \notci Z \mid W$ in this case as well. We can therefore conclude that if $\tilde{Y} \ci Z \mid W$ for all linear structural equation model compatible $\g$ then Conditions (i) and (iii) of Theorem \ref{theorem:graphical valid CIS} hold.
\end{proof}

\begin{lemma}
	Consider two random variables $X$ and $Y$, and two random vectors $Z$ and $W$. Let $S=(X,W)$ and $T=(Z,W)$. Then $\Sigma_{st} \Sigma_{tt}^{-1} \Sigma_{st}^\top$ is invertible if and only if $\Sigma_{xz.w} \neq 0$.
	\label{lemma: cor neq 0}
\end{lemma}

\begin{proof}
	Using $\beta_{st}=\Sigma_{st} \Sigma_{tt}^{-1}$ it follows that $\Sigma_{st} \Sigma_{tt}^{-1} \Sigma_{st}^\top=\beta_{st} \Sigma_{tt} \beta^\top_{st}$. As $S=(X,W)$ and $T=(Z,W)$ it follows that 
	\[
	\beta_{st} = 
	\left(
	\begin{array}{cc}
	\beta_{xz.w} &  \beta_{xw.z}  \\
	0 & Id
	\end{array}
	\right).
	\]
	Here we use that $\beta_{nn.u}=\beta_{nn}=1$ and $\beta_{nu.n}=0$ for any random variable $N$ and random vector $U$.
	As $\beta_{xz.w} = \Sigma_{xz.w} \Sigma^{-1}_{zz.w}$, it follows that $\beta_{st}$ is of full row rank if and only if $\Sigma_{xz.w} \neq 0$. As $\Sigma_{tt}$ is a positive-definite matrix, $\beta_{st} \Sigma_{tt} \beta^\top_{st}$ is positive-definite if and only if $\beta_{st}$ has full row rank. 
\end{proof}

\begin{lemma}
	Consider mutually disjoint nodes $X$ and $Y$, and node sets $Z$ and $W$ in an acyclic directed mixed graph $\g=(V,E)$ such that $(Z \cup W) \cap \f{\g} \neq \emptyset$ and $X \not\perp_{\g} Z\mid W$. Then there exists a linear structural equation model compatible with $\g$, such that
	\[\tilde{Y} \notci Z \mid W,\]
	where $\tilde{Y}=Y - \tau_{yx} X$.
	\label{lemma: no forb}
\end{lemma}

\begin{proof}
	The statement is void for $\f{\g}\subseteq \{X,Y\}$. As $Y \notin \de(X,\g)$ implies that $\f{\g}=X$ we can therefore assume that $Y \in \de(X,\g)$. We now construct a Gaussian linear structural equation model in which $\beta_{\tilde{y}z.w} \neq 0$. This suffices to prove our claim since $\beta_{\tilde{y}z.w} \neq 0$ if and only if $\Sigma_{\tilde{y}z.w} \neq 0$ and we consider Gaussian linear structural equation models. To do so we first derive an equation for $\beta_{\tilde{y}z.w}$ that does not depend on $\tilde{Y}$. 
	
	Consider any Gaussian linear structural equation model compatible with $\g$. We augment $\g$ as well as the underlying linear structural equation model by adding $\tilde{Y}$ to $V$ as well as the edges $X \rightarrow \tilde{Y}$ and $Y \rightarrow \tilde{Y}$ with edge coefficients $-\tau_{yx}$ and $1$, respectively, to $E$. We denote the augmented graph with $\g'$. Clearly the augmented model is still a linear structural equation model.
In particular
	 \begingroup\leqnos
	 \begin{align}
	 & \beta_{\tilde{y}z.w} = \beta_{\tilde{y}z.xw} + \beta_{\tilde{y}x.zw}\beta_{xz.w}, \tag{i} \\
	 & \beta_{\tilde{y}z.xw} = \beta_{\tilde{y}z.xyw} + \beta_{\tilde{y}y.xzw}\beta_{yz.xw}  = \beta_{yz.xw} \textrm{ and} \tag{ii} \\
	  & \beta_{\tilde{y}x.zw}=\beta_{\tilde{y}x.yzw} +\beta_{\tilde{y}y.xzw}\beta_{yx.zw}= -\tau_{yx} + \beta_{yx.zw}, \tag{iii}	     
	 \end{align} \endgroup
 where we use that $\pa(\tilde{Y},\g')=\{X,Y\},\de(\tilde{Y},\g')=\{\tilde{Y}\}$ and that as a result $Z\cup W$ is an adjustment set relative to $(\{X,Y\},\tilde{Y})$ since there are no forbidden nodes and no non-causal paths relative to $(\{X,Y\},\tilde{Y})$ along with Lemma \ref{lemma:Cochran}. Plugging the latter two equations into the first one we obtain
  \[
 \beta_{\tilde{y}z.w} = \beta_{yz.xw} + (\beta_{yx.zw} -\tau_{yx}) \beta_{xz.w},
 \]
where the right hand side does not depend in $\tilde{Y}$.

In the remainder of this proof we will show that there exists a linear structural equation model compatible with $\g$, such that 
 \[
 \beta_{yz.xw} + (\beta_{yx.zw} -\tau_{yx}) \beta_{xz.w} \neq 0.
 \]
We now construct a graph $\g''$ where we drop some edges from $\g$ to ensure additional structure required for our proof. As any linear structural equation model compatible with $\g''$ is also compatible with $\g$ it suffices to show that there exists a linear structural equation model compatible with $\g''$, such that 
 \[
 \beta_{yz.xw} + (\beta_{yx.zw} -\tau_{yx}) \beta_{xz.w} \neq 0
 \]
 to complete our proof.

To construct $\g''$, consider a proper path $p$ from $Z$ to $X$ open given $W$ in $\g$. Such a path exists by assumption. For any such path $p$ if it contains a collider $C$, then $\de(C,\g) \cap W \neq \emptyset$, that is, there exists a causal path from $C$ to some node in $W_c \in W$. Let $q_c$ denote a shortest such path. This ensures that $q_c$ only contains one node in $W$. Starting from $p$ we now construct a proper path $p'$ from $Z$ to $X$ that is open given $W$ in $\g$ and for any collider $C$ on $p'$, $q_c$ intersects with $p'$ only at $C$.

If $p$ has this property already we are done, so suppose it does not and let $C$ be the collider closest to $Z$ on $p$ such that $q_c$ and $p$ intersect at some node that is not $C$. We will now construct a path $p'$ that is a proper path from $Z$ to $X$ open given $W$, with $C$ not a collider on $p'$ and for any collider $C'$ on $p'$ that is not a collider on $p$, $q_{c'}$ intersects with $p'$ only at $C'$. By iterative application of this construction we can therefore obtain a path with the required no-intersection but at the colliders property. 

To construct this path, let $I_1\dots,I_k$ be all nodes on $p$ where $p$ and $q_c$ intersect, ordered by how close to $Z$ they appear on $p$. By assumption there are at least two such nodes. Let $p'=p(Z,I_1) \oplus q_c(I_1,I_k) \oplus p(I_k,X)$. We assume for simplicity that $q_c$ points from $I_1$ to $I_k$. The other case follows by the same arguments.

We first show that $p'$ is open given $W$. The paths $p$ and $q_c$ are open given $W$. We therefore only need to consider the intersection points $I_1$ and $I_k$. By assumption, $I_1$ may not be a collider on $p'$ and $I_1 \notin W$ by the fact that the only node in $W$ on $q_c$ is the endpoint node $W_c$. The node $I_k$, on the other hand, may or may not be a collider on $p'$ and may or may not be $W_c \in W$. If $I_k = W_c$, then by the fact that $p$ is open given $W$, $W_c$ must be a collider on $p$. Because $q$ points towards $W_c$, this implies that $W_c$ is also a collider on $p'$ and therefore $p'$ is also open given $W$. If $I_k \neq W$ and it is a collider on $p'$, then by the existence of $q_c(I_k,W_c)$, $\de(I_1,\g) \cap W \neq \emptyset$ and $p'$ is open given $W$. If $I_k \neq W_c$ is a non-collider, then again $p'$ is open. We can therefore conclude that $p'$ is open given $W$.

We now show that $p'$ has at least one less collider violating the no-intersection but at the collider property. The path $p'$ either does not contain $C$ or $C=I_1$. In either case $C$ is not a collider on $p$. Further, the only potential collider on $p'$ that is not necessarily also a collider on $p$, is $I_k$. But the path $q(I_k,W_c)$ does not intersect with $p'$ but at $I_k$. 

While $p'$ may or may not be a proper path from $Z$ to $X$, it must contain a proper subpath that inherits the other two properties we require from $p'$, so let $p''$ denote this subpath. The path $p''$ is a proper path from $Z$ to $X$ that is open given $W$ with $C$ not a collider on $p''$ and no new collider $C'$ such that $q_{c'}$ intersects with $p''$ at a node that is not $C'$.

We can therefore let $p$ be a proper path from $Z$ to $X$ open given $W$, such that for all colliders $C$ on $p$ there exists a causal path $q_c$ from $C$ to $W$, such that $p$ and $q_c$ intersect only at $C$. Let $q_1,\dots,q_l$ a collection of such casual paths corresponding to all colliders on $p$. Set all edge coefficients and error covariances not corresponding to edges either on $p,q_1,\dots,q_l$ or on causal paths from $X$ to nodes in $\f{\g}$ to $0$. Consider the graph $\g''$ with the $0$ coefficient edges dropped. 
The graph $\g''$ has the following properties: $\f{\g''}=\f{\g}$, there exists at most one edge (stemming from $p$) into $X$ in $\g''$ and $p$ is also open given $W$ in $\g''$.

We now construct a linear SEM compatible with $\g''$ such that 
 \[
 \beta_{yz.xw} + (\beta_{yx.zw} -\tau_{yx}) \beta_{xz.w} \neq 0,
 \]
where we proceed differently in the following cases. First, we assume there exists a path from $Z$ to $Y$ in $\g''$ that is open given $W$ and that does not contain $X$. Second, we assume that no such path exists.
 
	 First, suppose that there exists a path $k$ from $Z$ to $Y$ in $\g''$ that is open given $W$ and that does not contain $X$. We show that there then exists a walk from $Z$ to $Y$ open given $W$ that does not contain $X$, where we recall that the definitions of an open path and open walk differ. We first show by contradiction that for any collider $C$ on $k$, $X\notin \de(C,\g'')$, irrespective of whether $k$ is open or not. So suppose there exists a collider $C$ such that $X\in \de(C,\g'')$ and consider the causal path $k'$ from $C$ to $X$. The path $k'$ consists of nodes in $\an(X,\g)$ and therefore must consist of edges on $p$ and the $q_i$'s by construction of $\g''$. As $k'$ ends with an edge into $X$, the only possible such edge in $\g''$ lies specifically on $p$. Further the node adjacent to $X$ may not be a collider on $p$ and hence may not lie on any of the paths $q_i$. Therefore the next edge on $k'$ is also on $p$. By iterative application we can conclude the same for all other edges on $k'$ and thus $k'$ is a subpath of $p$. But then $C$ cannot be a collider on $p$ and the same argument therefore holds for the two edges into $C$ on $k$, but as $p$ is a path it cannot contain three edges adjacent to $C$ and we get a contradiction. Therefore, we have $X\notin \de(C,\g)$ for any collider $C$ on $k$. As $k$ is open given $W$ it holds that $W \cap \de(C,\g'') \neq \emptyset$ for any collider $C$. Combining $k$ with the causal paths from any collider $C$ to the node in $W$ and back we obtain a walk in $\g''$ from $Z$ to $Y$ that is open given $W$ and that does not contain $X$. We can therefore conclude that a walk from $Z$ to $Y$ that is open given $W$ and does not contain $X$ exists. 

  We now use the existence of this walk to construct a Gaussian linear structural equation such that  
  \[
 \beta_{yz.xw} + (\beta_{yx.zw} -\tau_{yx}) \beta_{xz.w} \neq 0
 \]
  Consider a Gaussian linear structural equation model $(\mathcal{A},\Omega)$ compatible with and faithful to $\g''$, with a diagonally dominant $\Omega$. Set all edge coefficients and error covariances for the edges adjacent to $X$ to some value $\epsilon >0$. Then as $\epsilon$ goes to $0$ so do $\beta_{yx.zw}$ and $\beta_{xz.w}$ by Lemma \ref{lemma:convergence} and $\tau_{yx}$ by Wright's path tracing rule. This is not the case for $\beta_{yz.xw}$ as by the walk we assume to exist, $Y \not\perp_{\g'''} Z \mid W$, where $\g'''$ is the graph $\g''$ with the edges adjacent to $X$ dropped. Choosing $\epsilon$ sufficiently small we therefore obtain Gaussian linear structural equation such that  
  \[
 \beta_{yz.xw} + (\beta_{yx.zw} -\tau_{yx}) \beta_{xz.w} \neq 0.
 \]
	 
We now suppose that all paths from $Z$ to $Y$ that are open given $W$ contain $X$. As there exists only one edge into $X$ in $\g''$, $X$ must be a non-collider on any such path and they are therefore blocked given $W$ and $X$. Further, by the fact that $X\notin\de(C,\g'')$ for all colliders $C$ on any path from $Z$ to $Y$, as shown previously, the addition of $X$ to the conditioning set also does not open any new paths from $Z$ to $Y$. It thus follows that $Z \perp_{\g''} Y \mid X \cup W$ and we can conclude that $\beta_{yz.xw}=0$ for all models compatible with $\g''$. The fact that $Z \perp_{\g''} Y \mid X \cup W$ also implies jointly with \ref{lemma:Cochran} that $\beta_{yx.zw}=\beta_{yx.w}$ and jointly with $(Z\cup W) \cap \f{\g} \neq \emptyset$ that $W \cap \f{\g''} = \emptyset$. The latter is true because for any node $F \in \f{\g}$ there exists a path of the form $F \leftarrow \dots \leftarrow C \rightarrow \dots \rightarrow Y$, with possibly $F=C \in \cn{\g''}$. This path consists entirely of nodes in $\f{\g}$ and is only blocked by $W$ if one of the nodes on $p$ is in $W$. It therefore suffices to construct a Gaussian linear structural equation model compatible with $\g''$ such that $\beta_{yx.w} \neq \tau_{yx}$ and $\beta_{xz.w} \neq 0$ to conclude our claim. By Theorem 57 of \citet[][]{perkovic16}, $\beta_{yx.zw} = \tau_{yx}$ for all Gaussian linear structural equation model compatible with $\g''$ if and only if $W \cup Z$ is a valid adjustment set relative to $(X,Y)$ in $\g''$. Since $W \cap \f{\g} \neq \emptyset$ this is not the case and therefore there exists a model such that $\beta_{yx.zw} \neq \tau_{yx}$. This implies that the equation $f(\mathcal{A},\Omega)=\beta_{yx.w}(\mathcal{A},\Omega) - \tau_{yx}(\mathcal{A},\Omega)$ is non-trivial in the non-zero entries of the tuple $(\mathcal{A},\Omega)$. By the same arguments as given in the proof of Lemma 3.2.1 from \citet{spirtes2000causation} the function $\beta_{yx.w} - \tau_{yx}$ is equivalent to a polynomial in the non-zero entries of $(\mathcal{A},\Omega)$. If we generate the entries according a distribution $P$ absolutely continuous with respect to the Lebesgue measure it follows that $P$-a.s. $\beta_{yx.w} - \tau_{yx} \neq 0$. Similarly, since by the construction of $\g''$, $X \not\perp_{\g''} Z \mid W$ it follows that $P$-a.s. $\beta_{xz.w} \neq 0$. Jointly this implies that $P$-a.s., $(\beta_{yx.w}-\tau_{yx}) \beta_{xz.w} \neq 0$ which concludes our proof.
	\end{proof}

\begin{lemma}
	Consider two nodes $X$ and $Y$, and two node sets $Z$ and $W$ in an acyclic directed mixed graph $\mathcal{G}=(V,E)$ such that $\f{\g} \setminus \{X,Y\}=\emptyset$ and $V$ is generated from a linear structural equation model compatible with $\mathcal{G}$. Let $\tilde{Y} = Y - \tau_{yx} X$. 
	Then $Z \perp_{\tilde{\mathcal{G}}} Y \mid W$ implies $Z \ci \tilde{Y} \mid W$, where $\tilde{\mathcal{G}}$ is the graph $\mathcal{G}$ with the edge $X \rightarrow Y$ removed. Further, if $Z \not\perp_{\tilde{\mathcal{G}}} Y \mid W$ there exists a linear structural equation model compatible with $\g$ such that $Z \notci \tilde{Y} \mid W$.
	
	\label{lemma:Martin}
\end{lemma}

\begin{proof}
Consider first the special case that $\f{\g}=\{X\}$, i.e., $Y \notin \de(X,\g)$. In this case $\tilde{Y}=Y$ and $\tilde{\g}=\g$. Therefore the claims follow from the fact that a linear structural equation model is Markov to the graph it is generated according to and that there exists a linear structural equation model faithful to $\g$.

Suppose for the remainder of the proof that $\f{\g} =\{X,Y\}$.
We first prove that $Z \perp_{\tilde{\mathcal{G}}} Y \mid W$ implies $Z \ci \tilde{Y} \mid W$.
	By the assumption that $\f{\g}=\{X,Y\}$, 
	\[
	\tilde{Y}=\sum_{V_i \in \pa(Y,\mathcal{G})} \alpha_{yv_i} V_i + \epsilon_y -\tau_{yx} X=\sum_{V_i \in \pa(Y,\mathcal{\tilde{G}})} \alpha_{yv_i} V_i + \epsilon_y, 
	\]
	as $\tau_{yx}=\alpha_{yx}$ and $\pa(Y,\mathcal{G}) = \{X\}\cup \pa(Y,\mathcal{\tilde{G}})$, where we use that the only causal path between $X$ and $Y$ is the edge $X \rightarrow Y$ and Wright's path tracing rule. As $\de(Y,\mathcal{\tilde{G}})=\{Y\}$ this suffices to show that the distribution of $\tilde{V}$, which is $V$ with $Y$ replaced by $\tilde{Y}$, is generated according to a linear structural equation model compatible with $\mathcal{\tilde{G}}$. By the fact that linear structural equation models are Markov with respect to the graphs they are compatible with, the claim follows.

  For the converse claim we just need to choose a model that is faithful to $\g$ and after setting $\alpha_{yx}=0$, is faithful to $\tilde{\g}$.
\end{proof}

\begin{lemma}

Consider nodes $X$ and $Y$ in an acyclic directed mixed graph $\g$. Let $F =\f{\g}\setminus \{X,Y\}$ and let $\tilde{\g}$ denote the graph $\g$ with all edges out of $X$ on causal paths from $X$ to $Y$ removed. Further let $\check{\g}^F$ be the graph $\g^F$ with the edge $X\rightarrow Y$ removed. Then $\tilde{\g}^F$ and $\check{\g}^F$ are the same graph.
    
    \label{lemma:forb + trunc}
\end{lemma}

\begin{proof}
Clearly, $\tilde{\g}^F$ and $\check{\g}^F$ contain the same nodes. We will now show that $\check{\g}^F$ and $\tilde{\g}^F$ also contain the same edges.

Assume first that $Y \notin \de(X,\g)$. Then $F = \emptyset$ and there are also no edges to remove. Therefore, both $\tilde{\g}^F$ and $\check{\g}^F$ are equal to $\g$. 

Assume now that $Y \in \de(X,\g)$ and consider any edge $e$ in $\g^F$ that is not the edge $X \rightarrow Y$. Suppose that $e$ is also an edge in $\g$. Clearly, it has to be an edge between two nodes not in $F$. As a result, it is not removed in the pruning step forming $\tilde{\g}$ from $\g$. It is also not removed in the latent projection step forming $\tilde{\g}^F$ from $\tilde{\g}$. Therefore, $e$ is also an edge in $\tilde{\g}^F$.

If $e$ is not an edge in $\g$ then it has to correspond to a path $p$ in $\g$ of the form $V_i \rightarrow \dots \rightarrow V_j,V_i \leftarrow \dots \rightarrow V_j$ or $V_i \leftarrow \dots \leftrightarrow \dots \rightarrow V_j$, such that all nodes but $V_i$ and $V_j$ are in $F$. If $p$ contains none of the edges removed in the construction of $\tilde{\g}$ it will clearly be mapped to the same edge $e$ in $\tilde{\g}^{F}$ it corresponds to in $\g^F$. If it does contain such an edge then that edge has to be without loss of generality the first edge on $p$ as $X\notin F$. Therefore, we can assume that $V_i=X$ and that the first edge on $p$ is of the form $\rightarrow$. As the only node in $\de(X,\g)$ that is not in $F$ is $Y$ it follows that $p$ in fact has to be of the form $X\rightarrow  \dots \rightarrow Y$. Therefore $p$ will be mapped to the edge $X\rightarrow Y$ in $\g^{F}$, that is, $e$ is the edge $X\rightarrow Y$. As the edge $X\rightarrow Y$, if it exists, is itself removed in $\tilde{\g}$ we can therefore conclude that $\tilde{\g}^{F}$ cannot contain the edge $X\rightarrow Y$ and that this is the only edge in which it may differ from $\g^F$. We can thus conclude that $\tilde{\g}^{F}=\check{\g}^{F}$.
\end{proof}

\subsection{Proof of Proposition \ref{prop:non forb}}

\begin{proposition}
    Consider nodes $X$ and $Y$ in an acyclic directed mixed graph $\g$ and let $F=\f{\g}\setminus\{X,Y\}$. Then $(Z,W)$ is a linearly valid conditional instrumental set relative to $(X,Y)$ in $\g$ if and only if it is a linearly valid conditional instrumental set relative to $(X,Y)$ in $\g^{F}$.
\end{proposition}

\begin{proof}
Suppose first that $(Z,W)$ is a valid conditional instrument set in $\g$. Then $(Z \cup W) \cap \f{\g}=\emptyset$ and therefore both $Z$ and $W$ are also node sets in $\g^F$. As $\f{\g^F}=\{X,Y\}$ it also follows that $(Z \cup W) \cap \f{\g^F}=\emptyset$. Further, $Z \not\perp_{\g} X \mid W$ implies that $Z \not\perp_{\g^F} X \mid W$ by the fact that the latent projection preserves all m-separation statements among node sets disjoint with $F$. Finally, by Lemma \ref{lemma:forb + trunc}, $\tilde{\g}^F$ is the latent projection graph of the graph $\tilde{\g}$ from Theorem \ref{theorem:graphical valid CIS} over $F$. Therefore, $Z \perp_{\tilde{\g}} Y \mid W$ implies that $Z \perp_{\tilde{\g}^F} Y \mid W$ by the same m-separation preservation argument.

Suppose now that $(Z,W)$ is a valid conditional instrument set in $\g^F$. As $Z$ and $W$ are node sets in $\g^F$ which do not contain $X$ or $Y$, it follows that $(Z \cup W) \cap \f{\g}=\emptyset$. The other two properties follow again by the fact that the latent projection preserves m-separation statements. 
\end{proof}

\subsection{Proof of Proposition \ref{prop:descX}}

\begin{proposition}
	Consider nodes $X$ and $Y$ in an acyclic directed mixed graph $\g$ and let $(Z,W)$ be a linearly valid conditional instrumental set relative to $(X,Y)$ in $\g$. If $(Z \cup W) \cap \de(X,\g) \neq \emptyset$, then $W$ is a valid adjustment set relative to $(X,Y)$ in $\g$. 
\end{proposition}

	\begin{proof}
	We will prove our claim by contradiction. So consider a linearly valid tuple $(Z,W)$ such that $(Z \cup W) \cap \de(X,\g) \neq \emptyset$ and suppose that $W$ is not a valid adjustment set relative to $(X,Y)$ in $\g$. Since $(Z,W)$ is a linearly valid tuple, it holds that $W \cap \f{\g}=\emptyset$. Therefore, $W$ not being a valid adjustment set implies that there exists a non-causal path $p$ from $X$ to $Y$ that is open given $W$. We will now use $p$ and $(Z \cup W) \cap \de(X,\g) \neq \emptyset$ to construct a path $q$ from $Z$ to $Y$ that is open given $W$ in $\tilde{\g}$, where we recall that $\tilde{\g}$ is the graph obtained by deleting all the first edges out of $X$ on causal paths from $X$ to $Y$ in $\g$.
	
	In order to construct $q$, we first show that $p$ is also a path in $\tilde{\g}$ and that it is also open given $W$ in $\tilde{\g}$. As a non-causal path from $X$ to $Y$ that is open given $W$, where $W\cap \f{\g}=\emptyset$, $p$ may not begin with an edge of the form $X \rightarrow C$ with $C \in \cn{\g}$. Therefore, $p$ is also a path in $\tilde{\g}$. Consider any node $N$ in $\g$ and any node $M \in \de(N,\g) \setminus \de(N,\tilde{\g})$. By assumption there exists a causal path from $N$ to $M$ in $\g$ but not in $\tilde{\g}$. Therefore, any such causal path needs to contain an edge of the form $X \rightarrow C$ with $C \in \cn{\g}$. Therefore, $M \in \f{\g}$. It follows that for any node $N$ that if $W \cap \de(N,\g)\neq\emptyset$ it also holds that $W \cap \de(N,\tilde{\g})\neq\emptyset$. In particular this holds for any potential collider on $p$ and we can therefore conclude that $p$ is also open given $W$ in $\tilde{\g}$.

	We now consider in turn the cases $W\cap \de(X,\g) \neq \emptyset$ and $Z \cap \de(X,\g) \neq \emptyset$.

Suppose first that $W \cap \de(X,\g) \neq \emptyset$.
 By the assumption that $(Z,W)$ is a linearly valid conditional instrumental set, there exists a path $p'$ from some node $A \in Z$ to $X$ open given $W$. If the last edge on $p'$ is one of the edges removed in the construction of $\tilde{\g}$ then it must either (i) contain a collider that lies in $\f{\g}$ or (ii) $A \in \f{\g}$. Since by assumption $(Z \cup W) \cap \f{\g}=\emptyset$, (i) would imply that $p'$ is not open given $W$ and (ii) is not allowed. Therefore, $p'$ must be a path in $\tilde{\g}$ that is open given $W$ in $\tilde{\g}$ by the same argument as for $q$. Let $I$ be the node closest to $A$ on $p'$ that is also on $p$ and consider $q=p'(A,I) \oplus p(I,Y)$. It remains to show that $q$ is open given $W$, where we consider in turn the cases that $I=A,I=Y,I=X$, $I\neq X$ is a non-collider on $q$ and $I\neq X$ is a collider on $q$. If $I=A$ or $I=Y$ then $q$ is a subpath of $p$ and $p'$, respectively, and as a result clearly open given $W$. If $I=X$, then since by assumption $W \cap \de(X,\g) \neq \emptyset$ and $X \notin W$, $q$ is open irrespective of whether $X$ is a collider or a non-collider on $q$. If $I\neq X$ is a non-collider on $q$ then $I$ must be a non-collider on either $p'$ or $p$ and as a result $I \notin W$. Therefore, $q$ is open given $W$. Consider finally the case that $I\neq X$ is a collider on $q$. This implies that either $I$ is a collider on $p'$, $p'$ contains a collider on $p'(I,X)$ in $\de(I,\g)$ or $p'(I,X)$ is causal. Since there exists a node $B \in W \cap \de(X,\g)$ all three cases imply that $W \cap \de(I,\g) \neq \emptyset$. Thus, $q$ is open given $W$.

	Suppose now that $Z\cap \de(X,\g) \neq \emptyset$ and let $A' \in Z \cap \de(X,\g)$. 
	By the assumption that $A' \in \de(X,\g)$ there exists a directed path $p''$ from $X$ to $A'$. As $(Z \cup W) \cap \f{\g} = \emptyset$, $p'$ is not a subpath of a causal path from $X$ to $Y$ in $\g$. Therefore it is also a path in $\tilde{\g}$. If $p''$ is not open given $W$ then it contains at least one node $B$ such that $B \in W \cap \de(X,\g)$. Therefore this case reduces to the case that there exists a node $B \in W \cap \de(X,\g)$, which we have already considered. Therefore, we can suppose that $p''$ is open given $W$ in $\tilde{\g}$.  Let $I$ be the first node on $-p''$ that is also on $p$ and consider $q = -p''(A,I) \oplus p(I,Y)$, where $-p''$ denotes the reversed path $p''$ that goes from $A$ to $X$. We now show that $q$ is open given $W$ by considering sequentially the cases that $I=A,I=Y$ and that $I$ is a non-endpoint node on $q$. If $I=A$ or $I=Y$ than $q$ is a subpath of $p$ or $-p''$, respectively, and therefore open given $W$. Suppose now that $I$ is a non-endpoint node on $q$. Since the causal path $p''$ is open given $W$ and $X \notin W$ it follows that $I \notin W$. By the fact that $p''$ is causal it further follows that $I$ may not be a collider on $q$. Therefore, $q$ is open given $W$.
\end{proof}

\section{Proofs for Section \ref{mainIV}}

\subsection{Proof of Theorem \ref{theorem:varequ}}

\begin{theorem}
	Consider nodes $X$ and $Y$ in an acyclic directed mixed graph $\g$ such that $\de(X,\g)=\{X,Y\}$. 
	Let $(Z,W)$ be a linearly valid conditional instrumental set relative to $(X,Y)$ in $\g$ and $\tilde{Y}=Y-\tau_{yx}X$. Then for all linear structural equation models compatible with $\g$ such that $\Sigma_{xz.w}\neq 0$, $\hat{\tau}_{yx}^{z.w}$ is an asymptotically normal estimator of $\tau_{yx}$ with asymptotic variance
	\begin{equation}
	a.var(\hat{\tau}_{yx}^{z.w}) = \frac{\sigma_{\tilde{y}\tilde{y}.w}}{\sigma_{xx.w}-\sigma_{xx.zw}}.
	\end{equation}
\end{theorem}

\begin{proof}
 By Theorem \ref{theorem:graphical valid CIS}, $\hat{\tau}_{yx}^{z.w}$ is a consistent estimator of the total effect $\tau_{yx}$. By Lemma \ref{lemma: heteroskedastic}, it is an asymptotically normal estimator with asymptotic variance  
 \[ 
\frac{E\{(\delta_{x.w}-\delta_{x.zw})^2 \kappa^2_{ys.t}\}}{E\{(\delta_{x.w}-\delta_{x.zw})^2\}^2},
	\]
	with $\delta_{x.w}=X-\beta_{xw}W,\delta_{x.zw}=X-\beta_{xw.z}W-\beta_{xz.w}Z$ and $\kappa_{ys.t}=Y - \gamma_{ys.t} S$, where $S=(X,W)$ and $T=(Z,W)$. 
 By Lemma \ref{lemma: residual independence}, $(\delta_{x.w}-\delta_{x.zw}) \ci \kappa^2_{ys.t}$ and as a result the asymptotic variance simplifies to
 \[ 
\frac{E(\kappa^2_{ys.t})}{E\{(\delta_{x.w}-\delta_{x.zw})^2\}}.
\]
By Lemma \ref{lemma: variance difference}, $E\{(\delta_{x.w}-\delta_{x.zw})^2\} = \sigma_{xx.w} - \sigma_{xx.zw}$. Lastly, we have shown in the proof of Theorem \ref{theorem:graphical valid CIS} that $\gamma_{ys.t}=(\tau_{yx},\beta_{\tilde{y}w})$, where $\tilde{Y}=Y-\tau_{yx}X$. Therefore, $E(\kappa^2_{ys.t})=E(\delta^2_{\tilde{y}w})=\sigma_{\tilde{y}\tilde{y}.w}$, where $\delta_{\tilde{y}w}= \tilde{Y} - \beta_{\tilde{y}w}W$.
\end{proof}

	\begin{lemma}
	Consider two random variables $X$ and $Y$, and two random vectors $Z$ and $W$. Suppose that $(X,Y,Z,W)$ has mean 0 and finite variance, and that $\Sigma_{xz.w}\neq 0$. Let $S=(X,W)$ and $T=(Z,W)$ and consider the two-stage least squares estimator of $Y$ on $X$ using the tuple $(Z,W)$, denoted $\hat{\gamma}_{yx}^{z.w}$. Then $\hat{\gamma}_{yx}^{z.w}$ converges in probability to some limit $\gamma_{yx}^{z.w}$ and $n^{(1/2)}(\hat{\gamma}_{yx}^{z.w}-\gamma_{yx}^{z.w})$ converges in distribution to a normally distributed random variable with mean $0$ and variance
	\[
\frac{E\{(\delta_{x.w}-\delta_{x.zw})^2 \kappa^2_{ys.t}\}}{E\{(\delta_{x.w}-\delta_{x.zw})^2\}^2},
	\]
	with $\delta_{x.w}=X-\beta_{xw}W,\delta_{x.zw}=X-\beta_{xw.z}W-\beta_{xz.w}Z$ and $\kappa_{ys.t}=Y - \gamma_{ys.t} S$. 
	\label{lemma: heteroskedastic}
\end{lemma}

\begin{proof}
We have shown in the proof of Theorem \ref{theorem:graphical valid CIS} that under the assumption that $\Sigma_{xz.w}\neq 0$, the vector valued two-stage least squares estimator $\hat{\gamma}_{ys.t}$ converges in probability to $\Sigma_{yt} \Sigma_{tt} \Sigma_{st} (\Sigma_{st} \Sigma^{-1}_{tt} \Sigma^\top_{st})^{-1}$. By Chapter 12.16 of \citet{hansen2019} the vector valued two-stage least squares estimator $\hat{\gamma}_{ys.t}$ is also asymptotically normal with asymptotic variance matrix
\begin{align*}
a.var(\hat{\gamma}_{ys.t})&= 
(\Sigma_{st}\Sigma^{-1}_{tt}\Sigma_{st}^{\top})^{-1} E(\Sigma_{st} \Sigma^{-1}_{tt} T \kappa^2_{ys.t} T^{\top} \Sigma_{tt}^{-\top} \Sigma_{st}^{\top}) (\Sigma_{st}\Sigma^{-1}_{tt}\Sigma_{ts})^{-\top}.
\end{align*}
Using that $\beta_{st}=\Sigma_{st}\Sigma^{-1}_{tt}$ we can conclude that the limit, the residual variance and the asymptotic variance of the estimator are the same as for the ordinary least squares regression of $Y$ on the fitted value $\hat{S}=\beta_{st}T$. In order to obtain the entry corresponding to $\hat{X}$ of this asymptotic variance matrix, we apply Corollary 11.1 from \citet{buja2014models} to this regression and obtain that $n^{1/2}(\hat{\gamma}_{yx}^{z.w} -  \gamma_{yx}^{z.w})$ converges in distribution to a normally distributed random variable with variance
\[
\frac{E(\delta_{\hat{x}.w}^2 \kappa^2_{ys.t})}{E(\delta_{\hat{x}.w}^2)^2},
	\]
where $\delta_{\hat{x}.w}= \hat{X} - \beta_{\hat{x}w}W$ with $\hat{X}=\beta_{xt}T$. 
Further, as $\hat{X}=X - (X - \beta_{xt}T)=X-\delta_{x.zw}$ it follows that
\begin{align*}
    \delta_{\hat{x}.w}
    &= \hat{X}  - \beta_{xw}W \\
    &= X - \beta_{xw}W - \delta_{x.zw} \\
    &= \delta_{x.w} - \delta_{x.zw},
\end{align*}
where we use that
\begin{align*}
    \beta_{\hat{x}w}
    &=E(\hat{X} W^\top)E(W W^\top)^{-1} \\ 
    &=E\{(\beta_{xz.w}Z + \beta_{xw.z}W) W^\top]\} E(W W^\top)^{-1} \\
    &=\beta_{xz.w} \beta_{zw} + \beta_{xw.z}\\
    &=\beta_{xw},
\end{align*}
by Lemma \ref{lemma:Cochran}. 
\end{proof}

The following two lemmas, which are adaptations of Lemma B.3 and B.4 of \citet{henckel2019graphical}, respectively, rely on the insight that we can rewrite any population level residual in a linear structural equation model as a linear function in the errors $\epsilon$ of the underlying model. For example, given node sets $A$ and $B$ in an acyclic directed mixed graph $\g=(V,E)$, such that $V$ is generated from a linear structural equation model,
\[
A - \beta_{ab}B= \sum_{W \in V}\tau_{aw} \epsilon_{w} -  \sum_{Z \in B} \left( \beta_{az.b_{-z}} \sum_{W \in V}  \tau_{zw} \epsilon_{w} \right),
\]
where $B_{-Z}=B \setminus Z$.

\begin{lemma}
Consider nodes $X$ and $Y$ in an acyclic directed mixed graph $\g=(V,E)$, such that $V$ is generated from a linear structural equation model compatible with $\g$ and $\de(X,\g)=\{X,Y\}$. Let $(Z,W)$ be a linearly valid conditional instrumental set relative to $(X,Y)$ in $\g$ and let $\delta_{x.w}=X-\beta_{xw}W,\delta_{x.zw}=X-\beta_{xw.z}W-\beta_{xz.w}Z$ and $\delta_{\tilde{y}.w}=\tilde{Y} - \beta_{\tilde{y}w} W$, where $\tilde{Y}=Y-\tau_{yx} X$. Then the random variables $\delta_{x.w}-\delta_{x.zw}$ and $\delta_{\tilde{y}.w}$ are independent.
	\label{lemma: residual independence}
\end{lemma}

\begin{proof}
Our proof relies extensively on Lemma \ref{lemma: non-zero coefficient} and is closely related to the proof of Lemma B.3 in \citet{henckel2019graphical}. For every node $A \in V$ in a linear structural equation model $V$ it holds that
\[
A = \sum_{B \in V} \tau_{ab} \epsilon_{b}.
\]
We can use this to rewrite both $\delta_{x.w}-\delta_{x.zw}$ and $\delta_{\tilde{y}.w}$ as linear functions in the errors of the underlying linear structural equation model. Let $S$ be the set of nodes whose corresponding error term has a non-zero coefficient in the equation for $\delta_{x.w}-\delta_{x.zw}$ and $S'$ be the corresponding set for $\delta_{\tilde{y}.w}$. We will now show that by the assumption that $(Z,W)$ is a linearly valid conditional instrumental set relative to $(X,Y)$ it follows that $S \cap S' = \emptyset$ and $ \mathrm{sib}(S,\g) \cap S'=\emptyset$. We do so by contradiction, that is, by showing that if either Case (i) $S \cap S' \neq \emptyset$ or Case (ii) $\mathrm{sib}(S,\g) \cap S' \neq \emptyset$ we can construct a path from $Z$ to $Y$ in $\tilde{\g}$ that is open given $W$.

Throughout this proof we will apply Lemma \ref{lemma: non-zero coefficient} to construct paths from $Z$ to some node $M \neq Y$ in $\g$ that is open given $W$. Any path from $Z$ to $M$ in $\g$ that is open given $W$ cannot contain the edge $X \rightarrow Y$, as it would then contain a collider $Y$ that cannot be opened by $W$ as $\de(Y,\g)=\{Y\}$ and $Y \notin W$. Therefore, any such path $p$ is also a path in $\tilde{\g}$. Further, as $\g$ and $\tilde{\g}$ only differ regarding a directed edge, the sets of siblings in both graphs are the same. We will use these two arguments throughout the remainder of the proof to simply assume that all statements we consider are in the graph $\tilde{\g}$. 

Consider first Case (i): $S \cap S' \neq \emptyset$. This implies that there exist some node $M$, whose corresponding coefficients in the functions for $\delta_{x.w}-\delta_{x.zw}$ and $\delta_{\tilde{y}.w}$ are both non-zero. Suppose first that $M \in W$. By Lemma \ref{lemma: non-zero coefficient} it then follows that there exists a path $p_1$ from some node $A \in Z$ to $M$ in $\g$ that is open given $W$ and whose last edge points into $M$. In addition, there also exists a path $p_2$ from $M$ to $Y$ in $\tilde{\g}$ that is open given $W$ and whose first edge points into $M$. Let $I$ be the first node on $p_1$ that also lies on $p_2$ and consider the path $q=p_1(A,I) \oplus p_2(I,Y)$. We will now show that $q$ is open given $W$. If $I=A$ or $I=Y$, $q$ is a subpath of either $p_1$ or $p_2$ and therefore open given $W$. If $I=M$, $q$ is open as $M\in W$ is a collider on $q$. Suppose now that $I\notin \{A,Y,M\}$ is a non-collider on $q$. Then $I$ may not be a collider on both $p_1$ and $p_2$. Since both $p_1$ and $p_2$ are open given $W$ it therefore follows that $I \notin W$ and therefore $q$ is open given $W$. Suppose now that $I\notin \{A,Y,M\}$ is a collider on $q$. If $I$ is also a collider on either $p_1$ or $p_2$ it follows that $I \in W$ and we are done. If $I$ is not a collider on either $p_1$ or $p_2$, then $p_1$ either of the form $A \cdots \rightarrow I \rightarrow \dots M$ or $A \cdots \leftrightarrow I \rightarrow \dots M$. As $p_1(I,M)$ is open given $W$ and $M \in W$ it follows that $\de(I,\tilde{\g}) \cap W \neq \emptyset$ and therefore $q$ is open given $W$. 

Suppose now that $M \notin W$. Consider first the case $M=Y$. As $Y \notin W$ and $\de(Y,\g)=\{Y\}$ it follows that the coefficient corresponding to $\epsilon_y$ in the equation for $\delta_{x.w}-\delta_{x.zw}$ is always $0$. We can therefore disregard this case. Consider now the case that $M \in Z$. There then exists a path $p$ from $M \in Z$ to $Y$ in $\tilde{\g}$ that is open given $W$ and we are done. Finally, suppose $M \notin Z 
\cup W \cup Y$. It then follows that there exists a path $p_1$ from some node $A \in Z$ to $M$ in $\g$ that is open given $W$. In addition, there also exists a path $p_2$ from $M$ to $Y$ in $\tilde{\g}$ that is open given $W$ and $\de(M,\tilde{\g}) \cap W \neq \emptyset$. Let $I$ be the first node on $p_1$ that is also on $p_2$ and consider $q=p_1(A,I) \oplus p_2(I,Y)$. If $I=A$ or $I=Y$, $q$ is a subpath of either $p_1$ or $p_2$ and therefore open given $W$. If $I=M$, $q$ is open by the fact that $M \notin W$ and $\de(M,\tilde{\g}) \cap W \neq \emptyset$, irrespectively of whether it is a collider or a non-collider on $q$. Suppose now that $I\notin \{A,Y,M\}$ is a non-collider on $q$. Then $I$ may not be a collider on both $p_1$ and $p_2$. Since both $p_1$ and $p_2$ are open given $W$ it therefore follows that $I \notin W$ and therefore $q$ is open given $W$. Suppose now that $I\notin \{A,Y,M\}$ is a collider on $q$. If $I$ is also a collider on either $p_1$ or $p_2$ it follows that $I \in W$ and we are done. If $I$ is not a collider on either $p_1$ or $p_2$, then $p_1$ either of the form $A \cdots \rightarrow I \rightarrow \dots M$ or  $A \cdots \leftrightarrow I \rightarrow \dots M$. As $p_1(I,M)$ is open given $W$ and $\de(M,\tilde{\g}) \cap W \neq \emptyset$ it follows that $\de(I,\tilde{\g}) \cap W \neq \emptyset$ and therefore $q$ is open given $W$.

Consider now Case (ii): $\mathrm{sib}(S,\tilde{\g}) \cap S' \neq \emptyset$. This implies that their exists a pair of nodes $M$ and $M'$, such that $M \in \mathrm{sib}(M',\tilde{\g})$, the coefficient for $\epsilon_m$ in the equation for $\delta_{x.w}-\delta_{x.zw}$ is non-zero and the coefficient for $\epsilon_{m'}$ in the equation for $\delta_{\tilde{y}.w}$ is non-zero. Suppose first that $M \in W$. In this case there exists a path $p$ from some node $A\in Z$ to $M$ that ends with an edge into $M$ and is open given $W$. Suppose also that $M' \in W$. If $M'$ lies on $p$ than it must be a collider on $p$ and therefore $p(A,M')$ is a path from $Z$ to $M'$ that is open given $W$ and whose last edge points into $M'$. If $M'$ does not lie on $p$ we can simply add the edge $M \leftrightarrow M'$ to $p$ to also obtain such a path. If $M' \in W$ we have therefore reduced this case to the corresponding case in Case (i). If $M' \notin W$, then we know that there either exists a causal path from $M'$ to $Y$ that is open given $W$ or $\de(M',\tilde{\g}) \cap W \neq \emptyset$ and there must exist a path from $M'$ to $Y$ that is open given $W$. In both cases we can use the edge $M \leftrightarrow M'$ to construct a path from $Y$ to $M$ that is open given $W$ and whose last edge is into $M$. Again we have reduced the problem to the corresponding case in Case (i). 

Suppose now that $M \notin W$. If $M' \in W$ we can argue as for the case $M \in W$ and $M' \notin W$. Lastly, suppose that both $M$ and $M'$ are not in $W$. Again we have four cases to consider: a) $\de(M,\tilde{\g}) \cap W \neq \emptyset$ and $\de(M',\tilde{\g}) \cap W \neq \emptyset$, b) $\de(M,\tilde{\g}) \cap W \neq \emptyset$ and $\de(M',\tilde{\g}) \cap W = \emptyset$, c) $\de(M,\tilde{\g}) \cap W = \emptyset$ and $\de(M',\tilde{\g}) \cap W \neq \emptyset$ and d) $\de(M,\tilde{\g}) \cap W = \emptyset$ and $\de(M',\tilde{\g}) \cap W = \emptyset$. In Case a), there exists a path $p_1$ from $Z$ to $M$ that is open given $W$ as well as a path $p_2$ from $M'$ to $Y$ that is open given $W$. As $\de(M',\tilde{\g}) \cap W \neq \emptyset$ we can use the edge $M \leftrightarrow M'$ to either extend $p_2$ to a path from $M$ to $Y$ that is open given $W$ or $M$ lies on $p_2$ and $p_2(M,Y)$ is such a path. In either case we have reduced this to the corresponding setting in Case (i). In Case b), there exists a path $p_1$ from $Z$ to $M$ that is open given $W$ as well as a causal path $p_2$ from $M'$ to $Y$ that is open given $W$. As $p_2$ is causal, $M$ may not be a collider on $p_2$, so we can again either enlarge $p_2$ with the edge $M \leftrightarrow M'$ or use $p(M,Y)$ instead to reduce this to the corresponding setting in Case (i). The same arguments apply to Cases c) and d).

In conclusion, the set nodes whose errors have non-zero coefficients in the equation for $\delta_{\tilde{y}.w}$ and the corresponding set for $\delta_{x.w}-\delta_{x.zw}$ are disjoint. In addition, there also exist no bi-directed edges between any two nodes from the two sets. By our assumption on the independences that hold between the errors in a linear structural equation model, we can therefore conclude that $\delta_{\tilde{y}.w} \ci (\delta_{x.w}-\delta_{x.zw})$.
\end{proof}

\begin{lemma}
Consider nodes $A$ and $M$, and node sets $Z$ and $W=\{W_1,\dots,W_k\}$ in an ADMG $\g = (V,E)$, such that $V=\{V_1,\dots,V_q\}$ is generated from a linear structural equation model compatible with $\g$. Let $T=(Z,W)$ and consider the residual $\delta_{a.w}=A-\beta_{aw}W$. If the coefficient for $\epsilon_m$ in the equation for $\delta_{a.w}$ is non-zero, then there are two cases: First, if $M \in W$ then there exists a path $p$ from $A$ to $M$ that ends with an edge pointing into $M$ and which is open given $W$. Second, if $M \notin W$, then there exists a causal path from $M$ to some node $M' \in \{A\} \cup W$ such that the coefficient of $\epsilon_{m'}$ is non-zero and there exists a path from A to M that is open given W.

Consider also the difference of residuals $\delta_{a.w} - \delta_{a.zw}=\beta_{at}T - \beta_{aw}W$. If the coefficient for $\epsilon_m$ in the equation for $\delta_{a.w} - \delta_{a.zw}$ is non-zero, then there exist two cases: First, if $M \in W$ then there exists a path $p$ from $Z$ to $M$ that ends with an edge pointing into $M$ and which is open given $W$. Second, if $M \notin W$, then there exists a causal path from $M$ to some node $M' \in (Z \cup W)$ such that the coefficient of $\epsilon_{m'}$ is non-zero and there exists a path from M to Z that is open given W.

\label{lemma: non-zero coefficient}
\end{lemma}

\begin{proof}
The first statement is a generalization of Lemma B.4 of \citet{henckel2019graphical}, which is for directed acyclic graphs, to acyclic directed mixed graph. Using Lemma \ref{lemma: cancellation}, it can be shown with the exact same arguments as in the proof of Lemma B.4 which we now give for completeness.

We first rewrite the random variable $\delta_{a.w}$ as a linear combination of the errors from the underlying linear structural equation model as follows:
\begin{align*}
\delta_{a.w}
&=A-\beta_{aw}W\\
&=\sum_{V_j \in V}\tau_{av_j} \epsilon_{v_j}
-\sum_{V_j \in V}\beta_{aw} \tau_{wv_j} \epsilon_{v_j}\\
&=\sum_{V_j \in V} \gamma_{v_j} \epsilon_{v_j},
\end{align*}
with $\gamma_{v_j}=\tau_{av_j}-\beta_{aw} \tau_{wv_j}$ the coefficient corresponding to $\epsilon_{v_j}$.

We now prove the first half of the first statement of the Lemma by contraposition. So suppose $M \in W$ and assume that there is no path from $A$ to $M$ that is open given $W$ and ends with an edge pointing into $M$. Under these assumption Lemma \ref{lemma: cancellation} holds and therefore
\begin{align*}
\gamma_M=\tau_{aw_j}-\beta_{aw} \tau_{ww_j}=0,
\end{align*}
where we assume that $M=W_j$.

We now prove the second half of the first statement of the Lemma. So suppose $M \notin W$ and let $W'=\{W_1,\dots,W_k,A\}$. Then
\begin{align*}
\gamma_M
&=\tau_{am}-\beta_{aw} \tau_{wm}\\
&=\sum_{W_i \in W'}\tau_{w_im.w'_{-i}}(\tau_{aw_j}-\beta_{aw} \tau_{ww_j})\\
&=\sum_{W_i \in W'}\tau_{w_im.w'_{-i}}\gamma_{w_j},
\end{align*}
where we use that by Lemma B.7 of \citet{henckel2019graphical},
$\tau_{w_im} = \sum_{W_j \in W'} \tau_{w_jm.w'_{-j}} \tau_{w_iw_j}$. 
Lemma B.7, although stated for directed acyclic graphs, also holds for acyclic directed mixed graphs as all of its arguments only depend on the linear structural equation model via the matrix of edge coefficients $\mathcal{A}$. Therefore, $\gamma_M$ is only non-zero if there exists a $W_j \in W$ such that $\tau_{w_jm.w'_{-j}} \neq 0$ and $\gamma_{w_j} \neq 0$, that is, if there exists a causal path $p$ from $M$ to some $W_j \in W'$ (that does not contain other nodes in $W'$) and the coefficient for $\epsilon_{w_j}$ is non-zero. The latter implies that there exists a path $p'$ from $A$ to $W_j$ that is open given $W$ and whose last edge points into $W_j$ by what we have already shown in this proof. 

It remains to show that there exists a path from $M$ to $A$ that is open given $W$. If $W_j=A$, $p$ is such a path so assume this is not the case. Consider the first node $I$ on $p$ where $p$ and $p'$ intersect and let $q=p(M,I) \oplus p'(I,A)$. If $I=M$ or $I=A$, $q$ is a subpath of $p'$ respectively $p$ and therefore open given $W$. Suppose first that $I \in W$. Then $I=W_j$, as $W_j$ is the only node in $W$ on $p$. But then $W_j$ is a collider by the fact that the edge on $p'$ points into $W_j$ and $p$ is causal. Therefore, $q$ is open given $W$. If $I \notin W$ then by the fact that any node on $p$ has $W_j$ as a descendant by nature of $p$, $W_j \in \de(I,\g)$ and therefore $q$ is open irrespective of whether $I$ is a collider on $q$ or not.

For the second statement of the lemma we will show that if the coefficient for $M$ in the equation for $\delta_{z.w}$, denoted $\gamma'_m$ is zero, so is the one in the equation for $\delta_{a.w} - \delta_{a.zw}$, denoted $\gamma_m$. We can write  
\begin{align*}
    \gamma_m
    &= \beta_{az.w} \tau_{zm} + \beta_{aw.z} \tau_{wm}  - \beta_{aw} \tau_{wm} \\
    &= \beta_{az.w} \tau_{zm} + \beta_{aw.z} \tau_{wm}  - (\beta_{aw.z} + \beta_{az.w} \beta_{zw}) \tau_{wm} \\
    &= \beta_{az.w} \tau_{zm} -  \beta_{az.w} \beta_{zw} \tau_{wm}
\end{align*}
where we use Lemma \ref{lemma:Cochran} to plug in $\beta_{aw}=\beta_{aw.z} + \beta_{az.w} \beta_{zw}$. This coefficient is zero whenever
\[
\gamma'_{m} = \tau_{zm} - \beta_{zw} \tau_{wm} = 0.
\]
But since we can apply the first half of the Lemma to $\gamma'_{m}$, our statement follows.
\end{proof}

\begin{lemma}
\label{lemma: cancellation}
Consider node sets $Z$ and $W=\{W_1,\dots,W_k\}$ in an acyclic directed mixed graph $\g=(V,E)$, such that $V$ is generated from a linear structural equation model  compatible with $\g$. Let $W_j$ be some node in $W$, such that all paths from $Z$ to $W_j$ that end with an edge of the form $\rightarrow$ or $\leftrightarrow$ are blocked by $W$, then 
\begin{align}
    \tau_{zw_j}=\beta_{zw} \tau_{ww_j}.
    \label{eq: cancellation}
\end{align}
\end{lemma}

\begin{proof}
We show our claim in two steps. First we construct an enlarged linear structural linear equation model, such that there exists valid adjustment sets with respect to $(W_j,Z)$ and $(W_j,W_{-j})$. We use these to replace the terms $\tau_{zw_j}$ and $\tau_{ww_j}$ in Equation \eqref{eq: cancellation} with ordinary least squares coefficients. Second, we apply Lemma \ref{lemma:Cochran} to the ordinary least squares coefficients corresponding to $\tau_{zw_j}$ and then simplify the resulting equation by using Lemma \ref{lemma:Wermuth} to arrive at Equation \eqref{eq: cancellation}.

We first enlarge the underlying linear structural equation model as follows. For every node in $V \in \mathrm{sib}(W_j,\g)$ add a node $\epsilon_v$, a directed edge from $\epsilon_v$ to $V$ and a bidirected edge from $\epsilon_v$ to $W_j$. Let $\g'=(V',E')$ denote this enlarged model and its causal graph. Consider the set $P'=pa(W_j,\g)\cup sib(W_j,\g')$ and two sets $A \subseteq \de(W_j,\g)\setminus W_j$ and $B \subseteq V \setminus \de(W_j,\g)$. By construction the set $P'\cup B$ does not contain any descendants of $W_j$. In addition, any non-causal path from $W_j$ to $A$ that starts with an edge into $W_j$, must contain a node in $P'\cup B$ that is a non-collider. The latter is due to the fact that by the construction of $\g'$ no node in $\mathrm{sib}(W_j,\g')$ may be a collider on any path. Therefore, $P'\cup B$ is a valid adjustment set with respect to $(W_j,A)$. In particular, $\tau_{zw_j}=\beta_{zw_j.p'}$ and $\tau_{w_{-j}w_j}=\beta_{w_{-j}w_j.p'}$, where $W_{-j}=W \setminus W_j$. Applying this to equation  \eqref{eq: cancellation} we obtain
\begin{align*}
    \tau_{zw_j}
    &=\beta_{zw_j.p'} \\
    &=\beta_{zw_j.w_{-j}p''} + \beta_{zw_{-j}.w_jp''} \beta_{w_{-j}w_j.p'} \\
    &=\beta_{zw.p''} \tau_{ww_j},
\end{align*}
where we use Lemma \ref{lemma: degenerate} in the second step and where $P''$ is some maximal subset of $P'$, such that the covariance matrix of $(W,P'')$ has full rank. The covariance matrix $(W_j,P')$ is of full rank by our construction of $\g'$. By assumption $P' \perp_{\g'} Z \mid W$, as any path that would contradict this statement could be extended to a path from $Z$ to $W_j$ whose last edge points into $W_j$ and that is open given $W$. As a result $\beta_{zw.p''}=\beta_{zw}$ and our claim follows.
\end{proof}
	
	\begin{lemma}
	Consider a random variable $X$, and two random vectors $Z$ and $W$. Then
	\[E\{(\delta_{x.w}-\delta_{x.zw})^2\}=\sigma_{xx.w}-\sigma_{xx.zw},\]
	where $\delta_{x.w}=X - \beta_{xw} W$ and $\delta_{x.zw}=X - \beta_{xw.z} W - \beta_{xz.w} Z$.
	\label{lemma: variance difference}
	\end{lemma}
	
	\begin{proof}
	\begin{align*}
	    E\{(\delta_{x.w}-\delta_{x.zw})^2\}
	    &= E(\delta_{x.w})^2-2 E(\delta_{x.w} \delta_{x.zw}) + E(\delta_{x.zw}^2)\\
	    &= \sigma_{xx.w}-2 E(\delta_{x.w} \delta_{x.zw}) + \sigma_{xx.zw}.
	\end{align*}
	Further,
	\begin{align*}
	    E(\delta_{x.w} \delta_{x.zw})
	    &= E\{(X-\beta_{xw}W)(X-\beta_{xw.z}W-\beta_{xz.w}Z)^\top\}\\
	    &= \sigma_{xx} - \Sigma_{xw}\beta^\top_{xw.z} - \Sigma_{xz}\beta^\top_{xz.w} - \beta_{xw} \Sigma_{wx} + \beta_{xw}\Sigma_{ww}\beta^\top_{xw.z} + \beta_{xw}\Sigma_{wz}\beta^\top_{xz.w} \\
	    &= (\sigma_{xx} - \beta_{xw} \Sigma_{ww} \beta^\top_{wx}) - \beta_{xw} \Sigma_{ww} \beta^\top_{xw.z} - \beta_{xz}\Sigma_{zz}\beta^\top_{xz.w}  + \\ 
	    & \quad \quad \beta_{xw}\Sigma_{ww}\beta^\top_{xw.z} + \beta_{xw}\Sigma_{ww}\beta_{zw}^\top\beta^\top_{xz.w}\\
	    &= \sigma_{xx.w} - \beta_{xz}\Sigma_{zz}\beta^\top_{xz.w} 
	    + \beta_{xw}\Sigma_{ww}\beta_{zw}^\top\beta^\top_{xz.w} \\
	    &= \sigma_{xx.w} - (\beta_{xz}\Sigma_{zz} - \beta_{xw}\Sigma_{ww}\beta_{zw}^\top)\beta^\top_{xz.w} \\
	    &= \sigma_{xx.w} - \{(\beta_{xz.w} + \beta_{xw.z}\beta_{wz})\Sigma_{zz} - (\beta_{xw.z}+\beta_{xz.w}\beta_{zw})\Sigma_{ww}\beta_{zw}^\top\}\beta^\top_{xz.w} \\
	    &= \sigma_{xx.w} - \beta_{xz.w}(\Sigma_{zz} - \beta_{zw}\Sigma_{ww}\beta_{zw}^\top)\beta^\top_{xz.w} \\
	    &= \sigma_{xx.w} - \beta_{xz.w}\Sigma_{zz.w}\beta^\top_{xz.w} \\
	    &= \sigma_{xx.zw}.
	\end{align*}
    As a result,
	\begin{align*}
	    E\{(\delta_{x.w}-\delta_{x.zw})^2\}
	    &= \sigma_{xx.w} - \sigma_{xx.zw}
	\end{align*}
	follows.
	\end{proof}
	
	\begin{lemma}
	Let $A,B,C$ and $D$ be mean 0 random vectors with finite variance, such that $\Sigma_{e_1e_1}$, with $E_1=(B,C,D)$, is not of full rank but $\Sigma_{e_2e_2}$ and $\Sigma_{e_3e_3}$, with $E_2=(B,C)$ and $E_3=(B,D)$, are of full rank.
	Then 
	\[
	\beta_{ab.c} = \beta_{ab.c'd} + \beta_{ad.bc'} \beta_{db.c},
	\]
	where $C'$ is any maximally sized subset of $C$, such that $\Sigma_{e_4e_4}$, with $E_4=(B,C',D)$ is of full rank.
	\label{lemma: degenerate}
	\end{lemma}
	
\begin{proof}
Regressing $A$ on $B,C'$ and $D$, yields the equation
\[
A = \beta_{ab.c'd} B + \beta_{ac'.bd} C' + \beta_{ad.bc'} D + \epsilon_{a.bc'd}.
\]
Regressing $D$ on $B,C$, yields the equation
\[
D = \beta_{db.c} B + \beta_{dc.b} C + \epsilon_{d.bc}.
\]
Plugging the latter equation into the former yields
\begin{equation*}
A = (\beta_{ab.c'd} + \beta_{ad.bc'}\beta_{db.c})B + \beta_{ac'.bd} C' + \beta_{ad.bc'} \beta_{dc.b} C + \beta_{ad.bc'}\epsilon_{d.bc} + \epsilon_{a.bc'd}.  
\end{equation*}
We can rewrite this equation as
\[
A = (\beta_{ab.c'd} + \beta_{ad.bc'}\beta_{db.c})B + \tilde{\beta} C + \tilde{\epsilon},
\]
where $\tilde{\beta}$ is some real-valued vector of the same length as the vector $C$ and $\tilde{\epsilon}=\beta_{ad.bc'}\epsilon_{d.bc} + \epsilon_{a.bc'd}$. As $E\{\epsilon_{a.bc'd}(B,C',D)\}=0$ it follows that it also holds that $E\{\epsilon_{a.bc'd}(B,C,D)\}=0$ by our assumptions on $C$. We can therefore conclude that $E\{\tilde{\epsilon}(B,C)\}=0$. Our claim then follows by the uniqueness of the ordinary least squares regression coefficient vector $(\beta_{ab.c},\beta_{ac.b})$.
\end{proof}
	
\subsection{Proof of Theorem \ref{theorem:CIVcomparison}}

\begin{theorem}
	Consider nodes $X$ and $Y$ in an acyclic directed mixed graph $\g$ such that $\de(X,\g)=\{X,Y\}$. Let $(Z_1,W_1)$ and $(Z_2,W_2)$ be linearly valid conditional instrumental sets relative to $(X,Y)$ in $\g$. Let $W_{1\setminus 2}=W_1 \setminus W_2$ and $W_{2\setminus 1}=W_2 \setminus W_1$. If the following four conditions hold
	\begingroup\leqnos
	\begin{align}
	    & W_{1\setminus 2} \perp_{\tilde{\g}} Y \mid W_2, \tag{a} \label{app:condition1}\\
	    & \text{(i) } W_{1\setminus 2} \perp_{\g} Z_2 \mid W_2 \text{ or (ii) }W_{1\setminus 2} \setminus Z_2 \perp_{\g} X \mid Z_2 \cup W_2, \tag{b}  \label{app:condition2}\\
	    & \text{(i) } W_{2\setminus 1} \setminus Z_1 \perp_{\g} Z_1\mid W_1 \text{ and } W_{2\setminus 1} \cap Z_1 \perp_{\g} X \mid W_1 \cup (W_{2\setminus 1}\setminus Z_1) \text{ or} \tag{c} \label{app:condition3}\\ 
	    & \quad \text{(ii) } W_{2\setminus 1} \perp_{\g} X \mid  W_1, \nonumber\\
	    & Z_1 \setminus (Z_2 \cup W_{2\setminus 1}) \perp_{\g} X \mid Z_2 \cup W_1 \cup W_{2\setminus 1}, \tag{d} \label{app:condition4}
	\end{align}\endgroup
	then for all linear structural equation models compatible with $\g$ such that $\Sigma_{xz_1.w_1}\neq 0$, 
	$
	a.var(\hat{\tau}_{yx}^{z_2.w_2}) \leq a.var(\hat{\tau}_{yx}^{z_1.w_1}).
    $
\end{theorem}

\begin{proof}
Fix some linear structural equation model such that $\Sigma_{xz_1.w_1}\neq 0$. By the m-separation assumptions we can apply both Lemma \ref{lemma: residual variance} and Lemma \ref{lemma: instrumental strength} to conclude that $\sigma_{\tilde{y}\tilde{y}.w_1} \geq \sigma_{\tilde{y}\tilde{y}.w_2}$ and $\sigma_{xx.w_1} - \sigma_{xx.z_1w_1} \leq \sigma_{xx.w_2} - \sigma_{xx.z_2w_2}$.

Given that $\Sigma_{xz_1.w_1} \neq 0$, it holds that $\sigma_{xx.w_1}-\sigma_{xx.z_1w_1}=\Sigma_{xz_1.w_1}\Sigma_{z_1z_1.w_1}^{-1}\Sigma^\top_{xz_1.w_1}>0$. Therefore, $\sigma_{xx.w_2}-\sigma_{xx.z_2w_2} = \Sigma_{xz_2.w_2}\Sigma_{z_2z_2.w_2}^{-1}\Sigma^\top_{xz_2.w_2}>0$ which implies that $\Sigma_{xz_2.w_2} \neq 0$ also holds. As a result, Theorem \ref{theorem:varequ} applies to both $(Z_1,W_1)$ and $(Z_2,W_2)$, so we can conclude that $	a.var(\hat{\tau}_{yx}^{z_1.w_1})=\sigma_{\tilde{y}\tilde{y}.w_1}/(\sigma_{xx.w_1} - \sigma_{xx.z_1w_1})$ and $a.var(\hat{\tau}_{yx}^{z_1.w_1})=\sigma_{\tilde{y}\tilde{y}.w_2}/(\sigma_{xx.w_2} - \sigma_{xx.z_2w_2})$. The claim than follows by the fact that $\sigma_{\tilde{y}\tilde{y}.w_1} \geq \sigma_{\tilde{y}\tilde{y}.w_2}$ and $\sigma_{xx.z_1w_1} \leq \sigma_{xx.z_2w_2}$ as shown already.
\end{proof}

\begin{lemma}
	Consider mutually disjoint nodes $X$ and $Y$, and node sets $W_1$ and $W_2$ in an acyclic directed mixed graph $\g=(V,E)$ such that $V$ is generated from a linear structural equation model compatible with $\g$ and $\de(X,\g)=\{X,Y\}$. Let $\tilde{Y} = Y - \tau_{yx}X$. If 
	$
	W_1 \setminus W_2\perp_{\tilde{\g}} Y \mid W_2,
	$
	then $\sigma_{\tilde{y}\tilde{y}.w_2} \leq \sigma_{\tilde{y}\tilde{y}.w_1}$.
	\label{lemma: residual variance}
\end{lemma}

\begin{proof}
		As we assume that $W_{1 \setminus 2}  \perp_{\tilde{\g}} Y \mid W_2$ it follows that $W_{1 \setminus 2}  \ci \tilde{Y} \mid W_2$ by Lemma \ref{lemma:Martin}. It follows that
		\begin{align*}
		\sigma_{\tilde{y}\tilde{y}.w_1} \geq \sigma_{\tilde{y}\tilde{y}.w_{1\setminus 2} w_2} = \sigma_{\tilde{y}\tilde{y}.w_2}
		\end{align*}
		holds.
\end{proof}
	
	\begin{lemma}
	Consider nodes $X$ and $Y$ in an acyclic directed mixed graph $\g=(V,E)$, such that $V$ is generated from a linear structural equation model compatible with $\g$ and $\de(X,\g)=\{X,Y\}$. Let $(Z_1,W_1)$ and $(Z_2,W_2)$ be two pairs of disjoint node sets in $\g$. Let $W_{1\setminus 2}=W_1 \setminus W_2$ and $W_{2\setminus 1}=W_2 \setminus W_1$. If the following three conditions hold
	\begingroup\leqnos
	 \begin{align}
		& W_{1\setminus 2} \perp_{\g} Z_2 \mid W_2 \text{ or } W_{1\setminus 2} \setminus Z_2 \perp_{\g} X \mid Z_2 \cup W_2 \text{ and }  \label{efficiency:ind2} \tag{a} \\
		& \text{(i) } W_{2\setminus 1} \setminus Z_1 \perp_{\g} Z_1\mid W_1 \text{ and } W_{2\setminus 1} \cap Z_1 \perp_{\g} X \mid W_1 \cup (W_{2\setminus 1}\setminus Z_1) \label{efficiency:ind3} \tag{b} \\ 
		& \quad \quad \text{or (ii) } W_{2\setminus 1} \perp_{\g} X \mid  W_1 \text{ and}  \nonumber  \\
		& Z_1 \setminus (Z_2 \cup W_{2\setminus 1}) \perp_{\g} X \mid Z_2 \cup W_1 \cup W_{2\setminus 1},  \label{efficiency:ind4} \tag{c}
	\end{align}\endgroup
	then $\sigma_{xx.w_2} - \sigma_{xx.w_2z_2} \geq \sigma_{xx.w_1} - \sigma_{xx.w_1z_1}$.
	\label{lemma: instrumental strength}
\end{lemma}

	\begin{proof}
		Let $A=W_{1\setminus 2}$ and $C=W_{2\setminus 1}$ to unclutter the notation.
	By Condition \eqref{efficiency:ind3} it follows that 
	\begin{align*}
	\sigma_{xx.w_1} - \sigma_{xx.z_1w_1} \leq \sigma_{xx.w_1c} - \sigma_{xx.z_1w_1c'},
	\end{align*}
	with $C'=C \setminus Z_1$. This follows either by first applying Lemma \ref{lemma:equalinstruments} with $B=C'$ and then using that if $ C \cap Z_1 \perp_{\g} X \mid (W_1 \cup C')$, $\sigma_{xx.w_1c'} =  \sigma_{xx.w_1c}$ or by applying the latter argument to all of $C$ using the alternative assumption $C \perp_{\g} X \mid W_1$. By Condition \eqref{efficiency:ind4}, it holds that $\sigma_{xx.z_1w_1c'} \geq \sigma_{xx.z_1z''_2w_1c'} = \sigma_{xx.z'_2w_1c}$, where $Z'_2 = Z_2 \setminus W_1$ and $Z''_2 = Z_2 \setminus (W_1 \cup Z_1)$.
	Hence,
	\begin{align*}
	\sigma_{xx.w_1c} - \sigma_{xx.z_1w_1c'} \leq \sigma_{xx.w_1c} - \sigma_{xx.z'_2w_1c}.
	\end{align*}
	Finally, by Condition \eqref{efficiency:ind2} it follows that 
	\begin{align*}
	\sigma_{xx.w_1c} - \sigma_{xx.z'_2w_1c} = \sigma_{xx.aw_2} - \sigma_{xx.z'_2aw_2} \leq \sigma_{xx.a'w_2} - \sigma_{xx.z_2w_2} \leq \sigma_{xx.w_2} - \sigma_{xx.z_2w_2},
	\end{align*}
	with $A'=A \cap Z_2$.
	This follows either by applying Lemma \ref{lemma:equalinstruments} with $B=A\setminus Z_2$ or by the fact that if $(A \setminus Z_2) \perp_{\g} X \mid Z_2 \cup W_2$, then $\sigma_{xx.z'_2w_2a} =  \sigma_{xx.z_2w_2}$ holds.
	\end{proof}

\begin{lemma}
	Consider a node $X$ and two disjoint node sets $Z$ and $W$ in an acyclic directed mixed graph $\g=(V,E)$ such that the distribution of $V$ is compatible with $\g$. Let $ W = A \cup B$ be a partition of $W$, i.e. $A \cap B = \emptyset$, and suppose that $A \perp_{\g} Z \mid B$. Then 
	\[
	\sigma_{xx.ab} - \sigma_{xx.zab} = \sigma_{xx.b} - \sigma_{xx.zb}.
	\]
	
	\label{lemma:equalinstruments}
\end{lemma}

\begin{proof}
	Using Lemma \ref{lemma:Wermuth} we can conclude that
	\[\sigma_{xx.ab} - \sigma_{xx.zab} = \beta_{xz.ab} \sigma_{zz.ab} \beta^\top_{xz.ab} = 
	\beta_{xz.b} \sigma_{zz.b} \beta^\top_{xz.b} = \sigma_{xx.b} - \sigma_{xx.zb}\]
	holds.
\end{proof}

\subsection{Proof of Proposition \ref{prop:moreW}}

\begin{proposition}
	Consider nodes $X,Y$ and $N$ in an acyclic directed mixed graph $\g$ such that $\de(X,\g) = \{X,Y\}$. Let $(Z,W)$ and $(Z,W\cup N)$ be distinct linearly valid conditional instrumental sets relative to $(X,Y)$ in $\g$. 
	If $(Z\cup N,W)$ is not a linearly valid conditional instrumental set relative to $(X,Y)$ in $\g$, then for all linear structural equation models compatible with $\g$ such that $\Sigma_{xz.w}\neq 0$,
	$a.var(\hat{\tau}_{yx}^{z.wn}) \leq a.var(\hat{\tau}_{yx}^{z.w}).$
\end{proposition}
\begin{proof}
	We will show our claim by showing that $(Z,W)$ and $(Z,W\cup N)$ being linearly valid conditional instrumental sets while $(Z\cup N,W)$ is not, implies that $N \perp_{\g} Z \mid W$. Using this we can apply Theorem \ref{theorem:CIVcomparison} with $Z_1=Z_2=Z,W_1=W$ and $W_2=W \cup  N$ to conclude the proof. 
	
	We first show that $N \not\perp_{\tilde{\g}} Y \mid W$, as we use this later. The fact that $(Z,W\cup N)$ is a linearly valid conditional instrumental set implies that $(Z\cup N,W)$ fulfills Condition (i) of Theorem \ref{theorem:graphical valid CIS}. Similarly, the fact that $(Z,W)$ is a linearly valid conditional instrumental set, implies that $(Z\cup N,W)$ also fulfills Condition (ii). Therefore $(Z\cup N,W)$ not being a linearly valid conditional instrumental set, requires a violation of Condition (iii), i.e., a path $p$ from $N$ to $Y$ in $\tilde{\g}$ that is open given $W$.

	We now show that $N \perp_{\g} Z \mid W$ by contradiction, so suppose that $N \not\perp_{\g} Z \mid W$, i.e., there exists a path $p'$ from some node $A \in Z$ to $N$ open given $W$ in $\g$. By Lemma \ref{lemma:always blocked}, $p'$ may not contain the edge $X \rightarrow Y$ and be open given $W$. Therefore $p'$ is also a path in $\tilde{\g}$. Further, $\de(M,\g) \setminus \de(M,\tilde{\g})=Y$ for any node $M$ and as a result $p'$ must also be open given $W$ in $\tilde{\g}$. We have already shown that there exists a path $p$ from $N$ to $Y$ open given $W$ in $\tilde{\g}$. Let $I$ be the first node on $p'$ where these two paths intersect and consider the concatenated path $q = p'(A,I) \oplus p(I,Y)$. We will now sequentially discuss the cases that $I=A,I=Y,I=N,I\neq N$ is a non-collider on $q$ and $I\neq N$ is a collider on $q$. We will show that in all these cases $q$ is open given $W$ or $W \cup  N$ and that the existence of the paths $p$ and $p'$ therefore contradicts our assumption that both $(Z,W)$ and $(Z,W\cup N)$ are linearly valid conditional instrumental sets.

If $I=A$ or $I=Y$, then $q$ is a subpath of $p'$ or $p$, respectively, and as a result is trivially open given $W$. Suppose now that $I=N$. Then if $N$ is a collider, $q$ is open given $W \cup N$. If not, $q$ is open given $W$. Suppose now that $I\neq N$ is a non-collider on $q$. Then it also has to be a non-collider on either $p$ or $p'$. Therefore, $I \notin W$ and $q$ is open given $W$. Suppose finally that $I\neq N$ is a collider on $q$. Then $q$ is open given $W$ if $\de(I, \g) \cap W \neq \emptyset$. If $\de(I, \g) \cap W = \emptyset$, on the other hand, then $p(I,N)$ must be of the form $I \rightarrow \dots \rightarrow N$. Otherwise some node in $\de(I,\g)$ would have to be a collider $C$ on $p(I,N)$ with $\de(C,\g) \subset \de(I,\g)$, which would contradict our assumption that $p$ is open given $W$. Thus, $N \in \de(I,\g)$ and $q$ is an open path given $W \cup  N$ in this case.

\end{proof}

\begin{lemma}\label{lemma:always blocked}
    Consider a node $X$ in an acyclic directed mixed graph $\g$ such that $\de(X,\g)=X$. Let $W$ be a node set in $\g$ such that $X \notin W$. Then any path in $\g$ that contains $X$ as a non-endpoint node is blocked by $W$.
\end{lemma}

\begin{proof}
    Consider a path $p$ in $\g$ that contains $X$ as a non-endpoint node. Since $\de(X,\g)=X$ there is no edge of the form $X \rightarrow N$ in $\g$. Therefore, $X$ must be a collider on $p$. Since, $W \cap \de(X,\g) = \emptyset$ it follows that $p$ is blocked by $W$.
\end{proof}

\subsection{Proof of soundness of Algorithm \ref{algorithm}}
\label{sec:algo prelim}

\begin{proposition}\label{proposition:algo}
Consider nodes $X$ and $Y$ in an acyclic directed mixed graph $\g$ such that $\de(X,\g) = \{X,Y\}$. Let $(Z,W)$ be a linearly valid conditional instrumental set relative to $(X,Y)$ in $\g$. Then applying Algorithm \ref{algorithm} to $(\g,X,Y,Z,W)$ minimizes $a.var(\hat{\tau}_{yx}^{z'.w'})$ at each step of either phase. Further, the first phase of Algorithm \ref{algorithm} is order independent.
\end{proposition}

\begin{proof}
Consider a step of the first phase of Algorithm \ref{algorithm} given a linearly valid tuple $(Z',W')$ and node $N$. Suppose $(Z' \cup \{N\},W'\setminus \{N\})$ is not valid. If $(Z',W'\setminus \{N\})$ is valid, then it follows by Proposition \ref{prop:moreW} that $(Z',W')$ is more efficient and therefore continuing with $(Z',W')$ greedily minimizes the asymptotic variance. Suppose $(Z' \cup \{N\},W'\setminus \{N\})$ is valid. Then $(Z' \cup \{N\},W'\setminus \{N\})$ is more efficient than $(Z',W')$ by Corollary \ref{corollary:always instrument} and than $(Z',W'\setminus \{N\})$, if it is valid, by Corollary \ref{corollary:more instruments}. Therefore adding $N$ to $Z'$ greedily minimizes the variance. Finally, $(Z' \cup \{N\},W'\setminus \{N\})$ is valid if and only if $N \perp_{\tilde{G}} Y \mid W'$. By the weak union property of d-separation it follows that the first phase of Algorithm \ref{algorithm} is order independent and always returns $(Z \cup S, W \setminus S)$ where $S \subseteq W$ is the largest set such that $S \perp_{\tilde{G}} Y \mid W$ \citep[cf. Lemma D.2.][]{henckel2019graphical}.

Consider a step of Algorithm \ref{algorithm} given a linearly valid tuple $(Z',W')$ and node $N$. We will now show that in all possible cases Algorithm \ref{algorithm} uses $N$ in a manner that greedily minimizes the asymptotic variance. Suppose that $(Z'\cup N,W')$ is valid and $N \not\perp_{\g} X \mid W' \cup Z'$. Then by Corollary \ref{corollary:more instruments}, $(Z'\cup N,W')$ is more efficient than $(Z',W')$. Further, if $(Z',W'\cup N)$ is valid, it is less efficient than $(Z'\cup N,W')$ by Corollary \ref{corollary:always instrument}. Therefore, adding $N$ to $Z'$ greedily minimizes the asymptotic variance. Suppose now that $(Z'\cup N,W')$ is not valid, $(Z',W'\cup N)$ is valid and $N \not\perp_{\tilde{\g}} Y \mid W'$. In this case, $(Z',W'\cup N)$ is more efficient than $(Z',W')$ by Proposition \ref{prop:moreW}. Therefore, adding $N$ to $W'$ greedily minimizes the asymptotic variance. Suppose now that $(Z'\cup N,W')$ is valid, $N \perp_{\g} X \mid W' \cup Z'$, $(Z',W'\cup N)$ is valid and $N \not\perp_{\tilde{\g}} Y \mid W'$. Since $(Z\cup N,W')$ being valid requires that $N \perp_{\tilde{\g}} Y \mid Z' \cup W'$ this case cannot occur by the contraction property of d-separation.

It remains to show that in all other cases, discarding $N$ greedily minimizes the asymptotic variance. Suppose that neither $(Z'\cup N,W')$ nor $(Z',W'\cup N)$ are valid. Then it is trivial that discarding $N$ greedily minimizes the asymptotic variance. Suppose that $(Z'\cup N,W')$ is not valid, $(Z',W'\cup N)$ is valid and $N \perp_{\tilde{\g}} Y \mid W'$. The fact that $(Z'\cup N,W')$ is not valid and $(Z',W')$ is, implies that $N \not\perp_{\tilde{\g}} Y \mid W'$ and therefore this case cannot occur. Suppose that $(Z'\cup N,W')$ is valid, $N \perp X \mid W' \cup Z'$ and $(Z',W'\cup N)$ is not valid. We can apply Theorem \ref{theorem:CIVcomparison} with $(Z_1,W_1)=(Z'\cup N,W')$ and $(Z_2,W_2)=(Z',W')$, to conclude that $(Z',W')$ is as efficient as $(Z'\cup N,W')$. Therefore, discarding $N$ greedily minimizes the asymptotic variance. Suppose that $(Z'\cup N,W')$ is valid, $N \perp X \mid W' \cup Z', (Z',W'\cup N)$ is valid and $N \perp_{\tilde{\g}} Y \mid W'$. By Corollary \ref{corollary:always instrument}, $(Z' \cup N,W')$ is more efficient than $(Z',W'\cup N)$. We can apply Theorem \ref{theorem:CIVcomparison} with $(Z_1,W_1)=(Z'\cup N,W')$ and $(Z_2,W_2)=(Z',W')$, to conclude that $(Z',W')$ is as efficient as $(Z'\cup N,W')$. Therefore, discarding $N$ greedily minimizes the asymptotic variance and this concludes all possible cases.
\end{proof}

\subsection{Proof of Theorem \ref{theorem:optimalCIS}}

\begin{theorem}
	Consider nodes $X$ and $Y$ in an acyclic directed mixed graph $\g$ such that $\de(X,\g) = \{X,Y\}$. Let $W^o= \mathrm{dis}^+_{X}(Y,\g) \setminus \{X,Y\}$ and $Z^o=\mathrm{dis}^+_Y(X,\g) \setminus (\{X,Y\} \cup W^o)$. Then the following two statements hold: (i) if $Z^o \neq \emptyset$  then $(Z^o, W^o)$ is a linearly valid conditional instrumental set relative to $(X,Y)$ in $\g$; (ii) if $Z^o \cap \{\pa(X,\g) \cup \mathrm{sib}(X,\g)\} \neq \emptyset$ then $(Z^o, W^o)$ is also graphically optimal relative to $(X,Y)$ in $\g$.
\end{theorem}

\begin{proof}
	We first show that if $Z^o\neq \emptyset$ then $(Z^o,W^o)$ is a linearly valid conditional instrumental set. By Lemma 
	\ref{lemma:Wo max info}, $Z^o \perp_{\tilde{\g}} Y \mid W^o$. Further, by construction of $Z^o$ and our assumption that $Z^o=\emptyset$ there exists a path $p$ from any node in $Z^o$ to $X$ consisting entirely of colliders in $Z^o\cup W^o$. Hence, $Z^o \not\perp_{\g} X \mid W^o$.
	
	We now consider the second statement of the theorem. The claim that $(Z^o,W^o)$ is graphically optimal is equivalent to saying that for all linearly valid conditional instrumental sets $(Z,W)$ relative to $(X,Y)$ in $\g$ such that there exists a linear structural equation model $\mathcal{M}_1$ compatible with $\g$ for which $a.var(\hat{\tau}_{yx}^{z.w},\mathcal{M}_1) < a.var(\hat{\tau}_{yx}^{z^o.w^o},\mathcal{M}_1)$, there also exists a linear structural equation model $\mathcal{M}_2$ compatible with $\g$ for which $a.var(\hat{\tau}_{yx}^{z.w},\mathcal{M}_2) > a.var(\hat{\tau}_{yx}^{z^o.w^o},\mathcal{M}_2)$, where $a.var(\hat{\tau}_{yx}^{z.w},\mathcal{M})$ denotes the asymptotic variance of the estimator $\hat{\tau}_{yx}^{z.w}$ in the linear structural equation model $\mathcal{M}$. We will now show this characterization of graphical optimality holds for $(Z^o,W^o)$.

 First we will show that for any linearly valid conditional instrumental set $(Z,W)$ such that $W^o \setminus W =\emptyset$, we can use Theorem \ref{theorem:CIVcomparison} with $(Z_1,W_1)=(Z,W)$ and $(Z_2,W_2)=(Z^o,W^o)$ to conclude that 
	\[
	a.var(\hat{\tau}_{yx}^{z^o.w^o}) \leq a.var(\hat{\tau}_{yx}^{z.w}).
	\]
	By Lemma \ref{lemma:Wo max info}, $W \setminus W^o \perp_{\tilde{\g}} Y \mid W^o$ and as a result Condition \eqref{condition1} of Theorem \ref{theorem:CIVcomparison} holds. By Lemma \ref{lemma:WoZo max info}, $W \setminus (W^o \cup Z^o) \perp_{\g} X \mid Z^o \cup W^o$ and therefore Condition \eqref{condition2} of Theorem \ref{theorem:CIVcomparison} holds. By Lemma \ref{lemma:WoZo max info}, $W \setminus (W^o \cup Z^o) \perp_{\g} X \mid Z^o \cup W^o$ and $Z \setminus Z^o \perp_{\g} X \mid Z^o \cup W^o$, where we use that $Z \cap W^o=\emptyset$ since $Z \cap W=\emptyset$. By the weak union property it follows that $Z \setminus Z^o \perp_{\g} X \mid Z^o \cup W$ and therefore Condition \eqref{condition4} of Theorem \ref{theorem:CIVcomparison} holds. Finally, Condition \eqref{condition3} is void by the fact that $W^o\setminus W = \emptyset$. We can therefore suppose that $W^o\setminus W \neq \emptyset$ for the remainder of the proof.
	
 Let $(Z,W)$ be a linearly valid conditional instrumental set in $\g$, such that $W^o \setminus W \neq \emptyset$ and consider the ratio of asymptotic variances
 \[
\frac{a.var(\hat{\tau}_{yx}^{z^o.w^o})}{a.var(\hat{\tau}_{yx}^{z.w})}
=\frac{\sigma_{\tilde{y}\tilde{y}.w^o}(\sigma_{xx.w} -\sigma_{xx.wz})}{\sigma_{\tilde{y}\tilde{y}.w}(\sigma_{xx.w^o} -\sigma_{xx.w^oz^o})}.
 \]
We will now construct a Gaussian linear structural equation model such that this ratio is smaller than 1. To keep our notation readable we make the dependence of our conditional variances on the underlying Gaussian linear structural equation model explicit only when necessary.

Let $F'$ be the set of all edges that are the first on any proper path from $W^o$ to $X$ that is open given $Z^o$. 
Similarly, let $F''$ be the set of all edges that are first on any proper path from $W^o$ to $X$ that is open given the empty set. 
We then define the following class of Gaussian linear structural equation models. Let $(\mathcal{A},\Omega)$ be a Gaussian linear structural equation model compatible with $\g$ such that$\Omega$ is strictly diagonally dominant. Let $(\mathcal{A}_{F}(\epsilon),\Omega_{F}(\epsilon))$ be the Gaussian linear structural equation model, such that the edge coefficients and error covariances corresponding to edges in $F=F'\cup F''$ are replaced with the value $\epsilon > 0$. 
Recall that by Lemma \ref{lemma:convergence} the limit for any conditional covariance as a function of $\epsilon$, with $\epsilon \rightarrow 0$ exists and is the corresponding conditional covariance in the model $(\mathcal{A}_{F}(0),\Omega_{F}(0))$. Clearly, $(\mathcal{A}_{F}(0),\Omega_{F}(0))$ is also compatible with the truncated acyclic directed mixed graph $\mathcal{G}'$ with the edges in $F$ removed and we can therefore invoke additional m-separation statements as implied by $\mathcal{G}'$.
Choose $(\mathcal{A},\Omega)$ such that the model is faithful to $\g$, the truncated model with $\tilde{Y}$ instead of $Y$ to $\tilde{\g}$ and the same is true for the model $(\mathcal{A}_{F}(0),\Omega_{F}(0))$ with respect to $\g'$ and $\tilde{\g}'$. Such a model exists by the fact that a faithfulness violation corresponds to a non-trivial polynomial in the non-zero parameters of $(\mathcal{A},\Omega)$ evaluating to zero \citep{spirtes2000causation}. As a long as we impose a finite number of faithfulness constraints the resulting set of parameters for which the resulting model violates one of our constraints therefore has Lebesgue measure $0$.

 We first make two observations that hold for any Gaussian linear structural equation model compatible with $\mathcal{G}$. Let $A = W \setminus W^o$ and $C = W^o \setminus W$. By Lemma \ref{lemma:Wo max info}, $\sigma_{\tilde{y}\tilde{y}.w^o}=\sigma_{\tilde{y}\tilde{y}.aw^o}=\sigma_{\tilde{y}\tilde{y}.wc}$. Therefore,
\begin{equation}
0 < \frac{\sigma_{\tilde{y}\tilde{y}.w^o}}{\sigma_{\tilde{y}\tilde{y}.w}}=\frac{\sigma_{\tilde{y}\tilde{y}.wc}}{\sigma_{\tilde{y}\tilde{y}.w}}
=\frac{\sigma_{\tilde{y}\tilde{y}.w}-\sigma_{\tilde{y}c.w} \Sigma^{-1}_{cc.w} \sigma^\top_{\tilde{y}c.w}}{\sigma_{\tilde{y}\tilde{y}.w}} = 1 - q, \quad q \geq 0.
\label{eq:residual variance}
\end{equation}
Second, 
\begin{equation}
\frac{\sigma_{xx.w} -\sigma_{xx.wz}}{\sigma_{xx.w^o} -\sigma_{xx.w^oz^o}}
\leq \frac{\sigma_{xx.w} -\sigma_{xx.w'z'z^o}}{\sigma_{xx.w^o} -\sigma_{xx.w^oz^o}}
\leq \frac{\sigma_{xx} -\sigma_{xx.w'z'z^o}}{\sigma_{xx.w^o} -\sigma_{xx.w^oz^o}},
\label{eq:instrumental strength}
\end{equation}
with $W'=W \setminus Z^o$ and $Z'=Z \setminus Z^o$.

By assumption on $Z^o$, there exists a path from some node in $Z^o$ to $X$ consisting entirely of colliders in $Z^o$. Clearly no edge in this path is adjacent to any node in $W^o$ and therefore no edge is in $F$. Thus, $Z^o \not\perp_{\mathcal{G}'} X \mid W^o$. By construction $(\mathcal{A}_{F}(0),\Omega_{F}(0))$ is faithful to $\mathcal{G}'$. Therefore, 
\begin{align*}
\lim_{\epsilon \rightarrow 0} \sigma_{xx.w^o}\{\mathcal{A}_{F}(\epsilon),\Omega_{F}(\epsilon)\} &- \sigma_{xx.w^oz^o}\{\mathcal{A}_{F}(\epsilon),\Omega_{F}(\epsilon)\} \\  &=\sigma_{xx.w^o}\{\mathcal{A}_{F}(0),\Omega_{F}(0)\}- \sigma_{xx.w^oz^o}\{\mathcal{A}_{F}(0),\Omega_{F}(0)\}> 0.
\end{align*}
This allows us to conclude that 
\begin{align*}
&\lim_{\epsilon \rightarrow 0} \frac{a.var[\hat{\tau}_{yx}^{z^o.w^o}\{\mathcal{A}_{F}(\epsilon),\Omega_{F}(0)\}]}{a.var[\hat{\tau}_{yx}^{z.w}(\{\mathcal{A}_{F}(\epsilon),\Omega_{F}(\epsilon)\}]}
=  \\
& \quad \quad \quad \frac{\sigma_{\tilde{y}\tilde{y}.w^o}\{\mathcal{A}_{F}(0),\Omega_{F}(0)\}}{  \sigma_{\tilde{y}\tilde{y}.w}\{\mathcal{A}_{F}(0),\Omega_{F}(0)\}}
\frac{[\sigma_{xx.w}\{\mathcal{A}_{F}(0),\Omega_{F}(0)\} -\sigma_{xx.wz}\{\mathcal{A}_{F}(0),\Omega_{F}(0)\}]}{[\sigma_{xx.w^o}\{\mathcal{A}_{F}(0),\Omega_{F}(0)\} -\sigma_{xx.w^oz^o}\{\mathcal{A}_{F}(0),\Omega_{F}(0)\}]}.
\end{align*}
By a similar argument $\sigma_{\tilde{y}c.w}\{\mathcal{A}_{F}(0),\Omega_{F}(0)\}  \neq 0$ and therefore by Equation \eqref{eq:residual variance} 
\begin{align*}
\frac{\sigma_{\tilde{y}\tilde{y}.w^o}\{\mathcal{A}_{F}(0),\Omega_{F}(0)\}}{  \sigma_{\tilde{y}\tilde{y}.w}\{\mathcal{A}_{F}(0),\Omega_{F}(0)\}} &= 1-\frac{ \sigma_{\tilde{y}c.w}\{\mathcal{A}_{F}(0),\Omega_{F}(0)\} \Sigma^{-1}_{cc.w}\{\mathcal{A}_{F}(0),\Omega_{F}(0)\} \sigma^\top_{\tilde{y}c.w}\{\mathcal{A}_{F}(0),\Omega_{F}(0)\}}{\sigma_{\tilde{y}\tilde{y}.w}\{\mathcal{A}_{F}(0),\Omega_{F}(0)\}} \\ 
&= 1-q^0 < 1.
\end{align*}
 Combined with Equation \eqref{eq:instrumental strength} we have
\begin{align}
\lim_{\epsilon \rightarrow 0} \frac{a.var(\hat{\tau}_{yx}^{z^o.w^o})}{a.var(\hat{\tau}_{yx}^{z.w})}
&= (1-q^0)
\frac{\sigma_{xx.w}\{\mathcal{A}_{F}(0),\Omega_{F}(0)\} -\sigma_{xx.wz}\{\mathcal{A}_{F}(0),\Omega_{F}(0)\}}{\sigma_{xx.w^o}\{\mathcal{A}_{F}(0),\Omega_{F}(0)\} -\sigma_{xx.w^oz^o}\{\mathcal{A}_{F}(0),\Omega_{F}(0)\}}
\nonumber \\
&\leq (1-q^0)
\frac{\sigma_{xx}\{\mathcal{A}_{F}(0),\Omega_{F}(0)\} -\sigma_{xx.w'z'z^o}\{\mathcal{A}_{F}(0),\Omega_{F}(0)\}}{\sigma_{xx.w^o}\{\mathcal{A}_{F}(0),\Omega_{F}(0)\} -\sigma_{xx.w^oz^o}\{\mathcal{A}_{F}(0),\Omega_{F}(0)\}} \label{eq: fraction}.
\end{align}

We now show that by construction of $F$, $W^o \perp_{\mathcal{G}'} X$ and $S \perp_{\mathcal{G}'} X \mid Z^o$ for any $S \subseteq V \setminus (Z^o \cup \{X,Y\})$ in order to further simplify Equation \eqref{eq: fraction}.

Consider first the claim that $W^o \perp_{\mathcal{G}'} X$. Let $p$ be a path in $\g'$ from $W^o$ to $X$. Clearly, $p$ is also a path in $\g$ that does not contain any edges in $F$. By the fact that $F''$ contains all first edges on paths from $W^o$ to $X$ that are open given the empty set it follows that $p$ is blocked in $\g$. By Lemma \ref{lemma: removal}, it follows that $p$ is also blocked in $\g'$ and our claim follows.  

Consider now the claim that $S \perp_{\mathcal{G}'} X \mid Z^o$ for any $S \subseteq V \setminus (Z^o \cup \{X,Y\})$. 
To do so we show that any proper path from some node $N \in S$ to $X$ that is open given $Z^o$ in $\g$ contains an edge in $F$. This suffices by the arguments given in the previous paragraph.

Let $p$ be proper path from some node $N \in S$ to $X$ that is open given $Z^o$ in $\g$. By Lemma \ref{lemma:always blocked}, $p$ may not contain $Y$ and be open given $Z^o$, since $\de(Y,\g)=Y$ and $Y \notin Z^o$. Further, any other node adjacent to $X$ in $\g$ is either in $Z^o$ or $W^o$. If the latter is the case we are done, so suppose that $p$ contains a subsegment $p(Z,X)$ chosen to be of maximal possible length, such that every node in $p(Z,X)$, except for $X$ is in $Z^o$. Let $M$ be the node before $Z$ on $p$ and consider $p(M,X)$. By assumption, $p(M,X)$ may not contain a non-collider and thus must be of the form 
\[ 
M \rightarrow Z \leftrightarrow \dots \leftrightarrow X \text{ or } M \leftrightarrow Z \leftrightarrow \dots \leftrightarrow X.
\]
In either case, to not contradict our assumption that $p(Z,X)$ is of maximal length, $M \in W^o$ must hold. Thus the edge $M \rightarrow Z$, respectively $M \leftrightarrow Z$, is in $F''$ and our claim follows.

Since $W^o \perp_{\mathcal{G}'} X$, 
\[\lim_{\epsilon \rightarrow 0}\sigma_{xx.w^o}\{\mathcal{A}_{F}(\epsilon),\Omega_{F}(\epsilon)\} = \sigma_{xx.w^o}\{\mathcal{A}_{F}(0),\Omega_{F}(0)\}=\sigma_{xx}\{\mathcal{A}_{F}(0),\Omega_{F}(0)\}.\] 

Since $S \perp_{\mathcal{G}'} X \mid Z^o$ for any $S \subseteq V \setminus (Z^o \cup \{X,Y\})$,
\[\lim_{\epsilon \rightarrow 0}\sigma_{xx.w'z'z^o}\{\mathcal{A}_{F}(\epsilon),\Omega_{F}(\epsilon)\}=\sigma_{xx.w'z'z^o}\{\mathcal{A}_{F}(0),\Omega_{F}(0)\}=\sigma_{xx.z^o}\{\mathcal{A}_{F}(0),\Omega_{F}(0)\}\]
and
\[\lim_{\epsilon \rightarrow 0}\sigma_{xx.w^oz^o}\{\mathcal{A}_{F}(\epsilon),\Omega_{F}(\epsilon)\}= \sigma_{xx.w^oz^o}\{\mathcal{A}_{F}(0),\Omega_{F}(0)\}=\sigma_{xx.z^o}\{\mathcal{A}_{F}(0),\Omega_{F}(0)\}.\]

Using these limit statements, we can conclude that
\begin{align*}
\lim_{\epsilon \rightarrow 0} \frac{a.var[\hat{\tau}_{yx}^{z^o.w^o}\{\mathcal{A}_{F}(\epsilon),\Omega_{F}(\epsilon)\}]}{a.var[\hat{\tau}_{yx}^{z.w}\{\mathcal{A}_{F}(\epsilon),\Omega_{F}(\epsilon)\}]}
&\leq (1-q^0) < 1.
\end{align*}

Therefore, there exists an $\epsilon > 0$, such that the Gaussian linear structural equation model $(\mathcal{A}_{F}(\epsilon),\Omega_{F}(\epsilon))$ is compatible with $\g$ and the asymptotic variance provided by $(Z,W)$ is larger than the one provided by $(Z^o,W^o)$.
\end{proof}

\begin{lemma}
	Consider nodes $X$ and $Y$ in an acyclic directed mixed graph $\g$ with node set $V$ such that $\de(X,\g) =\{X,Y\}$. Let $W^o = \mathrm{dis}^+_{X}(Y,\g) \setminus \{X,Y\}$, then for any $S \subseteq V \setminus (W^o \cup \{X,Y\})$,
	\[
	S \perp_{\tilde{\g}} Y \mid W^o.
	\]
	\label{lemma:Wo max info}
\end{lemma}

\begin{proof}
	Let $p$ be a proper path from some node $A \in S$ to $Y$ in $\tilde{\g}$.  We now show that $p$ is blocked by $W^o$.
 
    By Lemma \ref{lemma:always blocked}, $p$ is blocked if it contains $X$ so assume this is not the case. Every node adjacent to $Y$ in $\tilde{\g}$ is either in $W^o$ or $X$. Therefore $p$ contains at least one node $W \in W^o$, such that $p(W,Y)$ consists, except for $Y$, of nodes in $W^o$. Choose $W \in W^o$, such that $p(W,Y)$ is of maximal possible length. Let $N \notin W^o$ be the node adjacent to $W$ on $p$. We will now show that $p(N,Y)$ contains at least one non-collider by contradiction. Suppose that all nodes on $p(N,Y)$ are colliders, i.e., $p(N,Y)$ is of the form
	\[
	N \leftrightarrow W \leftrightarrow \dots \leftrightarrow Y  \text{ or } 
	N \rightarrow W \leftrightarrow \dots \leftrightarrow Y.
	\]
  Since $p$ does not contain $X$, $N \neq X$. By the existence of $p(N,Y)$ it follows that $N \in W^o$ which contradicts our assumption that $p(W,Y)$ is maximal. Therefore, $p(W,Y)$ contains a non-collider. As all nodes on $p(W,Y)$ are in $W^o$, it follows that $p(W,Y)$ and therefore $p$ are blocked by $W^o$. Since this is true for any proper path it follows that $S \perp_{\tilde{\g}} Y \mid W^o$. 
\end{proof}

\begin{lemma}
	Consider nodes $X$ and $Y$ in an acyclic directed mixed graph $\g$ with node set $V$ such that $\de(X,\g)=\{X,Y\}$. Let $W^o = \mathrm{dis}^+_{X}(Y,\g)\setminus \{X,Y\}$ and $Z^o=\mathrm{dis}^+_Y(X,\g) \setminus (\{X,Y\} \cup W^o)$, then for any $S \subseteq V \setminus (W^o \cup Z^o \cup \{X,Y\})$
	\[
	S \perp_{\g} X \mid W^o \cup Z^o.
	\]
	\label{lemma:WoZo max info}
\end{lemma}

\begin{proof}
Follows by the same arguments as given for Lemma \ref{lemma:Wo max info}.
\end{proof}

\begin{lemma}\label{lemma: removal}
Consider a node set $W$ and a path $p$ in an acyclic directed mixed graph $\g$, such that $p$ is blocked by $W$ in $\g$. Let $\g'$ be a graph obtained by removing edges from $\g'$ such that $p$ is also a path in $\g'$. Then $p$ is blocked by $W$ in $\g'$. 
\end{lemma}

\begin{proof}
Since $p$ is blocked by $W$ in $\g$, it must (i) contain a non-collider that is in $W$ or (ii) a collider $C$ such that $W \cap \de(C,\g) =\emptyset$. If (i) holds in $\g$ then it clearly also holds in $\g'$. The same is true for (ii), by the fact that $\de(C,\g') \subseteq \de(C,\g)$.
\end{proof}

\section{Additional examples and results}

\subsection{Additional example for Section \ref{sec:prep results}}

\begin{example}[Linearly valid conditional instrumental sets]
	\label{ex:CIS}
	To illustrate Theorem \ref{theorem:graphical valid CIS} and Lemma \ref{prop:non forb} we consider the graph $\g$ from Fig. \ref{fig:CIS example}. We are interested in classifying all linearly valid conditional instrumental sets relative to $(V_4,V_6)$ in $\g$. By Condition (i) of Theorem \ref{theorem:graphical valid CIS}, no linearly valid conditional instrumental set $(Z,W)$ may contain any nodes in $\mathrm{forb}(V_4,V_6,\g)=\{V_4,V_5,V_6,V_7\}$. Therefore, we only need to consider sets $Z,W \subseteq \{V_1,V_2,V_3\}$. 
	For Condition (ii) to hold, there must exist a path from $Z$ to $X$ that is open given $W$ in $\g$. For any $V_i$ with $i\in\{1,2,3\}$ there exist three paths to $X$ in $\g$: the path ending with the edge $V_1 \rightarrow V_4$, the one ending with the edge $V_5 \leftarrow V_4$ and the one ending with the edge $V_5 \leftrightarrow V_6$. The latter two always include $V_6$ as a collider and as $\de(V_6,\g)=\{V_6,V_7\}$ this implies that they are both always blocked by any $W \subset \{V_1,V_2,V_3\}$. The first path on the other hand does not contain a collider and is therefore only closed if there exists a node $V_j \in W$ with $j< i$. Therefore, Condition (ii) holds if and only if there exists an $i \in \{1,2,3\}$, such that $V_i \in Z$ and for all $j < i, V_j \notin W$. Finally, for Condition (iii) to hold there must not be a path from $Z$ to $Y$ that is open given $W$ in the graph $\tilde{\g}$, which is the graph $\g$ but with the edge $V_4 \rightarrow V_5$ removed. For any $V_i$ with $i\in\{1,2,3\}$ there exist two paths to $Y$ in $\tilde{\g}$: the path ending with the edge $V_3 \rightarrow V_6$ and the one ending with the edge $V_4 \leftrightarrow V_6$. The latter always includes $V_4$ as a collider and as $\de(V_4,\tilde{\g})=\emptyset$ it is always blocked by any $W \in \{V_1,V_2,V_3\}$. The first path on the other hand does not contain a collider and is therefore only closed if there exists a node $V_j \in W$ with $j > i$. Therefore, Condition (iii) holds if and only if for all $i \in \{1,2,3\}$, such that $V_i \in Z$ there exists a $j > i$ such that $V_j \in W$. 
	It follows that there are five possible linearly valid tuples relative to $(V_4,V_6)$ in $\g$: $(V_1,V_2)$, $(V_1,V_3)$, $(\{V_1,V_2\},V_3)$, $(V_1,\{V_2,V_3\})$ and $(V_2,V_3)$. 
\end{example}

\subsection{Additional results for Section \ref{subsec:asymptotic variance}}
\label{sec:asy formula app}

The asymptotic variance formula for the two-stage least squares estimator from Theorem \ref{theorem:varequ} elegantly mirrors the asymptotic variance formula for the ordinary least squares estimator, which is
\begin{equation*}
a.var(\hat{\beta}_{yx.w})=\frac{\sigma_{yy.xw}}{\sigma_{xx.w}}.
\end{equation*}
We can use this to derive the following simple proof for the known result that, in cases where it is possible to do either, it is more efficient to estimate a total effect with the ordinary least squares estimator than with the two-stage least squares estimator. Interestingly, the proof shows that this gain in efficiency depends primarily on the instrumental strength $\sigma_{xx.w}- \sigma_{xx.zw}$.

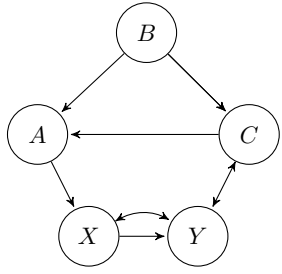
\begin{figure}[t]
	\centering
		\begin{tikzpicture}[>=stealth',shorten >=1pt,auto,node distance=0.7cm,scale=.9, transform shape,align=center,minimum size=3em]
		\node[state] (x) at (0,0) {$X$};
		\node[state] (y) [right =of x] {$Y$}; 
		\node[state] (w2) at ($(y)+(.75,1.5)$)   {$C$};  
		\node[state] (w3) at ($(w2)+(-1.5,1.5)$)   {$B$};  
		\node[state] (w4) at ($(x)+(-.75,1.5)$)   {$A$};  
		\path[->]   (x) edge    (y);
		\path[<->]   (x) edge  [bend left]  (y);
		\path[->]   (w4) edge     (x);
		\path[<->]   (w2) edge     (y);
		\path[->]   (w3) edge     (w4);
		\path[->]   (w3) edge     (w2);
		\path[->]   (w3) edge     (w2);
		\path[->]   (w2) edge     (w4);
		\end{tikzpicture}
	\caption{Acyclic directed mixed graph for Example \ref{ex:all opt}.}
	\label{fig:all opt}
\end{figure}

\begin{lemma} 
	\label{lemma:OLS better}
\cite[e.g. Chapter 5.2.3][]{wooldridge2010econometric}
	Let $X,Y$ be nodes in an ADMG $\g=(V,E)$, such that $V$ is generated from a linear structural equation model compatible with $\g$. Let $(Z,W)$ be a conditional instrumental set relative to $(X,Y)$ in $\g$ and let $W$ be an adjustment set relative to $(X,Y)$ in $\g$. Then 
	\[
	a.var(\hat{\tau}_{yx}^{z.w}) \geq 	a.var(\hat{\beta}_{yx.w}).
	\]
\end{lemma}

\begin{proof}
	By the assumption that $W$ is a valid adjustment set it follows that $\beta_{yx.w}=\tau_{yx}$. Therefore, $\beta_{\tilde{y}w}=\beta_{yw}-\tau_{yx}\beta_{xw}=\beta_{yw.x}$ by Lemma \ref{lemma:Cochran}. As a results $\sigma_{\tilde{y}\tilde{y}.w} = \sigma_{yy.xw}$. Further, by the non-negativity of conditional variances $\sigma_{xx.w}\geq \sigma_{xx.w}- \sigma_{xx.zw}$. Combined this implies that
	\[
	\frac{a.var(\hat{\beta}_{yx.w})}{a.var(\hat{\tau}_{yx}^{z.w})} = 
	\frac{\sigma_{xx.w}-\sigma_{xx.zw}}{\sigma_{xx.w}} = 1 - \frac{\sigma_{xx.zw}}{\sigma_{xx.w}} \leq 1.
	\]
	
\end{proof}

\subsection{Additional examples for Section \ref{subsec:efficiency}}
\label{sec: ex+}

  \begin{table}[t]
  \begin{center}
		\def~{\hphantom{0}}
		\caption{List of all linearly valid conditional instrumental sets in the graph from Fig. \ref{subfig:weird forb} of the Main}
		\begin{tabular}{ll}
			\def~{\hphantom{0}}
			$Z$ & $W$ \\
			$A$ & $\{B,D\},\{C,D\},\{B,C,D\}$\\
			$B$ & $C,\{A,C\},\{C,D\},\{A,C,D\}$\\
			$D$ & $B,C,\{A,B\},\{A,C\},\{B,C\},\{A,B,C\}$\\
			$\{A,B\}$ & $C,\{C,D\}$\\
			$\{A,D\}$ & $B,C,\{B,C\}$\\
			$\{B,D\}$ & $C,\{A,C\}$ \\
			$\{A,B,D\}$ & $C$\\
		\end{tabular}
		\label{table:valid cis}     
  \end{center}
	\end{table}

\begin{example}[Neutral conditioning]\label{ex:neutral}
    Consider the graph $\g$ from Fig. \ref{subfig:simple} of the Main. We are interested in tuples of the form $(D,W)$ and $(D,W')$, with $W \subseteq \{A,C\}$ and $W' = W \cup B$. By Example \ref{ex:harmful} of the Main, any such tuple is a linearly valid conditional instrumental set relative to to $(X,Y)$ in $\g$. Further, as $B \perp_{\g} D \mid W$ we can apply Theorem \ref{2SLSavar} with $W_1=W,W_2=W'$ and $Z_1=Z_2=D$ (Condition \eqref{condition3} (i) holds, Conditions \eqref{condition1}, \eqref{condition2} and \eqref{condition4} are void). As $B \perp_{\tilde{\g}} Y \mid W$ and $B \perp_{\g} D \mid W$ we can also apply Theorem \ref{2SLSavar} with the roles of $W$ and $W'$ reversed, i.e., $W_1=W',W_2=W$ and $Z_1=Z_2=D$ (Conditions \eqref{condition1} and \eqref{condition2} (i) hold, Conditions \eqref{condition3} and \eqref{condition4} are void). We can therefore conclude that adding $B$ to any $W \subseteq \{A,C\}$ has no effect on the asymptotic variance. 
    
    The reason conditioning on $B$ has no effect, is that $B \ci \tilde{Y} \mid W$ and $B \ci D \mid W$. Therefore, $\sigma_{\tilde{y}\tilde{y}.wb}=\sigma_{\tilde{y}\tilde{y}.w}$ and $\sigma_{xx.w}-\sigma_{xx.wd}=\sigma_{xx.wb}-\sigma_{xx.wbd}$ (see Lemma \ref{lemma:equalinstruments}). Conditioning on $B$ has no effect on either the residual variance or the conditional instrumental strength.
\end{example}

\begin{example}[Beneficial conditioning 2]
\label{ex: beneficial 2}
        Consider the graph $\g$ from Fig. \ref{subfig:weird} in the Main and the two tuples $(\{B,D\},C)$ and $(\{B,D\},\{A,C\})$. By Example \ref{ex:valid CIS}, both tuples are linearly valid relative to $(X,Y)$ in $\g$. Further, as $A \perp_{\g} X \mid  C$ we can apply Theorem \ref{2SLSavar} with $W_1=C,W_2=\{A,C\}$ and $Z_1=Z_2=\{B,D\}$ (Condition \eqref{condition3} (ii) holds, Conditions \eqref{condition1},\eqref{condition2} and \eqref{condition4} are void). We can therefore conclude that conditioning on $A$ can only improve the asymptotic variance. 
    
    The reason conditioning on $A$ is beneficial, is that $A \ci \tilde{Y} \mid C$ and $A \ci X \mid C$ but $A \notci X \mid \{B,C,D\}$. Therefore, $\sigma_{\tilde{y}\tilde{y}.ac} = \sigma_{\tilde{y}\tilde{y}.c}, \sigma_{xx.ac}=\sigma_{xx.c}$ but $\sigma_{xx.abcd} \leq \sigma_{xx.bcd}$. Interestingly, this also implies that $(\{B,D\},\{A,C\})$ has the exact same residual variance and conditional instrumental strength as the linearly valid tuple $(\{A,B,D\},C)$: $\sigma_{\tilde{y}\tilde{y}.c}$ and $\sigma_{xx.c} - \sigma_{xx.abcd}$, respectively. Conditioning on the covariate $A$ is therefore beneficial because adding it to the conditioning set has the same effect as adding it to the instrumental set and the latter is beneficial by Corollary \ref{corollary:more instruments}.
\end{example}

\begin{example}(Optimal tuple in Fig. \ref{subfig:simple} of the Main) \label{ex:optimal 2}
	Consider the graph $\g$ in Fig. \ref{subfig:simple} of the Main. We first characterize all linearly valid conditional instrumental sets relative to $(X,Y)$ in $\g$. As $\f{\g}=\{X,Y\}$, we only need to consider disjoint node sets $Z$ and $W$ that are subsets of $\{A,B,C,D\}$. Further, $Z \perp_{\tilde{\g}} Y \mid W$ if and only if $C \notin Z$ and therefore $Z \subseteq \{A,B,D\}$. Finally, $Z \perp_{\g} X \mid  W$ if and only if $\{B,D\} \cap Z = \emptyset$ and $D \in W$.  Therefore, any linearly valid tuple with respect to $(X,Y)$ in $\g$ has to be of the form $(Z,W)$, with $Z \subseteq \{A,B,D\}$ non-empty, $W \subseteq \{A,B,C,D\} \setminus Z$ if $Z \neq \{A\}$, and $W \subseteq \{B,C\}$ if $Z = \{A\}$. There are $34$ tuples of this form. 
	
	Let $(Z,W)$ be any of these $34$ linearly valid tuples. Suppose that $Z\cup W \neq \{A,B,C,D\}$. By the same argument as in Example \ref{ex:weird} we can use Proposition \ref{prop:moreW} and Corollary \ref{corollary:more instruments} to obtain a more efficient tuple $(Z',W')$, such that $Z' \cup W' = \{A,B,C,D\}$. We can also apply Corollary \ref{corollary:always instrument} with $S = \{A,B,D\} \setminus Z'$ and conclude that the linearly valid set $(\{A,B,D\},C)$ is at least as efficient as $(Z',W')$. Since we began from an arbitrary linearly valid tuple, this shows that $(\{A,B,D\},C)$ provides the smallest attainable asymptotic variance among all linearly valid tuples. However, as
	$A \perp_{\g} X \mid \{B,C,D\}$
	we can apply Theorem \ref{theorem:CIVcomparison} with $W_1 = W_2 = \{C\}$, $Z_1=\{A,B,D\}$ and $Z_2=\{B,D\}$ to conclude that $(\{B,D\},C)$ is at least as efficient as $(\{A,B,D\},C)$. Therefore, the smaller tuple $(\{B,D\},C)$ is also asymptotically optimal. 
\end{example}

\subsection{Additional example for Section \ref{subsec:CIVselection}}
\label{sec:weird optimal cis}

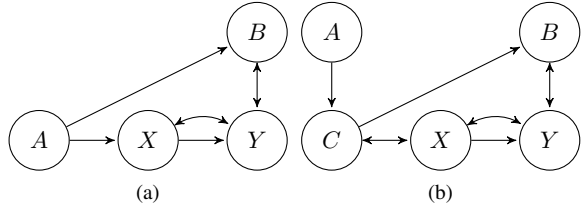
\begin{figure}[t]
	\centering
	\subfloat[\label{subfig:Zo empty}]{
		\begin{tikzpicture}[>=stealth',shorten >=1pt,auto,node distance=0.7cm,scale=.9, transform shape,align=center,minimum size=3em]
		\node[state] (x) at (0,0) {$X$};
		\node[state] (y) [right =of x] {$Y$};
		\node[state] (w1) [left =of x] {$A$};  
		\node[state] (w2) [above =of y] {$B$};  
		
		\path[->]   (x) edge    (y);
		\path[<->]   (x) edge  [bend left]  (y);
		\path[->]   (w1) edge     (x);
		\path[<->]   (w2) edge    (y);
		\path[->]   (w1) edge    (w2);;
		\end{tikzpicture}
	}
	\subfloat[\label{subfig:Zo not optimal}]{
		\begin{tikzpicture}[>=stealth',shorten >=1pt,auto,node distance=0.7cm,scale=.9, transform shape,align=center,minimum size=3em]
		\node[state] (x) at (0,0) {$X$};
		\node[state] (y) [right =of x] {$Y$};
		\node[state] (w1) [left =of x] {$C$};  
		\node[state] (w2) [above =of y] {$B$};  
		\node[state] (z) [above =of w1] {$A$}; 
		
		\path[->]   (x) edge    (y);
		\path[<->]   (x) edge  [bend left]  (y);
		\path[<->]   (w1) edge     (x);
		\path[<->]   (w2) edge    (y);
		\path[->]   (w1) edge    (w2);
		\path[->]   (z) edge    (w1);
		
		\end{tikzpicture}
	}
	\caption{\protect \subref{subfig:Zo empty} and \protect\subref{subfig:Zo not optimal} Acyclic directed mixed graphs for Example \ref{ex:optimal cis}.}
	\label{fig:Zo odd}
\end{figure}

    \begin{example}[All graphically optimal tuples]
        Consider the graph $\g$ from Fig. \ref{fig:all opt}. There are two linearly valid conditional instrumental sets with respect to $(X,Y)$ in $\g$: $(A,\{B,C\})$ and $(B,\emptyset)$. We consider two linear structural equation models compatible with $\g$: for model $\mathcal{M}_1$ let all error variances and edge coefficients be $1$, where we model the bi-directed edges with a latent variable, i.e, $X \leftarrow L \rightarrow Y$. For model $\mathcal{M}_2$ do the same, except for setting the edge coefficient for the edge $B \rightarrow A$ to $0.05$. For both models we computed the asymptotic variance corresponding to the two linearly valid tuples. The results are given in Table \ref{table:all opt} and they show that in $\mathcal{M}_1$ the most efficient tuple is $(A,\emptyset)$ while in $\mathcal{M}_2$ the most efficient tuple is $(A,\{B,C\})$. Therefore, all linearly valid tuples are graphically optimal.
        \label{ex:all opt}
    \end{example}

    \begin{table}[h]
    \begin{center}
\def~{\hphantom{0}}
		\caption{Asymptotic variances for Example \ref{ex:all opt}}
\begin{tabular}{rrr}
 & $(A,\{B,C\})$ & $(B,\emptyset)$ \\ 
$\mathcal{M}_1$ & 2.50 & 0.75 \\ 
  $\mathcal{M}_2$ & 2.50 & 2.72 \\ 
\end{tabular}
\label{table:all opt}
    \end{center}
\end{table}

    \begin{example}[Graphical conditions in Theorem \ref{theorem:optimalCIS}]
	Consider the acyclic directed mixed graphs in Fig. \ref{subfig:Zo empty} and \ref{subfig:Zo not optimal}, denoted here as $\g$ and $\g'$, respectively. Again, we are interested in estimating $\tau_{yx}$. In $\g$, $Z^o=\emptyset$, even though $\{A,\emptyset\}$ is a linearly valid conditional instrumental set relative to $(X,Y)$ in $\g$. It is therefore an example, where $(Z^o,W^o)$ is not a linearly valid tuple, even though one exists. In $\g'$, $Z^o=A$ and $W^o=\{B,C\}$, which is a linearly valid tuple relative to $(X,Y)$ in $\g'$. However, $Z^o \cap \{\pa(X,\g') \cup \mathrm{sib}(X,\g')\}=\emptyset$. As a result in many linear structural equation model compatible with $\g'$ the instrumental strength of $(A,\{B,C\}), \sigma_{xx.bc} - \sigma_{xx.abc}$, will be small.
    \end{example}

\section{Supplementary material for simulations}
\label{sec:simulations appendix}

We now provide some additional information regarding the simulations in our main paper. Reproducible code is made available at \texttt{https://github.com/henckell/efficientCIS}.

For completeness we provide violin plots of the ratios for the two-stage least squares estimator with all alternative linearly valid conditional instrumental sets in our simulation study from Section \ref{sec:simulations}. In addition we also also add violin plots for the ratio of the theoretical asymptotic standard deviation of $(Z^o,W^o)$ to the alternative linearly valid tuple per Theorem \ref{theorem:varequ}. These plots represent the case for $n=\infty$. The plots for $\g_1$ are given in Fig. \ref{fig:violin_long_1a}. The plots for $\g_2$ are given in Fig. \ref{fig:violin_long_2a_a} and \ref{fig:violin_long_2a_b}. 

In addition we also provide violin plots in Fig. \ref{fig:liml jive main} of the ratios of the limited information maximum likelihood and the jackknife instrumental variables estimator with $(Z^o,W^o)$ to the one with a representative subset of the alternative linearly valid conditional instrumental sets. 

 \begin{figure}	
  \captionsetup{width=5cm}
  		\centering
 		\includegraphics[height=18cm, width=14cm]{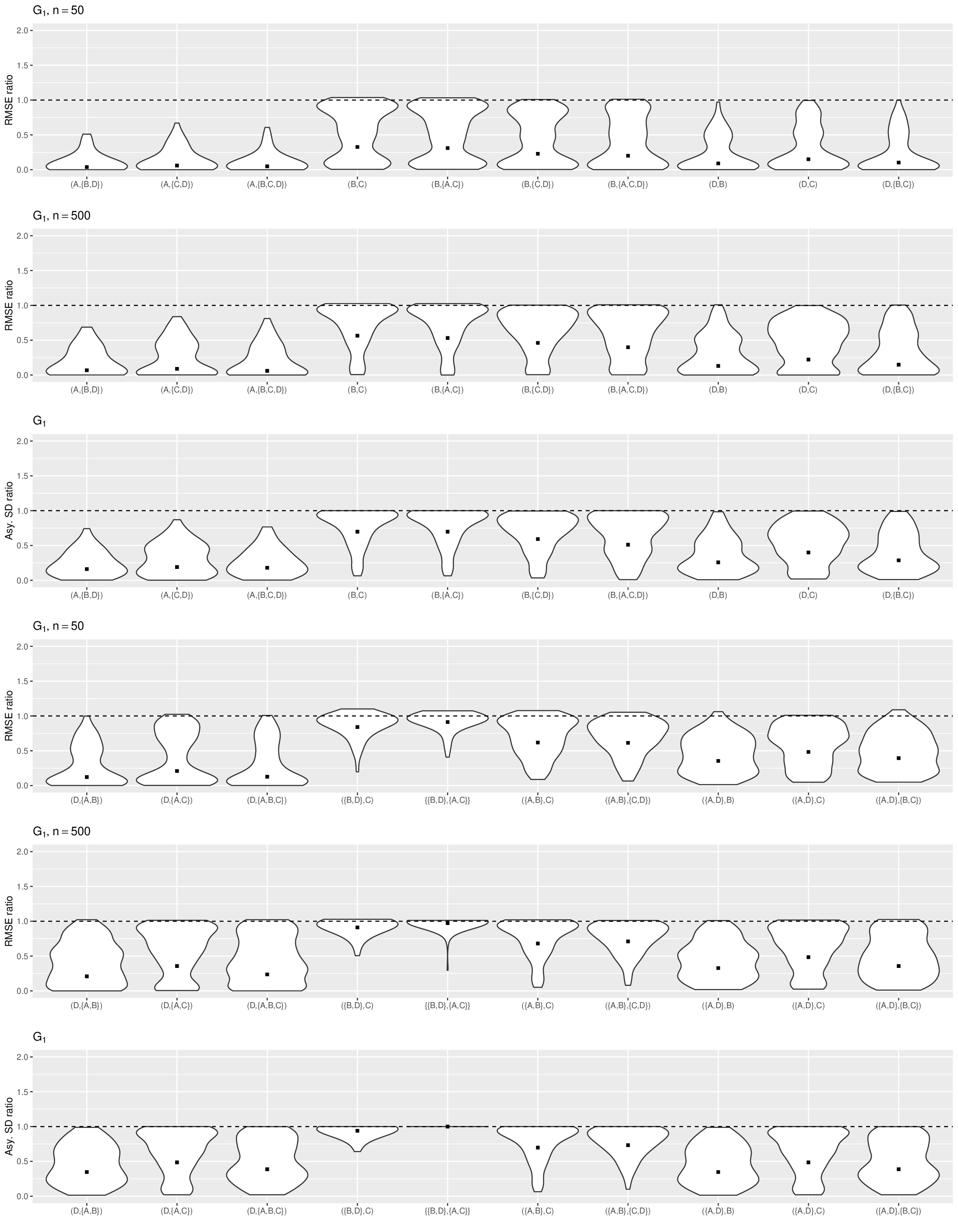}
 	\caption{Violin plot of the ratios of the root mean squared error, respectively the asymptotic standard deviation, for the two-stage least squares with $(Z^o,W^o)$ to the one for the tuple $(Z,W)$ given on the X-axis. The black dots mark the geometric mean of the ratios.}	
 	\label{fig:violin_long_1a}
 \end{figure}
 
  \begin{figure}[p]	
  \captionsetup{width=5cm}
  		\centering
 		\includegraphics[height=18cm, width=14cm]{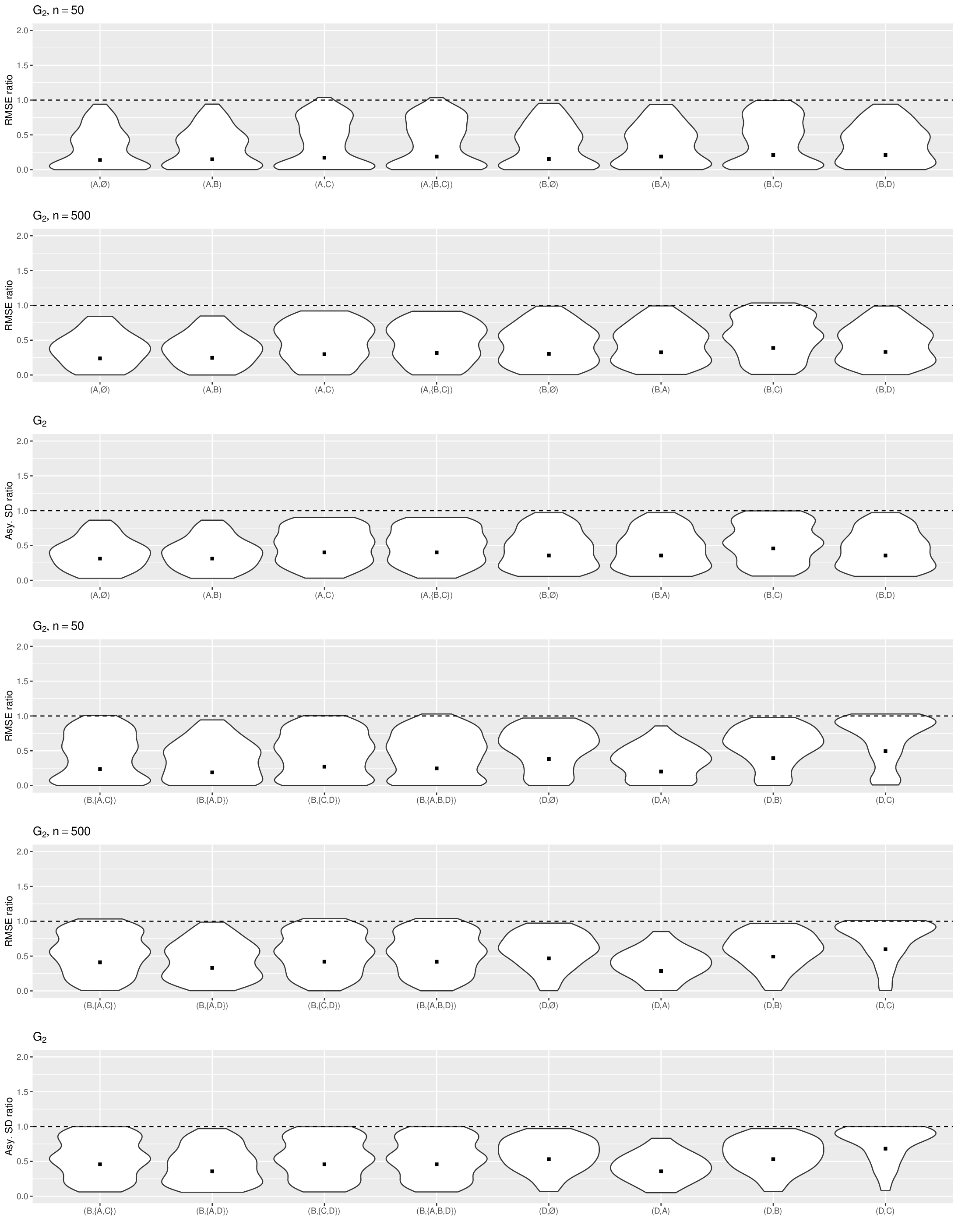}
 	\caption{Violin plot of the ratios of the root mean squared error, respectively the asymptotic standard deviation, for the two-stage least squares with $(Z^o,W^o)$ to the one for the tuple $(Z,W)$ given on the X-axis. The black dots mark the geometric mean of the ratios.}	
 	\label{fig:violin_long_2a_a}
 \end{figure}
 
  \begin{figure}[p]	
  \captionsetup{width=5cm}
  		\centering
 		\includegraphics[height=18cm, width=14cm]{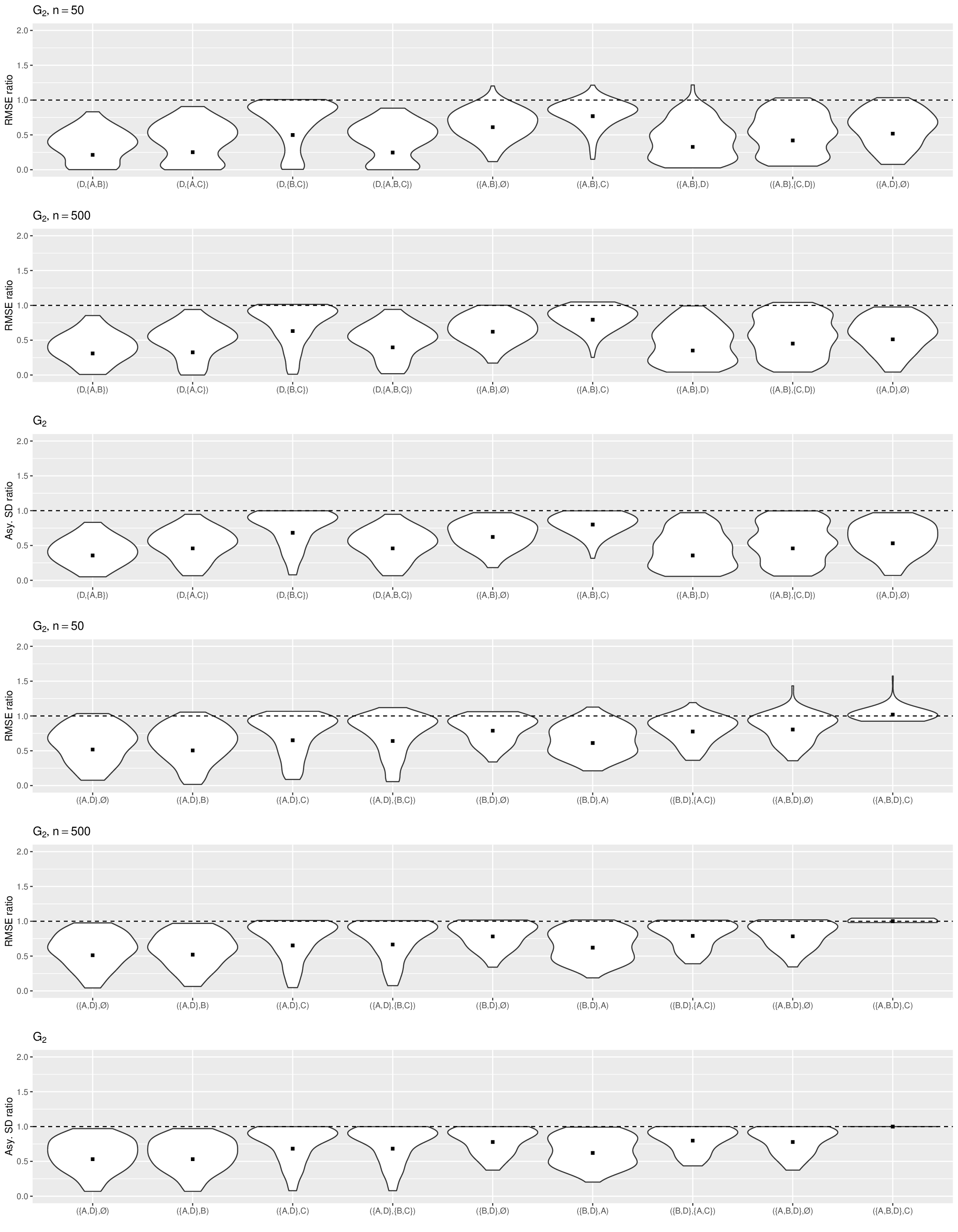}
 	\caption{Violin plot of the ratios of the root mean squared error, respectively the asymptotic standard deviation, for the two-stage least squares with $(Z^o,W^o)$ to the one for the tuple $(Z,W)$ given on the X-axis. The black dots mark the geometric mean of the ratios.}	
 	\label{fig:violin_long_2a_b}
 \end{figure}

  \begin{figure}[p]	
  \captionsetup{width=5cm}
  		\centering
 		\includegraphics[height=18cm, width=14cm]{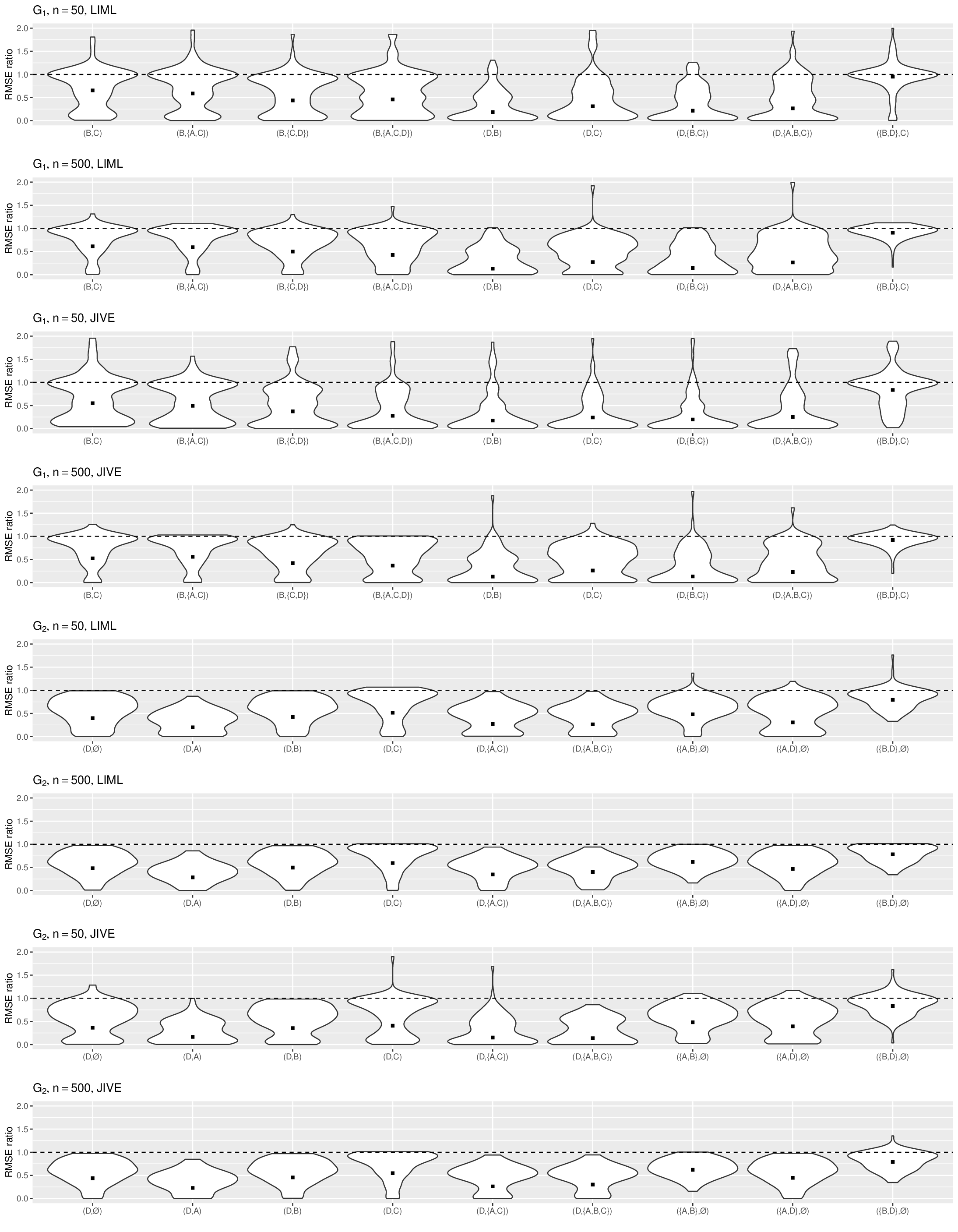}
 	\caption{Violin plot of the ratios of the root mean squared error, respectively the asymptotic standard deviation, with $(Z^o,W^o)$ to the one for the tuple $(Z,W)$ given on the X-axis. The black dots mark the geometric mean of the ratios.}	
 	\label{fig:liml jive main}
 \end{figure}

\end{document}